\documentclass[12pt]{amsart}
\usepackage{amssymb,amsmath,amsthm,mathrsfs,ulem}

\newtheorem{thm}{Theorem}[subsection]
\newtheorem{theorem}[thm]{Theorem}
\newtheorem{lemma}[thm]{Lemma}
\newtheorem{corollary}[thm]{Corollary}
\newtheorem{proposition}[thm]{Proposition}
\numberwithin{equation}{section} \theoremstyle{definition}

\newtheorem*{remark}{Remark}
\newtheorem*{remarks}{Remarks}
\usepackage{graphicx,color}
\usepackage{amsmath}
\usepackage{latexsym}
\usepackage{epsfig}
\usepackage{amsfonts}
\usepackage{amssymb}
\usepackage[margin=2cm]{geometry}

\newcommand{\diag}{{\rm diag}}
\newcommand{\Tr}{{\rm Tr}}
\newcommand{\B}{\mathcal{B}}

\newcommand{\U}{\mathcal{U}}
\newcommand{\A}{\mathcal{A}}
\newcommand{\I}{\mathcal{I}}
\newcommand{\oo}{\mathcal{O}_F}
\newcommand{\m}{^{\text{-1}}}
\newcommand{\ord}{{\rm ord}}

\newcommand{\PGL}{{\rm PGL}}
\newcommand{\GL}{{\rm GL}}

\newcommand{\SL}{{\rm SL}}
\newcommand{\Sp}{{\rm Sp}}

\newcommand{\Z}{\mathbb Z}

\newcommand{\G}{\Gamma}
\newcommand{\g}{\gamma}
\newcommand{\ka}{\kappa}
\newcommand{\ve}{\varepsilon}
\newcommand{\de}{\delta}
\newcommand{\vol}{{\rm vol}}

\begin{document}
\title[Zeta Functions of Complexes Arising from $\PGL(3)$]{Zeta Functions of Complexes Arising from $\PGL(3)$}
\author{Ming-Hsuan Kang and Wen-Ching Winnie Li}
\address{Ming-Hsuan Kang\\ Department of Applied Mathematics\\ National Chiao-Tung University\\
Hsinchu, Taiwan} \email{\tt mhkang@math.nctu.edu.tw}
\address{Wen-Ching Winnie Li\\ Department of Mathematics\\ The Pennsylvania State University\\
University Park, PA 16802 U.S.A. ~and National Center for Theoretical Sciences, Mathematics Division,
National Tsing Hua University, Hsinchu 30013, Taiwan, R.O.C.} \email{\tt wli@math.psu.edu}

\thanks{This research was supported in part by the DARPA
grant HR0011-06-1-0012 and the NSF grants DMS-0457574 and DMS-0801096 (both authors),   NSF grant DMS-1101368 (W-C. W. Li), and also the NSC grant 100-2115-M-009-008-MY2 (M. Kang). Part of the research was performed
while both authors were visiting the National Center for Theoretical
Sciences, Mathematics Division, in Hsinchu, Taiwan. They would like
to thank the Center for its support and hospitality.}


\begin{abstract} In this paper we obtain a closed form expression of
the zeta function $Z(X_\G, u)$ of a finite quotient {$X_\G$ of the Bruhat-Tits building of
$\PGL_3$ over a nonarchimedean local field $F$ by a discrete cocompact torsion-free subgroup $\Gamma$ of $\PGL_3$}. Analogous to a graph
zeta function, $Z(X_\G, u)$ is a rational function with two different expressions and it satisfies
the Riemann hypothesis if and only if $X_\G$ is a Ramanujan complex.
\end{abstract}
\maketitle

\section{Introduction}
\subsection{} First introduced by Ihara \cite{Ih} for groups and later
reformulated by Serre for regular graphs,
the zeta function of a finite,
connected, undirected graph $X$ is defined as
$$ Z(X,u) = \prod_{[C]} (1 - u^{l([C])})^{-1},$$
where the product is over equivalence classes $[C]$ of geodesic
tailless primitive cycles $C$, and $l([C])$ is the length of a cycle
in $[C]$. In this paper we adopt the convention that a cycle is an oriented closed path with a starting point and possible repetition of vertices. Two cycles are equivalent if one is obtained from the other by changing the starting vertex. A geodesic path on a graph means no backtracking.
A cycle is tailless if all cycles equivalent to it are geodesic; it is primitive if it is not a repetition of a shorter cycle more than once. Taking the logarithmic derivative of $Z(X, u)$, one gets
$$Z(X,u) =  \exp \bigg(\sum_{n \ge 1} \frac{N_n(X)}{n} u^n \bigg),$$ where
$N_n(X)$ counts the number of geodesic tailless cycles in $X$
of length $n$.

Not only defined analogous to the zeta function of a curve over a finite field, the zeta
function of a graph is also a rational function. This can be seen in two ways.
The first is the result of Ihara:

\begin{theorem}[Ihara \cite{Ih}]\label{Ihara}
Let $X$ be a
$(q+1)$-regular graph. Then its zeta function
 is a rational function of the form
$$ Z(X,u) = \frac{(1 - u^2)^{\chi(X)}}{\det(I - Au + qu^2I)},$$
where $\chi(X)$ is the Euler characteristic of $X$
and $A$ is the adjacency matrix of $X$.
\end{theorem}

This theorem is extended to irregular graphs in \cite{Ba}, \cite{Ha2}, \cite{ST}, and \cite{Ho}.
The reader is referred to
\cite{ST} and the references therein for the history and various
zeta functions attached to a graph.

Endow two opposite orientations on each edge of $X$. Define the
{out-neighbor} of the directed edge $u \to v$ to be the edges $v \to w$ with $w \ne u$.
The (directed) edge adjacency matrix $A_e$ has its rows and columns indexed by
the directed edges $e$ of $X$ such that the $e e'$ entry records the number of times $e'$ is an out-neighbor
of $e$.
Hashimoto
\cite{Ha} observed that $N_n(X) = \Tr A_e^n$ so that
\begin{eqnarray}\label{hashimoto}
Z(X, u) = \frac {1}{\det(I - A_e u)}.
\end{eqnarray}
This gives the second {proof} of the rationality of the graph zeta
function.

 A $(q+1)$-regular graph $X$ is called {\it Ramanujan} if all
eigenvalues $\lambda$ of its adjacency matrix $A$ other than $\pm
(q+1)$ satisfy $|\lambda| \le 2 \sqrt q$ (cf. \cite{LPS}). The
Ramanujan graphs are optimal expanders with extremal spectral
property. It is easily checked that $X$ is Ramanujan if and only if
its zeta function $Z(X,u)$ satisfies the Riemann hypothesis, that
is, the poles of $Z(X,u)$ other than $\pm 1$ and $\pm q^{-1}$
(called nontrivial poles) have the same absolute value $q^{-1/2}$ (cf.
\cite{ST}).

\subsection{} When $q$ is a prime power, the universal cover of a $(q+1)$-regular graph can be identified with the $(q+1)$-regular tree
{associated to $\PGL_2(F)$} for
 a nonarchimedean local field $F$ with $q$ elements in its
residue field. Denote by $\oo$ its ring of integers  and let $\pi$ be a uniformizer of $F$.
The vertices of the tree can be parametrized by the right cosets of the standard maximal compact subgroup $\PGL_2(\oo)$ and the directed
edges by the right cosets of the Iwahori subgroup $\I$ of $\PGL_2(\oo)$.
Moreover, the (vertex) adjacency operator $A$ on the tree is the
Hecke operator given by the double coset $\PGL_2(\oo)\diag(1, \pi)
\PGL_2(\oo)$ and the edge adjacency operator $A_e$ is the
Iwahori-Hecke operator given by the double coset $\I \diag(1, \pi)
\I$. One obtains a $(q+1)$-regular graph $X_{\tilde \G}$ by taking a left quotient
by a torsion-free discrete cocompact subgroup $\tilde \G$ of $\PGL_2(F)$.

This set-up has a higher dimensional extension to the Bruhat-Tits
building $\B_n$ associated to {$\PGL_n(F)$}, which is a
{contractable} $(n-1)$-dimensional simplicial complex. Like graphs, one
obtains finite complexes $X_\G$ by taking quotients of $\B_n$ by torsion-free discrete
cocompact subgroups $\G$ of $\PGL_n(F)$. The concept of Ramanujan complexes
 was introduced in \cite{Li},  called Ramanujan hypergraphs there. Three explicit constructions
of infinite families of Ramanujan complexes were given in
\cite{Li},  \cite{LSV1} and
\cite{Sa}, respectively, using deep results on the Ramanujan
conjecture over function fields for automorphic representations of
the multiplicative group of a division algebra by
Laumon-Rapoport-Stuhler \cite{LRS} and of $\GL_n$ by Lafforgue
\cite{La}. Further, the paper \cite{LSV2} discusses what kind of
$\G$ would fail to yield a Ramanujan complex.

To extend zeta functions from graphs to complexes, one seeks a
similarly defined zeta function counting closed geodesic tailless cycles in
 $X_\G$ with the following properties:
\begin{itemize}
\item[(a)] it is a rational function with a closed form expression;

\item[(b)] it captures both topological and spectral information of
$X_\G$; and

\item[(c)] it satisfies the Riemann hypothesis if and only if
 $X_\G$ is a Ramanujan complex.
\end{itemize}
The purpose of this paper is to present  zeta functions with such properties for $2$-dimensional complexes which are finite quotients of $\B_3$. This was previously considered in \cite{DH} by Deitmar and Hoffman. The zeta functions there were defined differently, and they were not shown to possess the properties (a)-(c). Recently, Fang, Li and Wang in \cite{FLW} obtained zeta functions for $2$-dimensional complexes arising from finite quotients of the building associated to $\Sp_4(F)$.

\subsection{}  In what follows, we fix a
local field $F$ with $q$ elements in its residue field as before.
Let $\B$ denote the Bruhat-Tits building $\B_3$ associated to $\PGL_3(F)$, which is a 2-dimensional contractable simplicial complex.
 Write $G$ for the group $\GL_3(F)$, $Z$ its center, and $K$ its standard maximal compact subgroup $\GL_3(\oo)$. Denote by $E$ and $B$ the standard parahoric and Iwahoric subgroups of $K$, respectively. Similar to the case of $\PGL_2(F)$, the vertices, type $1$ edges, and pointed chambers of the building $\B$ can be parametrized by the right $KZ-$, $EZ-$, and $BZ-$ cosets of $G$, respectively. The Hecke operators $A_1$ and $A_2$ associated to the double cosets $K \diag(1, 1, \pi)KZ$ and $K \diag(1, \pi, \pi)KZ$ describe the type $1$ and type $2$ out-neighbors of a vertex, the operator $L_E$ associated to the double coset $E \diag(1, 1, \pi)EZ$ describes the type 1 out-neighbors of a type 1 edge, and the out-neighbors of a pointed chamber are given by the operator $L_B$ associated to the double coset $B  \left( \begin{smallmatrix} 1 &  &  \\
& & 1 \\  & \pi & \end{smallmatrix}\right) BZ$. Details are given in \S3.

All $1$-dimensional paths in $\B$ considered in this paper are contained in the $1$-skeleton of $\B$.
A $1$-geodesic between two vertices in $\B$ is a shortest path in the $1$-skeleton of $\B$. As $\B$ is the union of apartments and each apartment is an Euclidean plane, there is a metric on $\B$ so that a geodesic in $\B$ is a straight line contained in an apartment. Thus a $1$-geodesic in $\B$ is a geodesic if and only if it consists of edges of the same type. Let $\G$ be a discrete torsion-free cocompact-mod-center subgroup of $G$ satisfying $\ord_\pi \det \G \subset 3\Z$. Denote by $X_\G$ the (finite) quotient $\G \backslash \B$.
A $1$-geodesic in $X_\Gamma$ is called a geodesic if one and hence all of its liftings in $\B$ are geodesics.

\subsection{} For $i=1$ or 2, the type $i$ edge zeta function of $X_\G$ is defined as
$$ Z_{1,i}(X_\G,u) = \prod_{[C]} (1 - u^{l_A([C])})^{-1},$$
\noindent where $[C]$ runs through the equivalence classes of
 tailless primitive closed geodesics $C$ in $X_\G$ consisting of edges of type $i$,
 and $l_A([C])$ is the algebraic
length of any geodesic in $[C]$ defined in \S5.3. The first main result below follows from Proposition \ref{Z1andLE} and Theorem \ref{logedgezeta}. It extends Hashimoto's identity (\ref{hashimoto}) to type $i$ edge zeta functions, and gives an explicit formula in terms of conjugacy classes of $\G$ for the number of tailless closed geodesics in $X_\G$ of a given length. 

 {\bf Theorem A.} {\it The edge zeta functions are rational functions in $u$ with the following expressions:
$$  Z_{1,i}(X_\G,u) = \frac{1}{\det(1 - L_Eu^{i})} = \exp(\sum_{n \ge 1} \frac{N_n(X_\G)}{n} ~u^{in}), \qquad i = 1, 2, $$
\noindent where
$N_n(X_\G)$ counts the number of closed tailless geodesics of algebraic length $n$ using only type $1$ edges in $X_\G$; it is given by
\begin{eqnarray}
N_n(X_\G) = \sum_{\g \in [\G],~ [\g] ~\rm{of~type} ~(n, 0)} \rm{vol}([\g]) \omega_{[\g]}.
\end{eqnarray}
Here  $[\G]$ is a set of representatives of conjugacy classes of $\G$, $[\g]$  is a set of closed geodesics defined by (\ref{classofgamma}), $\rm{vol}([\g])$ is given in  (\ref{fundamentaldomain'}), and $\omega_{[\g]}$ is as in Theorem \ref{numberinaclass} and Proposition  \ref{rankonenumberofalgtailless}. }
\smallskip

In addition to paths formed by directed edges, we also consider paths formed by edge-adjacent  chambers, called galleries.  %A gallery between two chambers in $\B$ is called a geodesic

The type $1$ chamber zeta function of $X_\G$ is defined similar to the type $1$ edge zeta function:
$$ Z_{2,1}(X_\G,u) = \prod_{[C]} (1 - u^{l([C])})^{-1},$$
\noindent where $[C]$ runs through the equivalence classes of
  primitive closed tailless galleries $C$ in $X_\G$ of type $1$,
 and $l([C])$ is the length of any gallery in $[C]$.  (See \S \ref{galleries} for definitions.) Our second main result is a detailed description of $Z_{2,1}(X_\G, u)$, obtained from Proposition \ref{rationalZ2} and Corollary \ref{numberoftaillessgalleries}.

{\bf Theorem B.} {\it The type $1$ chamber zeta function is a rational function with following expressions:
\begin{eqnarray}
Z_{2,1}(X_\G, u) = \frac{1}{\det (I - L_B u)} = \exp(\sum_{n \ge 1} \frac{M_n(X_\G)}{n} u^n),
\end{eqnarray}
where the number $M_n(X_\G)$ of closed tailless galleries in $X_\G$ of type $1$ and length $n$ is given below:

(1) If $n=2m+1$ is odd, then
\begin{eqnarray*}
M_n(X_\G) = \sum_{\substack{\g \in [\G] ~{\rm ramified ~rank-one ~split,}\\  ~[\g] {\rm~of ~type} ~(1,m)}}
\vol([\g]);
\end{eqnarray*}

(2) If $n = 2m$ is even, then
\begin{eqnarray*}
M_n(X_\G) &=& \sum_{\substack{\g \in [\G] ~{\rm split,}\\  ~[\g] {\rm~of ~type} ~(0,m)}} \vol([\g])\omega_{[\g]} +
\sum_{\substack{ \g \in [\G] ~{\rm irregular},\\ ~[\g] {\rm ~of ~type} ~(0, m)}} \vol([\g])q\\
&+& \sum_{ \substack{\g \in [\G] ~{\rm unramified ~rank-one
~split ,}\\ ~[\g] {\rm~of ~type} ~(0, m)}} \vol([\g])(\omega_{[\g]} - 2)
+ \sum_{\substack{\g \in [\G] ~{\rm ramified ~rank-one ~split,}\\  ~[\g] {\rm~of ~type} ~(0,m)}}
\vol([\g])(\omega_{[\g]} - 1).
\end{eqnarray*}
Here $[\G]$, $[\g]$, $\rm {vol}([\g])$ and $\omega_{[\g]}$ are as in Theorem A.}

The elements of $\G$ are classified in \S \ref{classificationofGamma} according to their eigenvalues. It is interesting to compare the above two theorems with the zeta function of the finite regular graph $X_{\tilde \G}$ in \S1.2.
 Thus the zeta function of $X_{\tilde \G}$ can be rewritten as
 \begin{eqnarray}\label{productbygroupelts}
Z(X_{\tilde \G}, u) = \prod_{[\tilde \g]} \frac{1}{1 - u^{l(\tilde \g )}}  = \exp (\sum_{n \ge 1} \frac{N_n(X_{\tilde \G})}{n} u^n),
\end{eqnarray} where $[\tilde \g]$ runs through conjugacy classes of primitive elements in $\tilde {\G}$. %, and $N_n(X_{\tilde \G})$ counts the number of tailless geodesic cycles in $X_{\tilde \G}$ with length $n$, as before.
The number $N_n(X_{\tilde \G})$ of tailless geodesic cycles in $X_{\tilde \G}$ with length $n$ is equal to
$$ N_n(X_{\tilde \G}) = \sum_{[\tilde \g] ~ {\rm primitive,}~l([\tilde \g])|n} l(\tilde \g).$$

All nontrivial elements in $\tilde {\G}$ are hyperbolic, analogous to the "split" elements in $\G$. One has $l(\tilde \g) = \rm{vol}([\tilde \g]) = \rm{vol}([\tilde \g^m])= \max(\ord_\pi a/b, \ord_\pi b/a)$, where $a$ and $b$ are eigenvalues of $\tilde \g$, and $\omega_{[\tilde \g]} = \omega_{[\tilde \g^m]} = 1$ for all $m \ne 0$. Therefore the formulas for $N_n(X_\G)$ and $M_n(X_\G)$ generalize that for $N_n(X_{\tilde \G})$. On the other hand, since both $\rm{vol}([\g^m])$ and $\omega_{[\g^m]}$ vary with the exponent $m$ in a complicated way, there are no simple expressions for  the edge and chamber zeta functions of $X_\G$ as Euler products over conjugacy classes in $\G$, similar to  (\ref{productbygroupelts}) for graphs.

The zeta function of $X_\G$ is defined as
$$ Z(X_\G,u) = Z_{1,1}(X_\G, u)Z_{1,2}(X_\G, u).$$
\noindent The explicit expressions of the edge and chamber zeta functions above lead to a new  expression for the zeta function $Z(X_\G, u)$, which can be viewed as a $2$-dimensional analogue of Theorem \ref{Ihara}.
\smallskip

{\bf Theorem C.} {\it
 The zeta function of the finite
complex $X_\G = \Gamma \backslash \B$ can be expressed as
\begin{eqnarray}\label{zeta}
Z(X_\G, u) %&=& \frac{1}{\det(1 - L_E u) \det(1- (L_E)^t u^2)} \\
=  \frac{(1-u^3)^{\chi(X_\G)}}{\det(I-A_1u+qA_2u^2-q^3
u^3I)\det(I + L_Bu)},
\end{eqnarray}
in which $\chi(X_\G)$ is the Euler characteristic of $X_\G$, $A_1$ and $A_2$ are operators on vertices, %$L_E$ (resp. $(L_E)^t$) is the operator on type $1$ (resp. $2$) edges,
and $L_B$ is the operator on pointed chambers in $X_\G$ introduced above.}

Combining Theorems A and B, and noting that the transpose $(L_E)^t$ of $L_E$ is the edge adjacency operator of type 2 edges in $X_\G$, we rephrase the identity (\ref{zeta}) in terms of the
operators on $X_\G$ as
\begin{eqnarray}\label{zetaidentity}
\frac{(1-u^3)^{\chi(X_\G)}}{\det(I-A_1u+qA_2u^2-q^3 u^3I)} =
\frac{\det(I + L_Bu)}{\det(I - L_E u) \det(I - (L_E)^t u^2)}.
\end{eqnarray}
Compared to the parallel identity of operators on a $(q+1)$-regular graph
$X$:
$$\frac{(1 - u^2)^{\chi(X)}}{\det(I - Au + qu^2I)} = \frac{1}{\det(I
- A_eu)},$$ the similarity is reminiscent of the zeta functions
attached to a surface and a curve over a finite field. It is likely that the identity
(\ref{zetaidentity}) expressed in terms of the operators on the
finite complex is a prototype of complex zeta
functions in general. Indeed, the identity on zeta functions in \cite{FLW} for the
$GSp_4(F)$ case is formulated after this. Theorem C was proved in
\cite{KLW} from representation-theoretical viewpoint
by comparing the eigenvalues of the operators in (\ref{zetaidentity}), while
the proof  in this paper explores the combinatorial and group-theoretic viewpoints of the identity.

\subsection{} Our $Z(X_\G, u)$ clearly has properties (a) and (b). Now we discuss its
connection with the Riemann hypothesis. The trivial zeros of
$\det(I-A_1u+qA_2u^2-q^3 u^3I)$ arise from the trivial eigenvalues
of $A_1$ and $A_2$ on $X_\G$; they are the roots of $(1-u^3)(1-q^3u^3)(1- q^6u^3)$. We say that $Z(X_\G, u)$ satisfies the Riemann hypothesis if the nontrivial zeros of
$\det(I-A_1u+qA_2u^2-q^3 u^3I)$ have the same absolute value $q^{-1}$, which is
equivalent to $X_\G$ being Ramanujan (cf.\cite{Li}).

The zeros of each determinant in (\ref{zetaidentity}) are computed in \cite{KLW}; they give rise to a description of the Ramanujan condition in terms of the operators on each dimension.

\begin{theorem}[\cite{KLW}, Theorem 2] The following four statements on $X_\G$ are equivalent.
\begin{enumerate}
\item[(1)] $X_\G$ is a Ramanujan complex;

\item[(2)] The nontrivial zeros of $\det(I-A_1u+qA_2u^2-q^3 u^3I)$ have absolute value $q^{-1}$;

\item[(3)] The nontrivial zeros of $\det(I - L_E u)$ have absolute values $q^{-1}$ and $q^{-1/2}$; and

\item[(4)] The nontrivial zeros of $\det(I + L_Bu)$ have absolute values $1$, $q^{-1/2}$ and $q^{-1/4}$.
\end{enumerate}
\end{theorem}
\noindent Thus the Riemann hypothesis for $Z(X_\G, u)$ is actually a statement concerning the nontrivial zeros of each determinant in (\ref{zetaidentity}),
analogous to the Riemann hypothesis for a surface zeta function.

When $F$ is the completion of a function field $M$ at a place $v$, as constructed in \cite{Li}, there are infinitely many $\Gamma$ arising from a suitable central division algebra of dimension $9$ over $M$ unramified at $v$ such that the polynomial $\det(I-A_1u+qA_2u^2-q^3 u^3I)/(1-u^3)(1-q^3u^3)(1-q^6u^3)$ from $X_\G$ agrees with the portion of the zeta function coming from the second $\ell$-adic cohomology of a moduli surface studied in \cite{LRS}. Computations in \cite{KLW} imply that
$$(1-u^3)(1 - q^6u^3)\det(I - L_E u)/\det(I-A_1u+qA_2u^2-q^3 u^3I)$$ is a polynomial whose zeros
have the same absolute value $q^{-1/2}$. It would be interesting to know whether there is any geometric interpretation for suitable choices of $\Gamma$.

\subsection{} 
{We sketch the main ingredients of the proofs of Theorems A, B and C. Each $\g \in \G$ has an associated rational form $r_\g$, constructed from the eigenvalues of $\g$. The base point free homotopy classes of closed $1$-geodesics in
$X_\G$ are partitioned into sets indexed by the conjugacy classes
$[\g]$ of $\G$. Theorem \ref{charlAandlG} asserts that the $1$-geodesics in $[\g]$ achieving minimal algebraic (resp. geometric) length, called algebraically (resp. geometrically) minimal, have the same algebraic (resp. geometric) length as $r_\g$.
Further, a set $[\g]$ contains tailless geodesic cycles of type $1$ or $2$ if and
 only if $r_\g$
%Our aim is to count the tailless geodesic cycles of type $1$ or $2$ in $[\g]$, which are present only when $[\g]$
has type $1$ or $2$. In this case, by Proposition \ref{algminequalgeommin}, there is no distinction among tailless geodesic, geometrically minimal, and algebraically minimal cycles. Algebraically minimal cycles in $[\g]$ afford an explicit algebraic
characterization, as shown in \S7 for $\g$ split or irregular and in \S8 for $\g$ rank-one split, and hence are more amenable to computation. In \S \ref{number-split} and \S \ref{number-rankone} we enumerate the number of cycles in $[\g]$ with given algebraic length, along with those of type $1$. These numbers establish Theorem A, and they are also used in the proof of Theorem C.}

The chamber zeta function defined above counts closed tailless galleries in $X_\G$ of type $1$.  In \S9 we first convert this to counting closed pointed galleries in $X_\G$ with respect to the adjacency defined by the operator $L_B$ (Proposition \ref{gallerytaillesscriterion}). Then we characterize the closed pointed galleries in \S\ref{charclosedpointedgalleries}. Using this criterion we prove Theorem B by comparing the logarithmic derivatives of the chamber zeta function and the type $2$ edge zeta function.

The proof of Theorem C given in \S \ref{proofofC} results from comparing the logarithmic
derivatives of both sides of (\ref{zetaidentity}). More precisely, that of the left hand side  counts the number of type $1$ tailless closed geodesics in $X_\G$, as given by the logarithmic derivative of $1/\det(I - L_E u)$, and some extra terms arising from sets
represented by irregular and rank-one split $\g$'s. These extra terms are shown to equal to the logarithmic derivative of $\det(I + L_Bu)/\det(I - L_E^tu^2)$
by comparing $N_n(X_\G)$ in Theorem A and $M_n(X_\G)$ in Theorem B. It should be pointed out that while the edge zeta functions count only tailless cycles of types $1$ and $2$, to prove the identity, we actually consider all cycles, with and without tails. This is similar in spirit to the proof of Theorem \ref{Ihara}  given in \cite{Ih}.

\section{Hecke operators on $\PGL_3(F)$}

\subsection{Hecke operators} By the elementary divisor theorem,
the group $G$ is equal to the disjoint union of the {$KZ$}-double cosets $$
T_{n,m}=K~\diag(1, \pi^m, \pi^{m+n})KZ$$ as $m, n$ run through all
non-negative integers. We shall also regard each $T_{n,m}$ as the Hecke
operator acting on functions $f \in L^2(G/KZ)$ via
$$T_{n,m}f(gKZ) = \sum_{\alpha KZ \in T_{n,m}/KZ} f (g\alpha KZ).$$
In particular, set
$$ A_1 = T_{1,0} \qquad {\rm and} \qquad A_2 = T_{0,1}.$$

\subsection{Recursive relations} It is well-known that each Hecke operator is a polynomial in $A_1$ and
$A_2$. Tamagawa \cite{Ta} obtained a recursive relation on Hecke operators for $\GL_n(F)$.
\noindent We prove a different recursive formula adapted for our needs.

\begin{theorem}
\begin{eqnarray}\label{recursivehecke}
~\qquad q \sum^{\infty}_{k=1} T_{k,0} u^k -
(q-1)(\sum^{\infty}_{k=1} \sum_{n+2m = k} T_{n,m}
u^k)\frac{1-q^2u^3}{1-u^3}= u \frac{d}{du}
\log\frac{(1-u^3)^{r}I}{I-A_1 u + A_2 q u^2 -q^3u^3 I},
\end{eqnarray} where $ r=\frac{(q+1)(q-1)^2}{3}$.
\end{theorem}

\begin{proof}
The Hecke algebra for $G/Z$ is isomorphic to
the polynomial ring $\mathbb{C}[z_1,z_2,z_3]^{S_3} /
\langle z_1z_2z_3 - 1 \rangle$, denoted by $H$, under the Satake isomorphism $\psi$ (cf. \cite{Sat}).
Our strategy is to show that the identity holds after applying the
Satake isomorphism. For this, we need to compute the values of $\psi$ on $\{T_{n,m}\}$.
Using $z_1, z_2, z_3 \in H$ we define a
quasi-character $\chi$ on the Borel subgroup $P$ of $G$ by
$$ \chi \left( \left[\begin{matrix} b_1 & * & * \\ & b_2 & * \\ & & b_3 \end{matrix}\right]\right)
= z_1^{\ord_\pi (b_1)}z_2^{\ord_\pi (b_2)}z_3^{\ord_\pi (b_3)},  $$
and regard it as a map from {$P/Z$} to $\mathbb{C}[z_1,z_2,z_3] /
\langle z_1z_2z_3 - 1 \rangle$. %(The relation $z_1z_2z_3 = 1$ follows from the fact
%that $\chi$ is trivial on $Z$.)
Denote by $\delta_P$ the modular
character on $P/Z$. Let $\phi$ be the function on $G/Z$ given by
$$\phi(bk)=\chi(b)\delta_P^{1/2}(b) \qquad (b\in P, k \in K).$$
\noindent Then the value of the Satake isomorphism at $T_{n,m}$ is
%$$\psi(T_{n,m}) = \int_{T_{n,m} /Z}\phi(g) dg \qquad (n \ge 0, ~m \ge 0),$$
% where $dg$ is induced from the
%Haar measure on $G$ so that $KZ/Z$ has volume $1$.  More precisely, if we write $T_{n,m}=\bigsqcup_{g \in I_{n,m}} g KZ$, then
$$\psi(T_{n,m}) = \sum_{g \in I_{n,m}} \phi(g),$$
where $T_{n,m}=\bigsqcup_{g \in I_{n,m}} g KZ$.

Direct computations give
$\psi(A_1)=q(z_1 + z_2 + z_3)$ and $\psi(A_2)=q(z_1z_2 + z_2z_3 + z_3z_1)$ so that
$$ \psi(I-A_1 u + q A_2 u^2 - q^3 u^3 I) = (1-qz_1u)(1-qz_2u)(1-qz_3u),$$
which allows us to get the value of the right hand side of the identity under $\psi$.

 For $k \ge 1$, let $T_k = \sum_{n+2m=k} T_{n,m}$, and set
$$\sigma_{k,1}(z_1,z_2,z_3)= z_1^k + z_2^k + z_3^k,  \quad \sigma_{k,2}(z_1,z_2,z_3)= \sum_{1\le a \le k-1} z_1^a z_2^{k-a} + z_2^a z_3^{k-a} + z_3^a z_1^{k-a},$$ and $$\sigma_{k,3}(z_1,z_2,z_3)= \sum_{a, b, c \ge 1, a+b+c=k}z_1^a z_2^b z_3^c. $$
%Our strategy is to show that the identity (\ref{recursivehecke}) holds after applying the
%Satake isomorphism $\psi$.
For the left hand side of the identity, we compute the coefficient of
$z_1^{a_1}z_2^{a_2}z_3^{a_3}$ in $\psi(T_k)$ with $a_1 \ge a_2 \ge
a_3 \ge 0$ and $a_1+ a_2 + a_3 = k$, then use symmetry to determine $\psi(T_k)$.

It is straightforward to check that the number of  cosets
$gKZ $ in $\bigsqcup_{n+2m=k}  T_{n,m} $ mapped to $z_1^{a_1}z_2^{a_2}z_3^{a_3}$ by $\chi$ is equal to
$q^{2a_1+a_2}$ if $a_3=0$, and $(q^3-1) q^{2a_1+a_2-3}$ if $a_3>0$.
Moreover, for such $gKZ$ we have $\delta_P(gKZ)^{1/2}= q^{a_3-a_1}$.
Therefore the coefficient of $z_1^{a_1}z_2^{a_2}z_3^{a_3}$ in
$\psi(T_k)$ is equal to $q^{a_1+a_2+a_3}$ or
$q^{a_1+a_2+a_3-3}(q^3-1)$ according to $a_3 = 0$ or $a_3 > 0$. By
symmetry, this gives rise to
$$\psi(T_k)= q^{k}(\sigma_{k,1}+\sigma_{k,2}+\frac{q^3-1}{q^3}\sigma_{k,3}).$$
Noting that
$$\sum^{\infty}_{k=1}\sigma_{k,3}u^k=
((z_1z_2z_3)u^3+(z_1z_2z_3)^2u^6+\cdots)\sum^{\infty}_{k=0}(1+\sigma_{k,1}+\sigma_{k,2})u^k
=
\frac{u^3}{1-u^3}\sum^{\infty}_{k=0}(1+\sigma_{k,1}+\sigma_{k,2})u^k,$$
we obtain
\begin{eqnarray*}
\psi( \sum^{\infty}_{k=1} T_k u^k )
&=& \sum^{\infty}_{k=1} (\sigma_{k,1}+\sigma_{k,2}+\frac{q^3-1}{q^3}\sigma_{k,3})(qu)^k \\
&=& \frac{(q^3-1)u^3}{1-q^3u^3}+\frac{1-u^3}{1-q^3u^3}\sum^{\infty}_{k=1} (\sigma_{k,1}+\sigma_{k,2})(qu)^k.
\end{eqnarray*}
On the other hand, put $G_0 = \bigsqcup^{\infty}_{k=1} T_{k,0}$. One
verifies that the number of elements in $G_0/KZ$ mapped to
$z_1^{a_1}z_2^{a_2}z_3^{a_3}$ by $\chi$ is $q^{2a_1}$ if
$a_2=a_3=0$, $(q-1) q^{2a_1+a_2-1}$ if $a_2>a_3=0$, and
$(q-1)^2q^{2a_1+a_2-2}$ if $a_2\geq a_3>0$. Therefore,
\begin{eqnarray*}
\psi( \sum^{\infty}_{k=1} T_{k,0} u^k )&=& \sum^{\infty}_{k=1} (\sigma_{k,1}+\frac{q-1}{q}\sigma_{k,2}+\frac{(q-1)^2}{q^2}\sigma_{k,3})(qu)^k \\
&=& \frac{q(q-1)^2u^3}{1-q^3u^3}+\frac{1+q u^3 - 2 q^2
u^3}{1-q^3u^3}\sum^{\infty}_{k=1}\sigma_{k,1}(qu)^k
+\frac{(q-1)(1-q^2u^3)}{q(1-q^3u^3)}\sum^{\infty}_{k=1}\sigma_{k,2}(qu)^k.\\
\end{eqnarray*}
Consequently,
\begin{eqnarray*}
& & \psi\bigg(q (\sum^{\infty}_{k=1} T_{k,0} u^k) -
(q-1)(\sum^{\infty}_{k=1} T_k u^k)\frac{1-q^2u^3}{1-u^3}\bigg)
=\sum^{\infty}_{k=0}\sigma_{k,1}(qu)^k + \frac{(q-1)(q^2-1)u^3}{1-u^3}\\
 &=& \frac{z_1 qu}{1-z_1 qu}+\frac{z_2 qu}{1-z_2 qu}+\frac{z_2 qu}{1-z_2 qu} - \frac{3ru^3}{1-u^3}
= u \frac{d}{du} \log \frac{(1-u^3)^r}{(1-z_1 qu)(1-z_2 qu)(1-z_3 qu)}\\
&=& \psi\bigg(u \frac{d}{du} \log\frac{(1-u^3)^r}{I-A_1 u + A_2 q
u^2 -q^3u^3 I} \bigg),
\end{eqnarray*}  where $ r=\frac{(q+1)(q-1)^2}{3}$.
\end{proof}

\section{Parametrizations of simplices in $\B$ and operators}
\subsection{Simplicial complex structure on $\B$} \label{lattices}
The vertices of $\B$ are homothety classes of rank-3 $\oo$-lattices in $F^3$.
Two distinct vertices $[L]$ and $[L']$ are adjacent if they are represented by lattices $L$ and $L'$ such that $\pi L \subset L' \subset  L$ (and hence $\pi L' \subset \pi L \subset  L'$). %The unordered pair $<[L], [L']>$ denotes the edge connecting $[L]$ and $[L']$.
Note that $\pi L$ has index $q^3$ in $L$, $L'$ has index $q^i$ in $L$, and $\pi L$ has index $q^{3-i}$ in $L'$, where $i = 1$ or $2$. %$L'/\pi L$ is a nontrivial proper subspace of the $3$-dimensional space $L/\pi L$ over the residue field of $F$. Its  codimension $i$ is equal to $1$ or $2$.
Call $[L']$ a type $i$ out-neighbor of $[L]$ and the directed edge $([L] , [L'])$ of type $i$; its opposite $([L'], [L])$ has type $3-i$.
An ordered triple $([L], [L'], [L''])$ of three distinct vertices form a pointed chamber if the vertices are represented by lattices $L, L', L''$ such that $\pi L \subset L'' \subset L' \subset  L$. Thus $([L'], [L''], [L])$ and $([L''], [L], [L'])$ are also pointed chambers. The unordered triple $<[L], [L'], [L'']>$ is called a chamber. Hence a chamber yields three pointed chambers. This describes the simplices in $\B$.

\subsection{Parametrization of simplices in $\mathcal B$} \label{complexonG} Each element $g \in G$ gives rise to a rank-$3$ lattice $L_g$ with $\oo$-basis the three columns of $g$, and all rank-$3$ lattices over $\oo$ arise this way.
Changing basis of $L_g$ amounts to right multiplication of $g$ by elements in $K$, and lattices equivalent to $L_g$ result
from multiplying $g$ by the center $Z$. Thus the assignment $gKZ \to [L_g]$ yields a parametrization of the vertices of $\B$ by $G/KZ$. Note that for each vertex $gKZ$, the number $\ord_\pi \det g$ mod $3$ is well-defined, called the type of $gKZ$.

The group $G$ acts transitively on vertices of $\mathcal B$ by left translations. It is straightforward to verify that this action preserves adjacency, the type of edges, and pointed chambers. Moreover, the actions on directed edges and pointed chambers are both transitive. %{\color{red} The vertex $[\oo \oplus \oo \oplus \oo]$ is stabilized by $KZ$ so that the vertices of $\B$ are parametrized by $G/KZ$.
%Note that for each vertex $gKZ$, the number $\ord_\pi \det g$ mod $3$ is well-defined, called the type of $gKZ$.}
Let $\sigma = \left( \begin{matrix}  & 1 &  \\
& & 1 \\ \pi & & \end{matrix}\right)$. It is easy to see that $F_0 := (KZ, \sigma KZ, \sigma^2 KZ)$ is a pointed chamber of $\mathcal B$, whose boundary contains the directed type $1$ edge $E_0 :=(KZ, \sigma KZ)$.
Then $E := K \cap \sigma K \sigma^{-1}$ is the standard parahoric subgroup and $B:=  K \cap \sigma K \sigma^{-1} \cap \sigma^2 K \sigma^{-2}$ is the standard Iwahori subgroup of $K$. As $EZ$ is the stabilizer of $E_0$ and $BZ$ the stabilizer of $F_0$ in $G$, so $G/EZ$ parametrizes all type $1$ (and also all type $2$) edges  and $G/BZ$ parametrizes all pointed chambers of $\mathcal B$.

Write $\mathbb F_q$ for the residue field of $F$. Counting the number of lines and planes in $\mathbb F_q^3$, we see that each vertex has $q^2 + q +1$ type 1 neighbors and $q^2 + q +1$ type $2$ neighbors. Further, the opposite of a type $i$ directed edge has type $3-i$.

\subsection{Operators on vertices $G/KZ$}
The $q^2 + q + 1$ type $1$ neighbors of $gKZ$ are $g\alpha KZ$, where $\alpha KZ$ are the $KZ$-cosets contained in the double
coset of the Hecke operator
\begin{eqnarray*}
A_1 &=& T_{1,0} = K \left(\begin{matrix} 1 &  &  \\ & 1&  \\
& & \pi \end{matrix}\right)KZ\\
&=& \bigcup_{a, b \in \oo / \pi \oo}\left(\begin{matrix} \pi & a &  b \\ & 1& \\
& & 1\end{matrix}\right)KZ \bigcup_{c \in \oo / \pi \oo}
 \left(\begin{matrix} 1 &  &  \\ & \pi& c \\
& & 1\end{matrix}\right)KZ \bigcup  \left(\begin{matrix} 1 &  &  \\ & 1&  \\
& & \pi \end{matrix}\right)KZ.
\end{eqnarray*}
This is because modulo $\pi \oo$, the columns of these coset representatives generate the distinct $2$-dimensional subspaces of $\mathbb F_q^3$. The $q^2 + q + 1$ type $2$ neighbors of $gKZ$ can be similarly described using the $KZ$-coset representatives of
$A_2=T_{0,1}$:
$$\left(\begin{matrix} \pi &  &  b \\ & \pi & c\\
& & 1\end{matrix}\right), \left(\begin{matrix} \pi & a &  \\ & 1 &  \\
& & \pi \end{matrix}\right) ~{\rm and} ~\left(\begin{matrix} 1 &  &  \\ & \pi &  \\
& & \pi \end{matrix}\right), ~{\rm where} ~a, b, c \in \oo /\pi
\oo.$$

\subsection{Operator on type $1$ edges $G/EZ$} Define the out-neighbors of a type $1$ edge $(g_1KZ, g_2KZ)$ to be the type $1$ edges $(g_2KZ, g_3KZ)$ such that $(g_1KZ, g_2KZ, g_3KZ)$ is not a pointed chamber. Since each line in $\mathbb F_q^3$ is contained in $q+1$ planes, among the $q^2+q+1$ type $1$ neighbors $g_3KZ$ of $g_2KZ$, exactly $q+1$ of them will form a pointed chamber $(g_1KZ, g_2KZ, g_3KZ)$. Hence a type $1$ edge has $q^2$ out-neighbors. %Let $L_E$ be a matrix describing the directed adjacency relation among type $1$ edges. More precisely, its rows and columns are parametrized by the type $1$ edges of $G/KZ$ so that the $e e'$ entry is $1$ if $e'$ is an out-neighbor of $e$, and $0$ otherwise.
Expressed in terms of $EZ$-cosets, the out-neighbors of a type $1$ edge $gEZ$ are given by $g \alpha EZ$, where $\alpha EZ$ are the $EZ$-cosets occurring in the double coset
$$L_E = E \left( \begin{matrix} 1 & & \\ & 1 & \\ & & \pi \end{matrix} \right) EZ =  \coprod_{x, ~y \in
~\oo/\pi\oo}
\left( \begin{matrix} 1 &  &  \\
& 1&  \\ x\pi & y\pi &\pi \end{matrix}\right)EZ.$$
It suffices to check the out-neighbors of $EZ$; the rest will follow from the action by $G$. For an $EZ$-coset representative $\alpha = \left( \begin{matrix} 1 &  &  \\
& 1&  \\ x\pi & y\pi &\pi \end{matrix}\right)$ of $L_E$, we have $\alpha = \diag (1, 1, \pi)k$ for some $k \in K$ so that $\alpha KZ = \sigma KZ$. On the other hand, from $\sigma^{-1} \alpha \sigma = \left(\begin{matrix} \pi & x &  y \\ & 1& \\
& & 1\end{matrix}\right) = :\beta$ we see that $\alpha \sigma KZ = \sigma \beta KZ$ is a type $1$ neighbor of $\sigma KZ$ not adjacent to $KZ$ and $\alpha EZ = (\alpha KZ, \alpha \sigma KZ) = (\sigma KZ, \sigma \beta KZ)$ runs through all out-neighbors of $EZ= (KZ, \sigma KZ)$ as $\alpha$ varies.

Similar to $A_1$ and $A_2$, $L_E$ may be regarded as the parahoric
operator on $L^2(G/EZ)$ sending a function $f \in L^2(G/EZ)$  to the function
$L_Ef$  given by
$$ L_Ef (gEZ) =
\sum_{x, ~y \in ~\oo/\pi\oo} f\bigg(g\left( \begin{matrix} 1 &  &  \\
& 1&  \\ x\pi & y\pi &\pi \end{matrix}\right)EZ\bigg).$$

%Since the type $2$ directed edges are the opposite of type $1$ edges, the transpose $L_E^t$ of $L_E$ describes the directed adjacency relation among type $2$ edges of $G/KZ$.

\subsection{Operator on pointed chambers $G/BZ$}\label{faceoperator} Define the out-neighbors of a pointed chamber $(g_1KZ, g_2KZ, g_3KZ)$
 to be $(g_2KZ, g_3KZ, g_4KZ)$ with $g_4KZ \ne g_1KZ$. As remarked above, there are $q+1$ choices of vertices $g_4KZ$ to make $(g_2KZ, g_3KZ, g_4KZ)$ a pointed chamber, so a pointed chamber has $q$ out-neighbors.  In terms of $BZ$ cosets, the out-neighbors of a pointed chamber $gBZ$ are $g \alpha BZ$, where $\alpha BZ$ are the $BZ$-cosets occurring in the Iwahori-Hecke operator
$$L_B = B\left( \begin{matrix} 1 &  &  \\
 & & 1 \\ & \pi & \end{matrix}\right) BZ = \coprod_{x \in \oo/\pi \oo} \left(\begin{matrix} 1 &  & \\  & & 1 \\ \pi x& \pi &  \end{matrix} \right)BZ.$$
To see this, given a $BZ$-coset representative $\alpha =  \begin{pmatrix}1 &  & \\  & & 1\\ \pi x& \pi &  \end{pmatrix}$ of $L_B$, it is straight forward to check that left multiplication by $\alpha$ sends $BZ = (KZ, \sigma KZ, \sigma^2KZ)$ to $\alpha BZ = (\alpha KZ, \alpha \sigma KZ, \alpha \sigma^2KZ) = (\sigma KZ,  \sigma^2 KZ, \alpha \sigma^2 KZ)$, where $\alpha \sigma^2 KZ = \left(\begin{matrix} &  &1 \\ & \pi & \\  \pi^2& & \pi x \end{matrix} \right)KZ \ne KZ$, so that $\alpha BZ$ runs through different out-neighbors of $BZ$ as $\alpha$ varies.

 Similar to the previous cases, $L_B$ may be interpreted as an operator on $L^2(G/BZ)$ which sends a function $f \in L^2(G/BZ)$  to
 $$L_Bf (gBZ) = \sum_{x \in \oo/\pi \oo} f(g \left(\begin{matrix} 1 &  & \\ & & 1\\ \pi x& \pi &  \end{matrix} \right) BZ).$$

\section{Finite quotients of $\mathcal B$}

\subsection{The group $\G$ and the quotient $X_\G$}\label{Gamma}

Let $\Gamma$ be a discrete torsion-free subgroup of $G$ such that $\G \backslash G/Z$ is compact.
Then $\G$ intersects any
compact subgroup of $G$ trivially. Assume that $\G$ intersects any conjugate of $KZ$ trivially and
$ \ord_{\pi} \det \G \subset 3 \mathbb Z.$  For instance we may choose $\G$ to be a subgroup of $\SL_3(F)$. See \cite[\S3]{Sa} for some examples of such $\G$. The action of $\G$ on $\mathcal B$ by left translation is free of fixed points and preserves the types of vertices.
The quotient $X_{\Gamma} = \Gamma\backslash \mathcal B$
is a finite connected $2$-dimensional simplicial complex, whose vertices are the double cosets $\Gamma
\backslash G /KZ$.
Since the vertices in an edge or a chamber of $\mathcal B$ have different types, each edge or chamber of $\mathcal B$ in
the quotient remains an edge or a chamber of $X_\G$.
{Therefore the type $1$ (and also type $2$) edges in $X_\G$ are parametrized by $\G \backslash G/EZ$, and pointed chambers by $\G \backslash G/BZ$.

The operators $A_1$ and $A_2$ on $G/KZ$, $L_E$ on $G/EZ$ and $L_B$ on $G/BZ$ defined  in the previous section induce operators on vertices, types $1$ edges and pointed chambers of $X_\G$, respectively. They will be denoted by the same notation. Since $X_\G$ has finitely many vertices, edges and chambers, these operators can also be interpreted combinatorially. More precisely, $A_i$, $i = 1, 2$, is the matrix parametrized by the vertices $v$ of $X_\G$ such that the $v v'$ entry is the number of type $i$ edges from $v$ to $v'$. As such, $A_2$ is the transpose of $A_1$. Similarly $L_E$ has its rows and columns parametrized by the type $1$ edges $e$ of $X_\G$ so that the $e e'$ entry denotes the number of times when $e'$ is an out-neighbor of $e$. Since type $2$ edges are the opposite of type $1$ edges, the transpose $L_E^t$ of $L_E$ describes adjacency relation among type $2$ edges of $X_\G$. Likewise, $L_B$ can be viewed as the matrix recording the adjacency relation among the pointed chambers of $X_\G$. %{\color{red}Note that the out-neighbors in $X_\G$ may not remain distinct, that is, multiple edges or chambers may occur, while the total number of out-neighbors of an $i$-simplex for $i = 0, 1, 2$ counting multiplicity in $X_\G$ is the same as that in $\B$. }

\subsection{Classification of elements in $\G$}\label{classificationofGamma}
Observe that every element in $\G$ has an eigenvalue in $F$. Indeed,
if $\g \in \G$ has no eigenvalues in $F$, then the
characteristic polynomial of $\g$ is
irreducible over $F$. As $\ord_{\pi}
(\det \g) = 3m$ for some integer $m$, the eigenvalues of $ \g' :=
\pi^{-m} \g$ are units in a cubic extension of $F$, which implies
that $\g$ lies in the intersection of $\G$ with a conjugate of $KZ$,
and hence is the identity element. Together with the fact that
every element in a discrete cocompact-mod-center lattice is semisimple (see \cite{Ra} Thm.1.12),
we arrive at
\begin{theorem}[Classification of elements in $\Gamma$] $ $ \\
Every element $\g$ of $\Gamma$ falls in one of the following types:\\
1) $\g$ is the identity; \\
2) $\g$ is split, that is, it has three distinct eigenvalues in $F^{\times}$;
 \\
3) $\g$ is ramified/unramified rank-one split, that is, {$\g$ has three distinct eigenvalues and} the field $F\langle \g \rangle$ obtained by $F$ joining eigenvalues of $\g$ is a ramfield/unramified quadratic extension of $F$; \\
4) $\g$ is irregular, that is, its eigenvalues are in $F^\times$ and one eigenvalue has multiplicity two.
\end{theorem}

The following conclusion on $\G$ shown in \cite{KLW} results from the closed form expression of the zeta function identity of $X_\Gamma$.

\begin{proposition} [\cite{KLW}, Corollary 4]
$\Gamma$ contains rank-one split elements.
\end{proposition}

\subsection{Rational form}\label{rationalform}
Let $\g$ be a non-identity element in $\G$ and $L= F\langle \gamma \rangle$ be the field over $F$ generated by the eigenvalues of $\g$. If $L = F$, then there is a scalar $z \in Z$ such that  $\g z$ is conjugate to
$r_\g := \diag(1, a, b)$ where $1, a, b \in F^\times$ satisfy
$\ord_{\pi}$$ b \ge \ord_{\pi}$$ a \ge 0$. 
If $L$ is a quadratic extension of $F$,
fix a generator $\lambda$ so that it is a unit if $L$ is
unramified over $F$ and it is a uniformizing element if $L$ is
ramified over $F$. Let $x^2 - bx - c$ be the irreducible polynomial
of $\lambda$ over $F$ and let $\bar {\lambda}$ be the Galois
conjugate of $\lambda$. Then $\ord_{\pi}$$ c = 0$ or $1$ according as
$L$ is unramified or ramified over $F$ and $\ord_{\pi}$$ b \ge \frac
{1}{2}\ord_{\pi}$$ c$. There are elements $a, e, d \in \oo$ with at least one of them a unit such that $a$, $e +
d \lambda$ and $e + d \bar {\lambda}$ are the eigenvalues of $\g z$ for some scalar $z \in Z$. Consequently, up to a scalar multiple, $\g$ is conjugate to
$r_\g :=\left( \begin{matrix} a &  &  \\ & e & dc \\ & d& e+db \end{matrix}\right)$.
Call $r_\g$ a {\it rational form of} $\g$. It is unique modulo scalars in $\oo^\times$, and it depends only on the conjugacy class of $\g$.

\subsection{Homotopy classes of closed paths in $X_\G$}\label{parametrization}

A cycle in $X_\G$ is a closed path starting at a vertex of $X_\G$ and contained in the 1-skeleton of $X_\G$. Repetition of vertices is allowed. A 1-geodesic  between two vertices of $\mathcal B$ is a path in the 1-skeleton which uses the minimal number of edges.
A cycle in $X_\G$ is called $1$-geodesic (resp. geodesic) if it can be lifted to a path in $\B$ which is $1$-geodesic (resp. geodesic).

A $1$-geodesic cycle in $X_\G$ starting at the vertex $\G g KZ$ can be
lifted to a $1$-{geodesic} in $\B$ starting at $gKZ$ and ending at $\g gKZ$ for
some $\g \in \G$. Two such $1$-{geodesic} cycles in $X_\G$ are homotopic
in $X_\G$ if and only if their liftings in $\B$ to two {1-geodesics} starting
at $gKZ$ have the same ending vertex. Denote by $\kappa_\g(gKZ)$ the homotopy class of
the {1-geodesics} from $gKZ$ to $\g gKZ$ in $\B$. When
projected to $X_\G$, these {1-geodesics} become homotopic closed {1-geodesics} which
use least number of edges among all cycles in its homotopy class in $X_\G$. By abuse of
notation, $\kappa_\g(gKZ)$ also denotes the homotopy class of its projection
in $X_\G$. Thus the fundamental group of $X_\G$ based at $\G g KZ$ is
$$\pi_1(X_\G, \G g KZ) = \{\kappa_\g(gKZ) : \g \in \G \}.$$

\noindent Since $\G$ has no fixed points, all $\kappa_\g(gKZ)$ are
distinct and $\pi_1(X_\G, \G g KZ)$ is isomorphic to $\G$.

We shall take all base points into account, but regroup the homotopy classes $\kappa_\g(gKZ)$ with respect to the conjugacy classes of $\G$.
%
%When all base points are taken into account, the set of all vertex-based
%homotopy classes of all closed 1-geodesics in $X_\G$, that is, $\coprod_{\G gKZ \in \G \backslash G/KZ}
%\pi_1(X_\G, \Gamma gKZ)$, can be expressed as $\G \times \G \backslash G/KZ$.
%
 For each conjugacy class of $\G$ fix a
representative $\g$ and denote that class by $\langle \g \rangle_{\G}$. Let $[\G]
= \{ \g \}$ be the set of chosen representatives of conjugacy classes. Denote by $C_\G(\g)$ the centralizer of $\g$ in $\G$. Given $\g \in \G$, the map $h \mapsto h^{-1}\g h$ is a bijection from $C_\G(\g)\backslash \G$
to the conjugacy class $\langle \g \rangle_\G$. %since $\langle \g \rangle_\G = \{h^{-1}\g h : h \in C_\G(\g)\backslash \G \}$,  there is a bijection from the conjugacy class $\langle \g \rangle_\G$ to $C_\G(\g)\backslash \G$.
 So $\G
= \coprod_{\g \in [\G]} \langle \g \rangle_\G$ corresponds bijectively to $\coprod_{\g \in [\G]}
C_\G(\g)\backslash \G$.
Letting, for each $\g \in [\G]$,
\begin{eqnarray}\label{classofgamma}
[\g] = \{\kappa_\g (gKZ) ~|~g \in C_\Gamma(\g)\backslash G/KZ \},
\end{eqnarray}
we obtain the following partition of all vertex-based homotopy classes of closed 1-geodesics in $X_\G$:
\begin{eqnarray*}
\coprod_{\G gKZ \in \G \backslash G/KZ} \pi_1(X_\G, \Gamma gKZ)
&=&  \{\kappa_\g(gKZ) : \g \in \G ,  g \in \G \backslash G/KZ \} \\
&=&  \{\kappa_{h^{-1}\g h}(gKZ) : \g \in [\G], h \in C_\G(\g)\backslash \G  , g \in \G \backslash G/KZ \} .
%&=&  \{\kappa_\g(gKZ) :  \g  \in [\G] , g \in C_\G(\g)\backslash G /KZ \}\\
%&=& \coprod_{\gamma \in [\G] } [\gamma].
\end{eqnarray*}
Note that  $\kappa_{h^{-1}\g h}(gKZ)$ consists of 1-geodesics from $gKZ$ to $h^{-1} \g h g KZ$; left multiplication by $h$ yields a bijection from $\kappa_{h^{-1}\g h}(gKZ)$ to $\kappa_{\g }(h gKZ)$. When $h \in \G$, both $\kappa_{\g }(h gKZ)$ and $\kappa_{h^{-1}\g h}(gKZ)$
project to the same homotopy class of 1-geodesic cycles in $X_\G$. Hence we rewrite
\begin{eqnarray*}
\coprod_{\G gKZ \in \G \backslash G/KZ} \pi_1(X_\G, \Gamma gKZ)
&=&  \{\kappa_\g(gKZ) :  \g  \in [\G] , g \in C_\G(\g)\backslash G /KZ \}\\
&=& \coprod_{\gamma \in [\G] } [\gamma]
\end{eqnarray*}
since $\G$ intersects conjugates of $KZ$ trivially.

\section{Type and lengths}

\subsection{Algebraic length and canonical algebraic length}
Given $g$ in $G$, there is a scalar $z \in F^\times$ such that $g' = z g$ is a matrix in $M_3(\oo) \smallsetminus \pi M_3(\oo)$; call $g'$ a minimally integral matrix associated to $g$. It is unique up to multiplication by $\oo^\times$. Define  the {\it algebraic length} of $g$ to be
$$l_A(g) =  \ord_\pi(\det( g')).$$
Thus we always have $l_A(g_1 g_2) \leq l_A(g_1)+l_A(g_2)$ for $g_1, g_2 \in G$.
Extend the definition of algebraic length to elements $g \in \GL_3(L)$ for any finite extension $L$ over $F$ by
$$ l_A(g) =  \frac{1}{[L:F]} \ord_\pi( N_{L/F}\circ\det(g')), $$
where $g'$ is a minimally integral matrix in $M_3(\mathcal O_L)$ associated to $g$. Note that $l_A(g)$ is independent of the choice of the field $L$ containing entries of $g$, and multiplication by scalars. Analogous to canonical heights, define the {\it canonical algebraic length} of $g$ to be
$$ L_A(g) = \lim_{n \to \infty} \frac{1}{n} l_A(g^n)$$
provided that the limit exists. We exhibit some properties of the canonical algebraic length.% for elements in $\GL_3(F)$.

\begin{proposition}\label{canonicallength}  Let $g$ be a semisimple element in $G$ and let  $d_g$ be a minimally integral diagonal matrix in $\GL_3(L)$ conjugate to $g$ up to a scalar multiple in a finite extension $L$ of $F$. Then

1. $L_A(g)$ exists and is equal to $l_A(d_g) = L_A(d_g)$, hence it is invariant under conjugation;

2. $L_A(g^n)=n L_A(g)$ for all integers $n \ge 1$;

3. $L_A(g) \leq l_A(g).$
\end{proposition}
\begin{proof} By assumption, there is a scalar $z \in L^\times$ and $h \in \GL_3(L)$ such that $zg = h d_g h^{-1}$. Since $d_g$ is a minimally integral diagonal matrix, so is $d_g^n$ and $l_A(d_g^n) = n l_A(d_g)$ for all integers $n > 0$. Thus $L_A(d_g)$ exists and is equal to $l_A(d_g)$. The relation $z^ng^n = h d_g^n h^{-1}$ implies
$$  l_A(d_g^n)-l_A(h)-l_A(h^{-1}) \leq l_A(g^n) \leq  l_A(d_g^n)+l_A(h)+l_A(h^{-1}) $$
for all $n>0$. Thus $L_A(g)$ also exists and equals to $L_A(d_g)$.
The remaining assertions are clear.
\end{proof}

Note that for $g \in T_{n,m}$, its algebraic length is equal to $l_A(g)=n+2m$.
In this case, we say it has {\it type} $(n,m)$ and  {\it geometric length} $l_G(g) = n+m$.
\subsection{Geodesics and lengths in $\mathcal B$}
The building $\B$ is the union of its apartments, and each apartment is an Euclidean plane. It can be shown that all {1-geodesics} between two vertices $g_1KZ$ and
$g_2KZ$ with $g_1^{-1}g_2 \in T_{n,m}$ lie in the same apartment, and
they use $n$ type $1$ edges and $m$ type $2$ edges. We say that they
have type $(n, m)$ (the same type as $g_1\m g_2$), geometric length $n+m = l_G(g_1\m g_2)$ and algebraic length $n+2m = l_A(g_1\m g_2)$. When $m = 0$ (resp. $n = 0$), the path is said to have
{\it type $1$} (resp. {\it type $2$}) for short.  Note that the same path
traveled backwards is of type $(m, n)$ and has algebraic length $m + 2n$. Further, when the
path has type $1$ or $2$, there is only one {1-geodesic} between the two
vertices, and it is a geodesic in the building $\B$, called a geodesic of type $1$ or $2$ accordingly.

\subsection{The type and lengths of a homotopy class}

 The type, geometric length
and algebraic length of a homotopy class $\kappa_\g(gKZ)$ of $X_\G$
are those of $\kappa_\g(gKZ)$ in $\B$. In other words, If $g^{-1}\g g
\in T_{n,m}$, then $\kappa_\g(gKZ)$ has algebraic length
$l_A(\kappa_\gamma(gKZ))=n + 2m$, geometric length
$l_G(\kappa_\g(gKZ))= n + m$, and type $(n,m)$. Moreover,
 $\kappa_\g(gKZ)$ is of type 1 if $m=0$ and type 2 if $n=0$. By assumption,
$\kappa_\g(gKZ)$ has positive length if and only if $\g$ is not
identity.

\subsection{The type and lengths of
$[\g]$}\label{typeandlengthofclass}
Let $\g \in [\G]$ be non-identity, and let $r_\g$ be a rational form of $\g$ defined in \S \ref{rationalform}.
Fix a choice of $P_\g \in G$ such that $r_\g = (P_\g)^{-1} \g P_\g z_\g$ for some $z_\g \in Z$.  As the centralizers of $\g$
 and $r_\g$ in $G$ are related by
  $C_G(\g) = P_\g C_G(r_\g)P_\g^{-1}$, we have $C_\G(\g)P_\g = P_\g C_{P_\g^{-1}\G
P_\g}(r_\g)$, and $[\g]$ may be expressed in two ways:
\begin{eqnarray}\label{classofgamma2}
[\g] &=& \{ \kappa_\g(gKZ) ~|~  g \in C_\G(\g)\backslash G/KZ \} \notag \\
&=& \{ \kappa_\g(P_\g gKZ) ~|~  g \in C_{P_\g^{-1} \G P_\g}
(r_\g)\backslash G/KZ \} .
\end{eqnarray}
The second expression will facilitate our computations later.

 Suppose $r_\g \in T_{n,m}$. We say that
$[\g]$ has type $(n, m)$, algebraic length $l_A([\g]) = n + 2m$ and
geometric length $l_G([\g]) = n + m$. For brevity, call $[\g]$ of type $1$ or $2$ according as $m=0$ or $n=0$. We shall prove
\begin{theorem}\label{charlAandlG} Let $\g \in [\G]$ and $\g \ne id$. Then
\begin{eqnarray*}
 l_A([\g]) = min_{\kappa_\g(g KZ) \in [\g]} ~l_A(\kappa_\g(g KZ)) \quad \text{and} \quad
l_G([\g]) = min_{\kappa_\g(g KZ) \in [\g]} ~l_G(\kappa_\g(g KZ)).
\end{eqnarray*}
Moreover, for $g \in C_G(r_\g)$, we have $l_A(\kappa_\g(P_\g gKZ))=
l_A([\g])$, $l_G(\kappa_\g(P_\g gKZ))= l_G([\g])$ and the type of
$\kappa_\g(P_\g gKZ)$ coincides with the type of $[\g]$.
\end{theorem}

 The second assertion is obvious because $(P_\g g)^{-1} \g P_\g g z_\g = g^{-1} r_\g g = r_\g$ for
 $g \in C_G(r_\g)$.
 The proof of the first assertion is contained in Theorem \ref{minlength} for $\g$ split or irregular,
 and Theorem \ref{rankoneminlength} for $\g$
 rank-one split.

Note that $l_A(\kappa_\gamma(gKZ)) \equiv \ord_{\pi} \det \g  \pmod
3$, hence $l_A(\kappa_\g(gKZ)) = l_A([\g]) + 3m$ for some
non-negative integer $m$.

\subsection{Algebraically minimal and geometrically minimal cycles}\label{taillesscycles}
In view of Theorem \ref{charlAandlG}, a homotopy class $\kappa_\g(gKZ)$ is called {\it algebraically minimal} if its algebraic length agrees with
$l_A([\g])$. Likewise, it is called {\it geometrically minimal} if its geometric length is
$l_G([\g])$.

\begin{proposition} \label{geometricalminimalequaltothesametype}
$\kappa_\g(gKZ)$ is geometrically minimal if and only if $\kappa_\g(gKZ)$ and $[\g]$ have the same type. Moreover, if $\kappa_\g(gKZ)$ is geometrically minimal, then it is also algebraically minimal.
\end{proposition}
\begin{proof}
Suppose $[\g]$ is of type $(n,m)$ and $\kappa_\g(gKZ)$ is of type $(i,j)$. Applying Theorem \ref{charlAandlG} to both $[\g]$ and $[\g^{-1}]$, we have
$$ n+2m \leq i+2j \qquad \mbox{and} \qquad 2n+m \leq 2i+j.$$
If $\kappa_\g(gKZ)$ is geometrically minimal, then $n+m=i+j$. Together with the above inequalities, we conclude that $(i,j)=(n,m)$.
On the other hand, if $\kappa_\g(gKZ)$ and $[\g]$ have the same type, then they obviously have the same algebraic and geometric lengths. Therefore, $\kappa_\g(gKZ)$ is both geometrically and algebraically minimal.
\end{proof}
\begin{corollary} \label{canonicallengthandtype}
If $L_A(\g) = l_A(\kappa_\gamma(gKZ))$ and $L_A(\g^{-1}) = l_A(\kappa_{\gamma^{-1}}(gKZ))$,  then $\kappa_\gamma(gKZ)$ and $[\g]$ have the same type. Consequently $\kappa_\gamma(gKZ)$ is geometrically and algebraically minimal.
\end{corollary}

\begin{proof} Recall that $l_A(\kappa_\g (gKZ)) = l_A( g^{-1} \g g)$ by definition. It follows from
Theorem \ref{charlAandlG} and Proposition \ref{canonicallength} that, for $\g \in \G$,
\begin{equation} \label{lengthcomparision}
L_A(\g)=L_A(r_\g) \leq  l_A(r_\g) =l_A([\g]) \leq l_A(\kappa_\gamma(gKZ)) = L_A(\g).
\end{equation}
Therefore  $l_A([\g])=l_A(\kappa_\gamma(gKZ))$. By the same argument, $l_A([\g^{-1}])=l_A(\kappa_{\g^{-1}}(gKZ)) $.
Suppose $\kappa_\gamma(gKZ)$ is of type $(m,n)$ and $[\g]$ is of type $(m',n')$, then we have $2m+n=2m'+n'$ and $m+2n=m'+2n'$, which implies $(m,n)=(m',n')$.
\end{proof}

\subsection{Tailless cycles}\label{taillesscycles}

Recall that a 1-geodesic cycle is tailless if it remains 1-geodesic when the starting vertex is changed.
We give useful criteria for 1-geodesic  type 1 cycles to be tailless.

\begin{proposition} \label{taillesscriterion}
Let  $\kappa_\g(gKZ)$ be a 1-geodesic  cycle  of  type $1$ and  geometric length $n > 1$
in $X_\G$. The following statements are equivalent:

1. $\kappa_\g(gKZ)$ is tailless;

2. $\kappa_{\g}(gKZ)$ repeated $m$-times is type 1 geodesic for all $m>0$;

3. $\kappa_\g(gKZ)$ is geometrically minimal.
\end{proposition}
\begin{proof}
(1$\Rightarrow$ 2) Suppose $\kappa_\g(gKZ)$ is tailless. Let $g_0KZ \to \cdots \to g_{2n}KZ $ be a lifting in $\B$ of $\kappa_{\g}(gKZ)$ repeated $2$ times. Then $g_{n+i}KZ = \g g_i KZ$ for $0 \le i \le n$. The path $g_iKZ \to g_{i+1}KZ \to \cdots \to  g_{i+n}KZ$ is a geodesic for $i = 0, ... , n-1$ by assumption. Hence $g_0KZ \to \cdots \to g_{2n}KZ $ is a 1-geodesic in $\B$ and thus $\kappa_\g(gKZ)$ repeated twice is a type 1 geodesic cycle in $X_\G$, and so are $\kappa_\g(gKZ)$ repeated $m$ times for $m > 0$.
\\
(2 $\Rightarrow$ 3)
Suppose $\kappa_{\g}(gKZ)$ repeated $m$-times is a type 1 geodesic of length $nm$ for all $m>0$. Then $g^{-1}\g^m g \in T_{nm,0}$ for all $m \geq 0$ and, by Proposition \ref{canonicallength},
$$ L_A(\g) = L_A(g^{-1} \g g ) = \lim_{m \to \infty} \frac{1}{m}l_A(g^{-1}\g^m g)=  \lim_{m \to \infty} \frac{1}{m}m l_A(g^{-1} \g g) = n = l_A(\kappa_\g(gKZ)).$$
As $\kappa_{\g^{-1}}(gKZ)$ is $\kappa_{\g}(gKZ)$ traveled backwards, it is a 1-geodesic cycle of type $2$ and algebraic length $2n$. Further $\kappa_{\g^{-1}}(gKZ)$ repeated $m$ times is a type 2 geodesic for all $m>0$. A similar argument gives $L_A(\g^{-1})=2n=l_A(\kappa_{\g^{-1}}(gKZ))$.
We conclude from  %Proposition \ref{geometricalminimalequaltothesametype}, $[\g]$ and $\kappa_\g(gKZ)$ have the same type.  Together with
Corollary \ref{canonicallengthandtype}  that $\kappa_\g(gKZ)$ is geometrically minimal.
\\
(3 $\Rightarrow$ 1)
Suppose $\kappa_\gamma(gKZ)$ is geometrically minimal.
Let $C: g_0KZ \to \cdots \to g_{2n}KZ = \g^2 g_0 KZ $ be a lifting  in $\B$ of $\kappa_{\g}(gKZ)$ repeated twice.
If we change the starting vertex of $\kappa_{\g}(gKZ)$ to obtain a new cycle, then a lifting in $\B$ of this new cycle is contained in $C$. Thus it suffices to show that $C $ is a 1-geodesic.
By Proposition \ref{geometricalminimalequaltothesametype} and the assumption on $\kappa_\gamma(gKZ)$,
$[\g]$ has type $(n,0)$ and $r_\g$ is of the form
$$ \begin{pmatrix} 1 & & \\ & a & \\ & & \pi^n b \end{pmatrix} \qquad \mbox{or} \qquad
\begin{pmatrix} \pi^n  & \\ & M \end{pmatrix},$$
where $a,b\in \oo^\times$ and $M \in GL_2(\oo)$. In both cases we find $[\g^2]$ of type $(2n,0)$ and $l_G([\g^2])=2n$. As $C$ has geometric length $2n$ and it is homotopic to a 1-geodesic from $g_0KZ$ to $\g^2g_0KZ$, combined with Theorem \ref{charlAandlG}, we get  $2n \ge l_G(\kappa_{\g^2}(gKZ)) \ge l_A([\g^2])=2n$. This shows that $C$ is a 1-geodesic, as desired.
\end{proof}

%\begin{remark} From \S \ref{Gamma}, it is clear that
%$X_\G$ does not have a cycle of length 1 so the condition
%$\kappa_\g(gKZ)$ of geometric length $n > 1$ always holds.
%\end{remark}

\begin{corollary} \label{taillesscriterion3}
 If $[\g]$ contains a tailless geodesic cycle of type $i \in \{1, 2\}$, then $[\g]$ is of type $i$.  In this case
the tailless geodesic cycles in $[\g]$ are those which are geometrically minimal.

\end{corollary}

\subsection{The number of tailless cycles in $[\g]$ and the volume of $[\g]$}\label{volume}
Since $X_\G$ is finite, it contains only finitely many $1$-geodesic cycles with a given algebraic or geometric length. Hence for each $\g \in [\G]$, there are only finitely many cycles $\kappa_\g(gKZ)$ in $[\g]$ with given algebraic or geometric length. Let
\begin{eqnarray}\label{deltaA}
\Delta_A([\g])= \{ gKZ \in G/KZ ~|~
l_A(\kappa_\g(P_\g gKZ)) = l_A([\g])\}.
\end{eqnarray} As noted before,
$\Delta_A([\g]) \supset C_G(r_\g)K/KZ$ and is invariant under left
multiplication by $C_{P_\g^{-1}\G P_\g}(r_\g)$. Define the {\it volume} of $[\g]$ to be
\begin{eqnarray}\label{fundamentaldomain'}
\vol([\g]) = \# \big(C_{P_\g^{-1} \G P_\g}(r_\g) \backslash
C_G(r_\g)/(C_G(r_\g)\cap KZ)\big).
\end{eqnarray}
It follows from (\ref{classofgamma2}) and Theorem \ref{charlAandlG} that the number of algebraically
minimal cycles in $[\g]$ is the cardinality of
$C_{P_\g^{-1}\Gamma P_\g}(r_\g)\backslash \Delta_A([\g])$,
which is at least $\vol([\g])$. This in particular implies the finiteness of $\vol([\g])$.

Set
\begin{eqnarray}\label{deltaG}
 \Delta_G([\g]) = \{ gKZ\in G/KZ ~|~ l_G(\kappa_\g(P_\g gKZ)) = l_G([\g]) \}.
 \end{eqnarray}
\noindent By Theorem \ref{charlAandlG} and Proposition
\ref{geometricalminimalequaltothesametype}, geometrically minimal cycles in $[\g]$ have the same type as $[\g]$, and they are also algebraically
minimal. Thus $\Delta_G([\g]) \subseteq \Delta_A([\g])$. The cardinality of
$C_{P_\g^{-1}\Gamma P_\g}(r_\g)\backslash \Delta_G([\g])$ counts the number of geometrically
minimal cycles in $[\g]$. The first statement below for $\g$ of type $1$ follows from Corollary \ref{typeoftailless} for $\g$ split or irregular and Corollary \ref{rankonetailless} for $\g$ rank-one split, hence it also holds for $\g$ of type $2$. The second statement is from Corollary \ref{taillesscriterion3}.

\begin{proposition}\label{algminequalgeommin} Suppose $\g \in [\G]$ has type $1$ or $2$. Then $\Delta_A([\g]) = \Delta_G([\g])$, i.e., in $[\g]$ there is no distinction among algebraically minimal,   geometrically minimal, and tailless cycles.
\end{proposition}

\section{Edge zeta functions of $X_\G$}

\subsection{Type $1$ and type $2$ edge zeta functions of $X_\G$}
Recall that a type $1$ or $2$ tailless $1$-geodesic is a geodesic in $X_\G$. A cycle is primitive if it is not a repetition of a shorter cycle.
Note that every 1-geodesic cycle $C$ is a repetition of a primitive 1-geodesic cycle $C'$ and the number of 1-geodesic cycles equivalent to $C$ is the geometric length of $C'$, which is equal to the algebraic length of $C'$ if $C$ is of type 1.

Denote by $N_n(X_\G)$ the number of geodesic type $1$ tailless cycles in $X_\G$ of length $n$. In terms of the operator $L_E$ described in \S \ref{Gamma}, we have $N_n(X_\G) = \Tr L_E^n$ for all integers $n \ge 1$. For $i = 1, 2$, define the type $i$ edge zeta function of $X_\G$ to be
\begin{eqnarray}\label{Z1}
 Z_{1,i}(X_\G, u) = \prod_{[C]} (1- u^{l_A([C])})^{-1},
\end{eqnarray}
where $[C]$ runs through the equivalence classes of tailless primitive geodesic cycles of type $i$ in $X_\G$, and $l_A([C ])$ is the algebraic length of any cycle in $[C]$. Similar to
Hashimoto's result \cite{Ha} for graphs, we have

\begin{proposition}\label{Z1andLE}  For $i \in \{1, 2\}$ the type $i$ edge zeta function has the following expressions:

$$ Z_{1,i}(X_\G, u) = \exp\bigg(\sum_{n \ge 1} \frac{N_n(X_\G)}{n}u^{in}\bigg) = \frac{1}{\det(I - L_E u^i)}.$$

\end{proposition}

\begin{proof} Since type $2$ edges are the opposite of type $1$ edges but with twice algebraic length, the number $N_n(X_\G)$ also counts tailless geodesic cycles in $X_\G$ using only type $2$ edges and with algebraic length $2n$. It suffices to prove the case $i=1$.  Taking the logarithmic derivative of (\ref{Z1}) yields
\begin{eqnarray*}
u\frac{d}{du} \log Z_{1,1}(X_\G, u) =  \sum_{[C]} \sum_{m \ge 1} l_A([C]) u^{l_A([C])m}
= \sum_C \sum_{m \ge 1} u^{l_A(C)m}
\end{eqnarray*}
since each primitive class $[C]$ consists of $l_A([C])$ cycles. Here $C$ runs through all tailless primitive geodesic cycles in $X_\G$ of type $1$. Clearly any such $C$ repeated $m$ times is a tailless geodesic cycle with algebraic length $l_A([C])m$, and we obtain all tailless geodesic cycles of type $1$ this way. So the last sum can be rewritten as
$$ \sum_{n \ge 1} N_n(X_\G) u^n = \sum_{n \ge 1} \Tr L_E^n u^n,$$
which, by Lemma 3 of
\cite{ST}, is equal to
$$ u\frac{d}{du} \det(I - L_E u)^{-1}.$$
This proves the proposition up to constant multiples. Finally noting that, as formal power series in $u$, all three expressions have the same constant term, we conclude the equality.
\end{proof}

In the next two sections, we shall enumerate $N_n(X_\G)$ by relating them to conjugacy classes of $\G$.

\section{Homotopy cycles in $[\g]$ for $\g$ split or irregular}
Let $|~|$ be the valuation on $F$ such that $|\pi| = q^{-1}$.
In this section we fix a split or irregular $\g \in [\G]$ with rational form
$r_\g = \diag(1, a, b)$, where $\ord_{\pi}$$ b \ge \ord_{\pi}$$ a \ge
0$.

\subsection{Minimal lengths of homotopy cycles in $[\g]$}

We begin by proving the first assertion of Theorem \ref{charlAandlG}
for the split and irregular cases.

\begin{theorem}\label{minlength}
Suppose $\g \in \G$ is split {or irregular} with $r_\g = \diag(1, a, b)$, where
$\ord_{\pi}$$ b \ge \ord_{\pi}$$ a \ge 0$. Then
\begin{itemize}
\item[(1)]
$l_A([\g])= \ord_{\pi}$$ a + \ord_{\pi}$$ b = min_{\kappa_\g(gKZ) \in
[\g]} ~l_A(\kappa_\g(gKZ))$ and

\item[(2)] $l_G([\g])=
 \ord_{\pi}$$ b = min_{\kappa_\g(gKZ) \in
[\g]} ~l_G(\kappa_\g(gKZ))$.
\end{itemize}
\end{theorem}
\begin{proof} For $\g$ split, the centralizer
$C_G(r_\g)$ consists of the diagonal matrices in $G$ so that, by Iwasawa decomposition, $G = C_G(r_\g)UK$, where
$$ U = \bigg \{\left(\begin{matrix} 1 & x & y \\ & 1 & z \\ & &
1\end{matrix}\right) ~|~ x, y, z \in F~ \rm{modulo}~\oo \bigg \}.$$  It suffices to consider the
lengths of $\kappa_\g(P_\g gKZ)$ with $g \in U$. Write $ g =
\left(\begin{matrix} 1 & x & y \\ & 1 & z \\ & &
1\end{matrix}\right)$.  Then
\begin{eqnarray*}(P_\g g)^{-1} \g P_\g g z_\g = g\m r_\g g &=& \left(\begin{matrix} 1 & x & y \\ & 1 & z \\ & &
1\end{matrix}\right)\m \left(\begin{matrix} 1 &  &  \\ & a &  \\
& & b\end{matrix}\right)\left(\begin{matrix} 1 & x & y \\ & 1 & z \\
& & 1\end{matrix}\right) \\
&=&   \left(\begin{matrix} 1 & x(1-a) & y(1-b)+xz(b-a) \\
& a & z(a-b) \\ & & b\end{matrix}\right) \in
K \left(\begin{matrix} \pi^{e_1} &  &  \\
& \pi^{e_2} & \\ & & \pi^{e_3}\end{matrix}\right) K
\end{eqnarray*}
for some integers $e_1 \le e_2 \le e_3$. In fact, for $1 \le i \le
3$, $e_1 + \cdots + e_i = min_{y} ~\{\ord_{\pi}$$ y\}$ where $y$ runs
through the determinant of all $i \times i$ minors of $g^{-1}r_\g
g$.
Consequently,
\begin{eqnarray}\label{e1}
e_1 = \min \{0, ~\ord_{\pi} x(1-a), ~\ord_{\pi} z(a-b), ~\ord_{\pi}
(y(1-b)+xz(b-a)) \} \leq 0,
\end{eqnarray}
\begin{eqnarray}\label{e1+e2}
e_1+e_2 = \min \{\ord_{\pi} a, ~\ord_{\pi} [x(1-a)z(a-b) -
a(y(1-b)+xz(b-a))] \} \leq \ord_{\pi} a, \end{eqnarray} and
\begin{eqnarray}\label{e1+e2+e3}
e_1+e_2+e_3 = \ord_{\pi} a + \ord_{\pi} b.
\end{eqnarray}
In particular, $e_3 \ge \ord_{\pi}$ $ b$ from the last two
inequalities. Moreover, we have, for any $g \in G$,
\begin{eqnarray}\label{alglengthbound}
\qquad l_A(\kappa_\g(P_\g gKZ))=e_3+e_2+e_1-3e_1=\ord_{\pi} a + \ord_{\pi}
b - 3e_1\geq \ord_{\pi} a + \ord_{\pi} b = l_A([\g])
\end{eqnarray}
 and
\begin{eqnarray}\label{geomlengthbound}
 l_G(\kappa_\g(P_\g gKZ))=e_3-e_1\geq
\ord_{\pi} b - e_1 \geq \ord_{\pi} b = l_G([\g]).
\end{eqnarray}
As noted before, the equalities in (\ref{alglengthbound}) and
(\ref{geomlengthbound}) hold for $g \in C_G(r_\g)$. Therefore
\begin{eqnarray*}
l_A([\g]) = \min_{\kappa_\g(gKZ) \in [\g]} l_A(\kappa_\g(gKZ)) \qquad
\text{ and} \qquad l_G([\g]) = \min_{\kappa_\g(gKZ) \in [\g]}
l_G(\kappa_\g(gKZ)).
\end{eqnarray*}

{For $\g$ irregular, we have either $a = b$ or $a=1$, and the centralizer $C_G(r_\g)$ is isomorphic to $ \GL_2(F)\times Z$  and $G = C_G(r_\g)U_0K$, where $U_0$ consists of the elements in $U$ with $z=0$ (when $a=b$) or $x=0$ (when $a=1$). The above argument still holds.}
This proves the theorem.
\end{proof}

The proof above shows that if $\kappa_\g(P_\g gKZ)$ is algebraically minimal,
then $e_1 = 0$; and it is geometrically minimal
if the additional condition  $e_1 + e_2 =
\ord_{\pi}$$ a$ is satisfied.  By (\ref{e1+e2}), this obviously holds
%this amounts to
%$\ord_{\pi}$$ x(1-a)z(a-b) \ge \ord_{\pi}$$ a$, which obviously holds
when $\ord_{\pi}$$ a = 0$, i.e., $\g$ has type $1$. The
proof above also shows that for $\g$  irregular of type $1$, a tailless $\kappa_\g(P_\g gKZ)$ has $g \in C_G(r_\g)K$.
We record this in

\begin{corollary}\label{typeoftailless} Suppose $[\g]$ has type $1$. Then algebraically minimal cycles in $[\g]$ are geometrically minimal, hence they agree with the tailless cycles in  $[\g]$.
Moreover, if $\g$ is irregular and has type $1$, then the tailless cycles in $[\g]$ are $\kappa_\g(P_\g gKZ)$ with $g \in C_G(r_\g)K$.
\end{corollary}

\subsection{Counting homotopy cycles in $[\g]$ in algebraic length}\label{number-split}
As discuss in \S \ref{volume}, the number of
algebraically minimal cycles in $[\g]$ is the cardinality of
$C_{P_\g^{-1}\Gamma P_\g}(r_\g)\backslash \Delta_A([\g])$ with $\Delta_A([\g])$ defined by (\ref{deltaA}). We showed in the previous section that
for $\g$ irregular,  $\Delta_A([\g]) = C_G(r_\g)K/KZ$ so that the number of algebraically tailless cycles in $[\g]$ is equal to $\vol([\g])$ given by (\ref{fundamentaldomain'}).

The following theorem, stated in terms of a formal power series,
counts the number of homotopy cycles in
$[\g]$ with given algebraic length.

\begin{theorem}\label{numberinaclass}
 Suppose $\g \in [\G]$ is split or irregular with $r_\g = \diag(1, a, b)$. Then
 $$\#(C_{P_\g^{-1}\Gamma P_\g}(r_\g)\backslash \Delta_A([\g])) = \vol([\g])\omega_{[\g]}$$ and
$$\sum_{\kappa_\g(gKZ) \in [\g]} u^{l_A(\kappa_\g(gKZ))} =
\left\{
\begin{array}{lcl}
 \vol([\g])\cdot {\omega_{[\g]}}\cdot u^{l_A([\g])}\frac{1-u^3}{1-q^3
u^3} && {\. if \,}~\g~{\, splits \,},\\
\vol([\g])\cdot {\omega_{[\g]}}\cdot u^{l_A([\g])}\frac{1-u^3}{1-q^2
u^3} && {\, if \,} ~\g~{\,is ~ irregular\,}. \\
\end{array}
\right.
$$
Here $\vol([\g])$ is given by
(\ref{fundamentaldomain'}), {$\omega_{[\g]} = (|1-a||a-b||b-1|)^{-1}$ for $\g$ split, and $\omega_{[\g]} = 1$ for $\g$ irregular.}
\end{theorem}

\begin{proof} {If $\g$ is split, the stabilizer $C_G(r_\g)$ is the diagonal subgroup of $G$ and $G = C_G(r_\g)UKZ$; while if $\g$ is irregular, the stabilizer $C_G(r_\g)$ is $\GL_2(F) Z$, where $\GL_2(F)$ is imbedded in $G$ as diagonal block $\GL_2(F) \times \{ 1\}$ (when $a=1$) or $\{1 \} \times \GL_2(F)$ (when $a=b$), and $G = C_G(r_\g)U_0KZ$. Here $U$ and $U_0$ are as in the proof of Theorem \ref{minlength}. Put $W = U$ or $U_0$ according as $\g$ split or irregular. Then $(C_G(r_\g) \cap KZ)WKZ = WKZ$. Suppose $S$ represents the double cosets in $C_{P_\g^{-1} \G P_\g}(r_\g) \backslash C_G(r_\g)/(C_G(r_\g) \cap KZ)$, then $S$ has cardinality $\rm{vol}([\g])$ by (\ref{fundamentaldomain'}), and $$G = \cup_{h \in S} C_{P_\g^{-1} \G P_\g}(r_\g) h (C_G(r_\g) \cap KZ)WKZ = \cup_{h \in S} C_{P_\g^{-1} \G P_\g}(r_\g) h WKZ.$$

 \begin{lemma} For $\g \in \G$ split or irregular, the elements  $h u$ with $h \in S$ and $u \in W$, where $S$ and $W$ are defined above, are double coset representatives of $C_{P_\g^{-1} \G P_\g}(r_\g) \backslash G/KZ$.
\end{lemma}

\begin{proof} Suppose   $C_{P_\g^{-1} \G P_\g}(r_\g) huKZ = C_{P_\g^{-1} \G P_\g}(r_\g) h' u'KZ$ for $h, h' \in S$ and $u, u' \in W$. Then there is some $c \in C_{P_\g^{-1} \G P_\g}(r_\g)$ such that $huKZ = ch'u'KZ$, i.e., $u^{-1}h^{-1}c h' u' \in KZ$. This together with the definition of $W$ implies that $h^{-1}ch' \in C_G(r_\g)\cap KZ$. Therefore $h$ and $h'$ in $S$ represent the same double coset of $C_G(r_\g)$, hence $h = h'$.  On the other hand, since $\G$ intersects $gZKg^{-1}$ trivially for all $g \in G$ by assumption, the same holds for its conjugate $h^{-1}P_\g^{-1}\G P_\g h$. Now $h^{-1}ch \in (h^{-1}P_\g^{-1}\G P_\g h) \cap KZ$, hence is equal to the identity in $G$. So $c = id$ and consequently $uKZ = u'KZ$. This implies $u = u'$ by definition of $W$, as desired.
\end{proof}

 Since $\kappa_\g(P_\g hgKZ)$ and $\kappa_\g(P_\g gKZ)$ have the same algebraic length for $h \in C_G(r_\g)$ and $g \in G$, we get
$$\sum_{\kappa_\g(P_\g gKZ) \in [\g]} u^{l_A(\kappa_\g(P_\g gKZ))} =
\vol([\g]) ~\sum_{v \in W}  u^{l_A(\kappa_{\g}(P_\g vKZ))},$$
where $W = U$ or $U_0$ according to $\g$ split or irregular.}

To proceed, we compute the sum on the right hand side. First assume $\g$ split so that $W = U$.
Given $v \in U$, write $v = \left(
\begin{matrix}
1 & x & y\\
& 1 & z\\
& & 1
\end{matrix} \right)$.
As computed in the proof of Theorem
\ref{minlength},

$$ (P_\g v)^{-1} \g P_\g v z_\g=  v^{\text{-}1}r_\g v =
\left(
\begin{matrix}
1 & x(1-a) & y(1-b)+ xz(b-a)\\
& a & z(a-b)\\
& & b
\end{matrix} \right) = (v_{i, j}).$$
For fixed $m \ge 0$, we count the number of $v$'s such that
$l_A(\kappa_{\g}(P_\g vKZ)) \leq l_A([\g])+3m$. By
(\ref{alglengthbound}), the constraints are  $ |v_{ij}| \leq q^{m}$
for all $1 \le i,j \le 3$. In other words,
\begin{eqnarray}\label{3m}
 |x(1-a)| \leq q^{m},\quad |z(a-b)|\leq q^{m}
\quad {\rm and} \quad |y(1-b)+ xz(b-a)|\leq q^{m}.
\end{eqnarray}  This implies
$$ |x| \leq q^{m}|1-a|^{-1} \quad {\rm and} \quad |z| \leq
q^{m}|a-b|^{-1} $$ so that the numbers of $x$ and $z$ in
$F/\mathcal{O}_F$ are $q^{m}|1-a|^{-1}$ and $q^m|a-b|^{-1}$,
respectively. Further, for chosen $x$ and $z$, there are
$q^{m}|1-b|^{-1}$ choices of $y$ satisfying the above constraint. We
have shown
\begin{eqnarray}\label{taillessinU}
\#\big\{v \in U \big| ~l_A(\kappa_{\g}(P_\g vKZ))=l_A([\g]) \big\} =
(|1-a||a-b||b-1|)^{-1} = \omega_{[\g]}
\end{eqnarray} and, for $m > 0$,
\begin{eqnarray}\label{taillength3minU}
\#\big\{v \in U \big| ~l_A(\kappa_{\g}(P_\g vKZ))=l_A([\g])+3m \big\}
= (q^{3m}-q^{3m-3})\omega_{[\g]}.
\end{eqnarray}
Put together, this gives
\begin{eqnarray*} \sum_{v \in U} u^{l_A(\kappa_{\g}(P_\g vKZ))} =
\omega_{[\g]}u^{l_A([\g])}\bigg(1+ \sum_{m \ge 1}(q^{3m}-q^{3m-3})u^{3m}\bigg)
=\omega_{[\g]}u^{l_A([\g])}\bigg(\frac{1-
u^3}{1-q^3u^3}\bigg).
\end{eqnarray*}
%\end{proof}

Next consider the case $\g$ irregular so that $W = U_0$. 
Recall that $U_0$ consists of elements in $U$ with $z = 0$ (when $a=b$) or $x=0$ (when $a=1$). Note that $\ord_\pi b > 0$, for otherwise $\g$ would lie in the intersection of $\Gamma$ with a conjugate of $K$, which is trivial. Consequently, $1 - b$ is a unit in $\oo$ so that $|1 - b| = 1$. The argument above restricted to elements in $U_0$ goes through as before, but the three inequalities in (\ref{3m}) are reduced to two with either $x(1-a) = 0$ or $z(a-b) = 0$. This then shows that the number of nonzero $x$ or $z$ is $q^m-1$ and the number of $y$ is $q^m$. Hence we obtain
\begin{eqnarray}\label{taillength3minU0}
\#\big\{v \in U_0 \big| ~l_A(\kappa_{\g}(P_\g vKZ))=l_A([\g])+3m \big\}
= q^{2m}-q^{2m-2},
\end{eqnarray}
which in turn gives
\begin{eqnarray*}
\sum_{v \in U_0} u^{l_A(\kappa_{\g}(P_\g vKZ))} =
u^{l_A([\g])}\bigg(1 + \sum_{m \ge 1} (q^{2m} - q^{2m-2}) u^{3m} \bigg)
= \omega_{[\g]}u^{l_A([\g])} \bigg(\frac{1- u^3}{1-q^2u^3}\bigg).
\end{eqnarray*}
\end{proof}

\subsection{Counting homotopy cycles of type $1$ in $[\g]$}

The theorem below gives the number of type $1$ homotopy cycles in
 $[\g]$ of given algebraic length. The result depends on
the type of $[\g]$.

\begin{theorem}\label{type0cyclesinaclass} With the same notation as in Theorem \ref{numberinaclass}, we have: %(A) Suppose $\g \in \G$ is split with
%$r_\g = \diag(1, a, b)$.
%The following assertions hold.
%\begin{itemize}
\item[(A)] If $[\g]$ splits and is not of type $1$, then
$$\sum_{\kappa_\g(gKZ) \in [\g], ~ \text{type $1$}} u^{l_A(\kappa_\g(gKZ))}=
\vol([\g])\omega_{[\g]}u^{l_A([\g])}(1 - q^{-1})( \frac{1-q^2
u^3}{1-q^3u^3}).$$ Moreover, no type $1$ cycles in $[\g]$ are
geometrically minimal.

 \item[(B)] If $[\g]$ splits and has type $1$, then
$$\sum_{\kappa_\g(gKZ) \in [\g], ~ \text{type $1$}} u^{l_A(\kappa_\g(gKZ))}=
\vol([\g])\omega_{[\g]}u^{l_A([\g])}\bigg(q^{-1}+ (1 -
q^{-1})( \frac{1-q^2 u^3}{1-q^3u^3}) \bigg).$$

\item[(C)] Suppose $\g \in \G$ is irregular. Then $[\g]$ contains no cycles of type $1$ if $[\g]$ is not of type $1$; while if $[\g]$ has type $1$, then
$$\sum_{\kappa_\g(gKZ) \in [\g], ~ \text{type $1$}} u^{l_A(\kappa_\g(gKZ))}=
\vol([\g])\omega_{[\g]} u^{l_A([\g])}.$$
\end{theorem}

\begin{proof} For $\g \in \G$ split or irregular, we have $r_\g = \diag(1, a, b)$, so $[\g]$ has type
$(\ord_{\pi}$$ b - \ord_{\pi}$$ a, ~\ord_{\pi}$$ a)$ and $l_A([\g]) =
\ord_{\pi}$$ b + \ord_{\pi}$$ a$. It has type $1$ if and only of
$~\ord_{\pi}$$ a = 0$. The argument is similar to the proof of Theorem
\ref{numberinaclass}; the difference is that we only need to
consider those $v \in W$ such that $\kappa_{\g}(P_\g vKZ)$ has type
$1$. Here $W = U$ or $U_0$ according as $\g$ split or irregular. So we determine the cardinality of the set
$$\{v \in W~ |~l_G(\kappa_{\g}(P_\g vKZ)) =l_A(\kappa_{\g}(P_\g vKZ)) =
l_A([\g])+3m = \ord_{\pi} ~b + \ord_{\pi} ~a + 3m \}$$ for each $m
\ge 0$. As before, writing $v$ as $\left(
\begin{matrix}
1 & x & y\\
& 1 & z\\
& & 1
\end{matrix} \right)$ and following the proofs of Theorem \ref{numberinaclass} and Theorem \ref{minlength}, we arrive at the
following constraints on $x, y, z \in F/\oo$:
\begin{itemize}
\item[(1)] $\min \{0, ~\ord_{\pi}$$ x(1-a), ~\ord_{\pi}$$ z(a-b),
~\ord_{\pi}$$ (y(1-b)+xz(b-a)) \} = -m,$ ~~~~ and

\item[(2)]  $\min \{\ord_{\pi}$$ a, ~\ord_{\pi}$$ [x(1-a)z(a-b) -
a(y(1-b)+xz(b-a))] \} = -2m.$
\end{itemize}

For $m > 0$, the two constraints are equivalent to
\begin{itemize}
\item[(3)]  $ \ord_{\pi}$$ x(1-a) = -m = \ord_{\pi}$$ z(a-b) ~~\text{
and}~~ \ord_{\pi}$$(y(1-b)+xz(b-a)) \ge -m.$
\end{itemize}
First assume $\g$ splits. The number of $x$ is $(1-q^{-1})q^m|1-a|^{-1}$, the number of
$z$ is $(1-q^{-1})q^m|a-b|^{-1}$, and the number of $y$ is
$q^m|1-b|^{-1}$ so that the total number of $v$ is
$(1-q^{-1})^2q^{3m}\omega_{[\g]}$.  For $m = 0$ and
$\ord_{\pi}$$ a > 0$, the same constraint (3) holds.  In this case the
number of $x$ is $|1-a|^{-1} = 1$, the number of $y$ is $|1-b|^{-1}
= 1$ and the number of $z$ is $(1-q^{-1})|a-b|^{-1}$ so that the
total number of $v$ is $(1-q^{-1})\omega_{[\g]}$. Finally,
when $m = \ord_{\pi}$$ a = 0$, the constraints (1) and (2) are
equivalent to
\begin{itemize}
\item[(4)]  $ \ord_{\pi}$$ x(1-a) \ge 0,~~ \ord_{\pi}$$ z(a-b) \ge 0
~~$ and $~~ \ord_{\pi}$$(y(1-b)+xz(b-a)) \ge 0.$
\end{itemize}
Hence the numbers of $x$, $y$ and $z$ are $|1-a|^{-1}$, $|1-b|^{-1}$
and $|a-b|^{-1}$, respectively, so that the number of $v$ is
$\omega_{[\g]}$. Note that $y = z = 0$ in this case.

Since $\vol([\g])\omega_{[\g]}$ is present in all cases, it
suffices to compute
$$\frac{1}{\vol([\g])\omega_{[\g]}}\sum_{\kappa_\g(gKZ) \in [\g], ~ \text{type $1$}}
u^{l_A(\kappa_\g(gKZ))}.$$ In case $\ord_{\pi}$$ ~a > 0$, namely $[\g]$ does not have type $1$, this sum is
equal to
$$u^{l_A([\g])}(1-q^{-1} + \sum_{m \ge 1}(1-q^{-1})^2q^{3m}u^{3m})
=u^{l_A([\g])}(1 - q^{-1})( \frac{1-q^2 u^3}{1-q^3u^3}),$$ and in
case $\ord_{\pi}$$ ~a = 0$, namely $[\g]$ has type $1$, it is equal to
$$u^{l_A([\g])}(1 + \sum_{m \ge 1}(1-q^{-1})^2q^{3m}u^{3m})
=u^{l_A([\g])}\bigg(q^{-1} + (1 - q^{-1})( \frac{1-q^2
u^3}{1-q^3u^3})\bigg).$$ This proves (A) and (B).

{When $\g$ is irregular, either $a = 1$ or $a = b$, so (3) never holds
and there are no cycles in $[\g]$ of type $1$ and algebraic length $> l_A([\g])$. Further, there are $\rm{vol}([\g])$ cycles in $[\g]$ with algebraic length equal to $l_A([\g])$ and they have the same type as $[\g]$. This proves the assertion (C).}
\end{proof}

Contained in the proof above is the following statement.

\begin{corollary}\label{primitivesplit}
Suppose $\g \in \G$ is split or irregular with $r_\g = \diag(1, a, b)$. Assume
that $\g$ has type $1$, $a \in \oo^\times$ and $n = \ord_{\pi}$$ b$. Let $\delta = \delta([\g]) = \ord_{\pi}$$
(1-a)$ { for $\g$ split, and $\delta = 0$ for $\g$ irregular}. Then
\begin{eqnarray*}
 \Delta_A([\g]) = \{h v_x KZ
~|~ h \in  C_G(r_\g)/(C_G(r_\g) \cap KZ),  ~ v_x = \left(
\begin{matrix}
1 & x & \\
& 1 & \\
& & 1
\end{matrix} \right) ~\text{with} ~
x \in \pi^{-\delta} \oo/\oo \}. \end{eqnarray*}
%and for $hv_xKZ \in
\end{corollary}

\section{Homotopy cycles in $[\g]$ for $\g$ rank-one split}

In this section we fix a rank-one split $\g \in [\G]$ whose
eigenvalues %$a, e+d\lambda, e + d \bar{\lambda}$, where $a, e, d \in
%\oo$ and at least one of them is a unit,
generate a quadratic
extension $L = F(\lambda)$ of $F$. Here $\lambda$ is a unit or
uniformizer in $L$ according as $L$ is unramified or ramified over
$F$, i.e., $\g$  is unramified or ramified rank-one split. Let
$r_\g = \left(
\begin{matrix} a &  &  \\ & e & dc \\ & d& e+db \end{matrix}\right)
$ be a rational form of $\g$ as in \S \ref{rationalform}. Fix a matrix $P_\g$ so that $P_\g^{-1}\g P_\g z_\g= r_\g$ for some $z_\g \in Z$.

\subsection{The centralizers of $r_\g$ for $\g$ rank-one split} \label{centralizers}

 Embed $L^\times$ in $\GL_2(F)$ as the subgroup
\begin{eqnarray}\label{imbeddingL}
\bigg\{\left(\begin{matrix} u & v c\\
v & u + vb \end{matrix}\right) ~|~ u, v \in F, ~\text{not both zero}
\bigg\},
\end{eqnarray} which is further imbedded in $\GL_3(F)$ as
$\bigg\{\left(\begin{matrix} 1 & & \\ &u & v c\\
& v & u + vb \end{matrix}\right)\bigg\}.$ %Embed $F^\times$ into
%$\GL_3(F)$ as the diagonal matrices $\diag(F^\times, 1, 1)$.
Note
that $C_G(r_\g) = L^\times Z$.  Further $C_{P_\g^{-1} \G
P_\g}(r_\g)\backslash C_G(r_\g)/(C_G(r_\g)\cap KZ)$ has cardinality
$\vol([\g])$ by (\ref{fundamentaldomain'}).

 The group of units $\U_L$ of $L^\times$ is
contained in $K$. If $L$ is unramified over $F$, then $L^\times = \langle
\pi \rangle \U_L$ so that $C_G(r_\g)K/KZ $ is represented by the vertices
$\diag(\pi^n, 1, 1)KZ$, $n \in \mathbb Z$, on a line in $\B$, and
$C_{P_\g^{-1} \G P_\g}(r_\g)\backslash C_G(r_\g)/(C_G(r_\g)\cap KZ)$
 represented by  $\diag(\pi^n, 1, 1)KZ$, $n \mod
\vol([\g])$. If $L$ is ramified over $F$, then $L^\times = \langle \pi_L
\rangle \U_L$, where the uniformizer $\pi_L$ does not lie in $F$ and
$\pi_L^2$ differs from $\pi$ by a unit multiple. In this case
$C_G(r_\g) K/KZ $ is represented by the vertices $\diag(\pi^n, 1,
1)KZ$ and $\diag(\pi^n, 1, 1) \pi_L KZ$, $n \in \mathbb Z$, lying on
two lines in $\B$. There are two possibilities for $C_{P_\g^{-1} \G
P_\g}(r_\g)$:

Case (i). The vertices in $C_{P_\g^{-1} \G P_\g}(r_\g)KZ/KZ$ are
contained in the line $\diag(\pi^n, 1, 1)KZ$, $n \in \mathbb Z$. Then
$\vol([\g])$ is even so that $C_{P_\g^{-1} \G P_\g}(r_\g)\backslash
C_G(r_\g)/(C_G(r_\g)\cap KZ)$ is represented by the vertices
$\diag(\pi^n, 1, 1)KZ$ and $\diag(\pi^n, 1, 1) \pi_L KZ$, $n \mod
\vol([\g])/2$.

Case (ii). $C_{P_\g^{-1} \G P_\g}(r_\g)KZ/KZ$ contains a vertex on the
line $\diag(\pi^n, 1, 1) \pi_L KZ$, $n \in \mathbb Z$. Let $y \in
C_{P_\g^{-1} \G P_\g}(r_\g)$ be such that $yKZ = \diag(\pi^N, 1, 1)
\pi_L KZ$ has the least non-negative $N$. Then $y$ generates the
group $C_{P_\g^{-1} \G P_\g}(r_\g)$, $y^2KZ = \diag(\pi^{2N-1}, 1,
1)KZ$, $\vol([\g])= 2N-1$ is odd, and $C_{P_\g^{-1} \G
P_\g}(r_\g)\backslash C_G(r_\g)/(C_G(r_\g)\cap KZ)$ is represented by
the vertices $\diag(\pi^n, 1, 1)KZ$, $0 \le n \le N-1 =
(\vol([\g])-1)/2$, and $\diag(\pi^n, 1, 1)\pi_L KZ$, $0 \le n \le N-2 =
(\vol([\g])-3)/2$.

\subsection{Double coset representatives of $C_G(r_\g)\backslash G
/KZ$}

\begin{proposition}\label{rankonerep} The double cosets in $C_G(r_\g)\backslash G /KZ$ are represented by elements in
$$ S = \bigg\{ \left(\begin{matrix}1 & x & y \\ & 1 & 0\\ & & \pi^n
\end{matrix} \right) ~|~ x, y \in F/\oo, n \ge 0 \bigg\}.$$
%represents the double coset $C_G(r_\g)\backslash G /KZ$.
\end{proposition}
\begin{proof}
Write an element $g \in G$ as $wk$ for some upper triangular $w$ and
some $k \in K$. Since $C_G(r_\g) = L^\times Z$, modulo the center $Z$, we may assume that $w =
\left(\begin{matrix} 1 & x & y
\\ & 1 & z \\ &  & \pi^n
\end{matrix}\right)$, where $x, y, z \in F /\oo$ and $n \in \mathbb
Z$. We are reduced to proving
\begin{eqnarray}\label{Ldoublecoset}
 \GL_2(F) = \coprod_{n \ge 0} L^\times \left(\begin{matrix} 1 &
\\ & \pi^n \end{matrix} \right) GL_2(\oo),
\end{eqnarray}
where $L^\times$ is given by (\ref{imbeddingL}) (cf. \cite{Fl}, Lemma 1 on p.30).

 First we check the disjoint union. Suppose otherwise. Then
there exist $m \ne n$ and $g$ satisfying
$$g \in L^\times \cap \left(\begin{matrix} 1 & \\ & \pi^m
\end{matrix}\right) \GL_2(\oo)\left(\begin{matrix} 1 & \\ & \pi^{-n}
\end{matrix}\right).$$ Replacing $g$ by its inverse if necessary, we may assume $m >
n$. Write $g = \left(\begin{matrix} x & y \pi^{-n}\\\pi^m z&
w\pi^{m-n}\end{matrix}\right) = \left(\begin{matrix} u & vc\\
v & u+vb \end{matrix}\right)$ for some $\left(\begin{matrix} x &
y\\z & w \end{matrix}\right) \in \GL_2(\oo)$ and $u, v \in F$.
Comparing entries, we find $y \pi^{-n} = c z \pi^{m}$ and $w
\pi^{m-n} = x + zb \pi^m$. Since $x, y, z, w, b, c$ are all
integral, we conclude that $x$ is a nonunit and hence $z$ and $y$
should both be units, but then $y \pi^{-n} = c z \pi^{m}$ cannot
hold by checking the order of both sides.

Next we prove equality. Let $w = \left(\begin{matrix} 1 & z\\ &
\pi^{m}
\end{matrix}\right) \in \GL_2(F)$. Observe that for $m
\ge 0$,  $$\left(\begin{matrix} 0 & c\\ 1& b
\end{matrix}\right) \left(\begin{matrix} 1 & 0\\ 0& \pi^m
\end{matrix}\right) = \left(\begin{matrix} 0 & c\pi^m\\ 1& b\pi^m
\end{matrix}\right) = \left(\begin{matrix} c \pi^m & 0\\ 0& 1
\end{matrix}\right)\left(\begin{matrix} 0 & 1\\ 1&
b\pi^{m}
\end{matrix}\right),$$ showing that $\left(\begin{matrix} 1 & 0\\ 0& \pi^m
\end{matrix}\right)$ and $\left(\begin{matrix} 1 & 0\\ 0&
\pi^{-m-\ord_{\pi} c}
\end{matrix}\right)$ represent the same double coset. Since
$\ord_{\pi} c = 0$ or $1$, only such diagonal matrices with $m \ge
0$ are needed as double coset representatives. Thus we assume
$\ord_{\pi} z < 0$. It suffices to reduce $w$ to a diagonal matrix
via left multiplication by elements in $L^\times$ and right
multiplication by elements in $\GL_2(\oo)$.

Case (I). $0 > \ord_{\pi} z \ge m + \ord_{\pi} c$. Choose $v \in
\oo$ with $\ord_{\pi} v + m + \ord_{\pi} c = \ord_{\pi} z$ and $u$ a
unit in $\oo$ satisfying $uz = - c v \pi^m$. Then
$\left(\begin{matrix} u & vc\\ v& u+ vb
\end{matrix}\right)w = \left(\begin{matrix} u & 0\\ v& vz+(u+ vb)\pi^m
\end{matrix}\right) = \left(\begin{matrix} 1 & 0\\ 0& \pi^m
\end{matrix}\right)k$ for some $k \in \GL_2(\oo)$. Here we used the
fact that $u(u + vb) - v^2c$ is a unit. It is obvious if $v$ or
$c$ (and hence $b$) is not a unit; when $v$ and $c$ are both
units, this results from the irreducibility of $x^2 - bx - c$.

Case (II). $m + \ord_{\pi} c > \ord_{\pi} z$. Choose $u \in \oo$
with $\ord_{\pi} u + \ord_{\pi} z = m + \ord_{\pi} c$ and $v$ a unit
such that $uz = -v c \pi^m$. Then $\left(\begin{matrix} u & vc\\ v&
u + vb
\end{matrix}\right)w = \left(\begin{matrix} u & 0\\ v& vz+(u+vb) \pi^m
\end{matrix}\right) = \left(\begin{matrix} u & 0\\ 0& z
\end{matrix}\right)k$ for some $k \in \GL_2(\oo)$.

In both cases we have shown that $w$ lies in the right hand side of
(\ref{Ldoublecoset}), therefore (\ref{Ldoublecoset}) holds. This
proves the proposition.
\end{proof}

\subsection{Minimal lengths of cycles in $[\g]$}\label{minimallengthrankone}

The type of $[\g]$, as defined in \S \ref{typeandlengthofclass}, is
$(n,m)$ such that $r_\g \in T_{n,m} = K \diag(1, \pi^m,
\pi^{n+m})KZ$. Observe that $\ord_{\pi}$$ \det \g \equiv \ord_{\pi}$$ \det
r_\g \equiv \ord_{\pi}$$ a(e+d\lambda)(e+d \bar{\lambda}) \equiv 0  \mod 3$
by the assumption on $\G$. Hence if $e+d\lambda$ is a unit in $L$,
then at least one of $e, d$ is a unit and $a$ is not a unit.
Consequently, $[\g]$ has type $(\ord_{\pi}$$ a, 0)$. Next assume
$e+d\lambda$ is not a unit. We distinguish two cases. If $L$ is
unramified over $F$ (hence $\lambda$ is a unit), then both $e$ and
$d$ are non-units and $a$ is a unit; in this case $[\g]$ has type
$(0, \min(\ord_{\pi}$$ e, \ord_{\pi}$$ d))$. If $L$ is ramified over $F$
(hence $\lambda$ is a uniformizer of $L$), then there are two
possibilities:

  (i) $\ord_{\pi}$$ (e+d\lambda)(e+d \bar{\lambda}) = 1$. This happens
if and only if $e$ is a non-unit, $d$ is a unit, and $\ord_{\pi}$$ a
\ge 2$; in this case
 $[\g]$ has type $(\ord_{\pi}$$ a - 1, 1)$.

 (ii) $\ord_{\pi}$$ (e+d\lambda)(e+d
\bar{\lambda}) > 1$. This happens if and only if both $e$ and $d$
are non-units and $a$ is a unit; in this case $[\g]$ has type $(0,
\ord_{\pi}$$ e)$ if $\ord_{\pi}$$ e \le \ord_{\pi}$$ d$, and type $(1,
\ord_{\pi}$$ d)$ if $\ord_{\pi}$$ e > \ord_{\pi}$$ d$.

This proves the first assertion of

\begin{theorem}\label{rankoneminlength}
Let $\g$ be a rank-one split element in $[\G]$ with rational form $r_\g = \left(
\begin{matrix} a &  &  \\ & e & dc \\ & d& e+db \end{matrix}\right) $.
Suppose that $r_\g \in K \diag(1, \pi^m, \pi^{m+n})KZ$. Then
\begin{itemize}
\item[(1)] The type $(n,~m)$ of $[\g]$ is as follows.
\begin{itemize}
\item[(1.i)] If $\ord_{\pi}$$ c = 0$, then $(n, m) = (\ord_{\pi}$$ a,
~\min\{\ord_{\pi}$$ e, ~\ord_{\pi}$$ d\})$.

\item[(1.ii)] If $\ord_{\pi}$$ c = 1$, then $(n, m) = (\ord_{\pi}$$ a,
~\ord_{\pi}$$ e)$ provided that $\ord_{\pi}$$ e \le \ord_{\pi}$$ d$,
otherwise $(n, m) = (\max\{\ord_{\pi}$$ a -1, ~1\}, ~\max\{\ord_{\pi}$$
d, ~1 \})$.
\end{itemize}

\item[(2)] $l_A([\g])= \min_{\kappa_\g(gKZ) \in [\g]}
l_A(\kappa_\g(gKZ)) = \ord_{\pi}$$ a(e^2 +edb - cd^2) = n+2m$.

 \item[(3)] $l_G([\g])=
  \min_{\kappa_\g(gKZ) \in
[\g]} l_G(\kappa_\g(gKZ)) = n+m$.
\end{itemize}
\end{theorem}

This theorem combined with Theorem \ref{minlength} completes the
proof of Theorem \ref{charlAandlG}.

\begin{remark} If $\g$ is ramified rank-one split and $[\g]$ has
type $(n, 1)$, then $[\g^2]$ has type $(2n+1, 0)$.
\end{remark}

\begin{proof} It remains to show that the algebraic and
geometric lengths of the cycles in $[\g]$ are at least those of
$[\g]$ since, as observed before, the cycles $\kappa_\g(P_\g gKZ)$
with $g \in C_G(r_\g)$
 have the same algebraic
and geometric lengths as $[\g]$. By Proposition \ref{rankonerep}, it
suffices to compute $(P_\g g)^{-1} \g P_\g g z_\g = g^{-1}r_\g g$ for $g
\in S$. Let $g = \left(\begin{matrix}1 & x & y
\\ & 1 & 0\\ & & \pi^i
\end{matrix} \right)$, where $x, y \in F/\oo$ and $i \ge 0$. Then

$$g^{-1}r_\g g
= \left(\begin{matrix} 1 & -x & -y\pi^{-i} \\ & 1 & 0 \\
 & & \pi^{-i}
\end{matrix} \right)
\left(
\begin{matrix} a &  &  \\ & e & dc \\ & d& e+db \end{matrix}\right)
\left(\begin{matrix}1 & x & y \\ & 1 & 0\\ & & \pi^i
\end{matrix} \right)$$
$$= \left(\begin{matrix}a & (a-e)x-dy\pi^{-i} & (a-e-db)y - cdx\pi^i \\ & e &
dc\pi^i\\ & d\pi^{-i} & e+db
\end{matrix} \right) \in K \left(\begin{matrix}\pi^{e_1} &  &  \\ & \pi^{e_2} & \\ & &
\pi^{e_3}
\end{matrix} \right)K.$$
Here $e_1 \le e_2 \le e_3$, and as in the proof of Theorem
\ref{minlength}, we have
\begin{eqnarray}\label{e1'}
e_1 \le \min \{\ord_{\pi} a, -i + \ord_{\pi} d, \ord_{\pi} e\} \le
\min\{\ord_{\pi} a, \ord_{\pi} d, \ord_{\pi} e \} = 0,
\end{eqnarray}
\begin{equation}\label{e1+e2'}
\begin{aligned}
e_1+e_2 &\le& \min \{\ord_{\pi} a e, ~ -i + \ord_{\pi} ad,
~\ord_{\pi}(e^2+bed-cd^2)\} \\
&\le& \min \{\ord_{\pi} a e, ~\ord_{\pi} ad,
~\ord_{\pi}(e^2+bed-cd^2)\} = m,
\end{aligned}
\end{equation} and
\begin{eqnarray}\label{e1+e2+e3'}
e_1+e_2+e_3 = \ord_{\pi} a(e^2+bed-cd^2)= n+2m,
\end{eqnarray}
in which the last upper bound for $e_1+e_2$ can be verified using
the statement (1). Therefore $l_A(\kappa_\g(P_\g gKZ)) = e_1+e_2+e_3
- 3 e_1 \ge e_1+e_2+e_3 = n+2m = l_A([\g])$ since $e_1 \le 0$. The
inequalities (\ref{e1+e2'}) and (\ref{e1+e2+e3'}) together give the
lower bound $e_3 \ge n+2m - m = n+m$, which in turn implies
$l_G(\kappa_\g(P_\g gKZ)) = e_3-e_1 \ge n+m$. This proves the
theorem.
\end{proof}

 As shown in the proof above, if $[\g]$ has type $1$, i.e. $m=0$, then an algebraically minimal
cycle in $[\g]$ satisfies $e_1 = 0$, which implies $e_1 + e_2 \ge 0$
and hence $e_1 + e_2 = 0 $ by (\ref{e1+e2'}) and $e_3 = n$ by (\ref{e1+e2+e3'}). This proves

\begin{corollary}\label{rankonetailless}
Suppose $\g \in [\G]$ is rank-one split. %Then all tailless cycles in
%$[\g]$ are also algebraically tailless, and they have the same type
%as $[\g]$. Moreover, if
If $[\g]$ has type $1$, then the algebraically
minimal cycles in $[\g]$ coincide with the geometrically minimal (hence tailless) cycles in $[\g]$.
\end{corollary}

\subsection{Counting the number of cycles in $[\g]$ in algebraic
length}\label{number-rankone}

As observed before, given $s \in S$, the cycles $\kappa_\g(P_\g gKZ)$ have
the same algebraic length for all $gKZ \in %C_{P_\g^{-1} \G P_\g}(r_\g) \backslash
C_G(r_\g)sK/KZ$. Since $S$ represents
the double coset $C_G(r_\g) \backslash G/KZ$, to count the number
of cycles in $[\g]$ of a given length,
we need to determine the cardinality of
$C_{P_\g^{-1} \G
P_\g}(r_\g) \backslash C_G(r_\g)sK/KZ$ for $s \in S$.  For this, we may take
as representatives the product of representatives of $C_{P_\g^{-1}
\G P_\g}(r_\g) \backslash C_G(r_\g)/(C_G(r_\g)\cap KZ)$ (independent
of $s$) by the representatives of $(C_G(r_\g)\cap KZ)sK/KZ$.
The number of the former representatives is $\vol([\g])$ by
(\ref{fundamentaldomain'}).

It remains
to compute the cardinality of the latter. Recall that $L^\times
\cap K$ consists of the units in $L^\times$, which are identified with
 the matrices
\begin{eqnarray*}
\U_L = \bigg\{\left(\begin{matrix} u & v c\\
v & u + v b \end{matrix}\right) ~|~ u, v \in \oo, u^2 + uvb - cv^2
~\text{is a unit} \bigg\}.
\end{eqnarray*}
Denote by $K'$ the group $GL_2(\oo)$. As analyzed in the proof of
Proposition \ref{rankonerep}, we are reduced to counting, for
each $m \ge 0$, the cardinality of $\U_L \left(\begin{matrix} 1 &
\\  & \pi^m \end{matrix}\right)K'/K'$.

\begin{proposition}\label{Cgammadoublecosets}
$$\#[\U_L
\left(\begin{matrix} 1 &
\\  & \pi^m \end{matrix}\right)K'/K'] =
\left\{
\begin{array}{lcl}
1&& {\,when\,}~m=0,\\
q^m && {\,when\,}~m \ge 1 ~{\, and \,}~\ord_\pi c=1, \\
q^m+q^{m-1}&& {\,when\,}~m \ge 1 ~{\,and\,}~\ord_\pi c=0. \\
\end{array}
\right.
$$
\end{proposition}
\begin{proof} It is clear that the cardinality is $1$ when $m =
0$. Thus assume $m \ge 1$.

Case (I) $\ord_\pi$$ c=1 $. Then any $\left(\begin{matrix} u & v c\\
v & u + v b \end{matrix}\right) \in \U_L$ satisfies $u \in
\oo^\times$. For $n \ge 0$, let
$$\U_L(n) = \bigg\{\left(
\begin{matrix} u &  vc\pi^n \\  v\pi^n & u +vb\pi^n \end{matrix}\right) \in \U_L \bigg| u,v \in \oo^\times \bigg\}$$
so that
$$\U_L = \U_L(\infty) \cup_{n \ge 0} \U_L(n),$$
where
$$\U_L(\infty) = \bigg\{\left(\begin{matrix} u & 0\\
0 & u \end{matrix}\right) ~|~u \in \oo^\times \bigg\}.$$
 One verifies that
$$\U_L(n)\left(
\begin{matrix} 1 &   \\  & \pi^m \end{matrix}\right)K' = \bigcup_{u
\in \oo^\times/\pi^{m-n}\oo}\left( \begin{matrix} \pi^{m-n} & u  \\
& \pi^n
\end{matrix}\right)K'$$
for $0 \le n < m$, and
$$\U_L(n)\left(
\begin{matrix} 1 &   \\  & \pi^m \end{matrix}\right)K' = \left(
\begin{matrix} 1 &   \\  & \pi^m \end{matrix}\right)K'$$
for $n \ge m$ and $n = \infty$. Therefore
$$ \#[\U_L
\left(\begin{matrix} 1 &
\\  & \pi^m \end{matrix}\right)K'/K'] = 1 + \sum_{0 \le n < m}(q^{m-n}-q^{m-n-1})=q^m.$$

Case (II) $\ord_\pi$$ c=0 $. Let
$$\U_L' = \bigg\{\left(
\begin{matrix} u &  vc \\  v & u + vb \end{matrix}\right) \in \U_L \bigg| u \in \oo^\times
\bigg\}$$ and $$\U_L'' =  \bigg\{\left(
\begin{matrix} u  &  vc \\  v & u +vb \end{matrix}\right) \in \U_L \bigg| u \in \pi\oo \bigg\}
$$
so that
$$\U_L = \U_L' \cup \U_L''.$$
As in Case (I), we have
$$ \U_L'\left(
\begin{matrix} 1 &   \\  & \pi^m \end{matrix}\right)K' =
\bigcup_{\substack{ m\geq n\geq 0 \\u \in \oo^\times/\pi^{m-n}\oo}} \left( \begin{matrix} \pi^{m-n} & u  \\
& \pi^n
\end{matrix}\right)K'.$$
One checks that
$$ \U_L''\left(
\begin{matrix} 1 &   \\  & \pi^m \end{matrix}\right)K' =
\bigcup_{z \in \pi\oo/\pi^m \oo} \left( \begin{matrix} \pi^{m} & z  \\
& 1
\end{matrix}\right)K'.
$$
Therefore $$ \#[\U_L \left(\begin{matrix} 1 &
\\  & \pi^m \end{matrix}\right)K'/K']= q^m+q^{m-1}$$
for $m \ge 1$.
\end{proof}

We summarize the above discussion in

\begin{corollary}\label{rankonedoublecosets}
For each $ s = \left(\begin{matrix}1 & x & y \\ & 1 & 0\\ & &
\pi^n
\end{matrix} \right) \in S$, the cardinality of
$C_{P_\g^{-1} \G P_\g}(r_\g) \backslash C_G(r_\g)sK/KZ$ is
$$
\vol([\g])\left\{
\begin{array}{lcl}
1&& {\,when\,}~n=0,\\
q^n && {\,when\,}~n \ge 1 ~\mbox{and $\g$ is ramified rank-one split}, \\ %\ord_\pi c=1, \\
q^n+q^{n-1}&& {\,when\,}~n \ge 1 ~\mbox{and $\g$ is unramified rank-one split}.\\%\ord_\pi c=0. \\
\end{array}
\right.
$$
\end{corollary}

Now we are ready to state the main result of this section.

\begin{theorem}\label{numberinrankoneclass}
Suppose $\g \in [\G]$ is rank-one split with rational form $r_\g = \left(
\begin{matrix} a &  &  \\ & e & dc \\ & d& e+db \end{matrix}\right) $. Set
$\de = \de([\g])= \ord_{\pi}$$ d$ and $\mu = \mu([\g])= \ord_{\pi}$$
((a-e)^2 - db(a-e) - cd^2)$.
\begin{itemize}
\item[(A)] If $\g$ is unramified rank-one split, then the following hold.
\begin{itemize}
\item[(A1)]
$$ \sum_{\kappa_\g(gKZ) \in [\g]} u^{l_A(\kappa_\g(gKZ))}
= \vol([\g])u^{l_A([\g])}\bigg(\frac{q^{\de+1}+q^{\de}-2}{q-1}+
\frac{(q+1)q^{\de+2}u^3}{1-q^3 u^3}\bigg) \bigg(\frac{1-u^3}{1-q^2
u^3}\bigg).$$

\item[(A2)] If $[\g]$ does not have type $1$, then
$$ \sum_{\kappa_\g(gKZ) \in [\g], ~\text{type $1$}} u^{l_A(\kappa_\g(gKZ))}=
\vol([\g])u^{l_A([\g])}\bigg(q^{\de} + q^{\de - 1}+
\frac{(q^2-1)q^{\de +1}u^3}{1-q^3u^3} \bigg).$$

\item[(A3)] If $[\g]$ has type $1$, then
$$ \sum_{\kappa_\g(gKZ) \in [\g], ~\text{type $1$}} u^{l_A(\kappa_\g(gKZ))}=
\vol([\g])u^{l_A([\g])}\bigg(\frac{q^{\de+1}+q^{\de}-2}{q-1}+
\frac{(q^2-1)q^{\de +1}u^3}{1-q^3u^3} \bigg).$$
\end{itemize}

\item[(B)] If $\g$ is ramified rank-one split, then the
following hold.
\begin{itemize}
\item[(B1)]
$$ \sum_{\kappa_\g(gKZ) \in [\g]} u^{l_A(\kappa_\g(gKZ))} =
\vol([\g])q^{\mu}u^{l_A([\g])}\bigg(\frac{q^{\de +1} -1}{q-1} +
\frac{q^{\de +3}u^3}{1-q^3u^3} \bigg)\frac{1-u^3}{1-q^2u^3}.$$

\item[(B2)] If $[\g]$ does not have type $1$, then
$$ \sum_{\kappa_\g(gKZ) \in [\g], ~\text{type $1$}} u^{l_A(\kappa_\g(gKZ))}=
\vol([\g])u^{l_A([\g])}\bigg(q^{\de}(q^{\mu} - \mu) +
\frac{(q-1)q^{\de + \mu +2}u^3}{1-q^3u^3}\bigg).$$

\item[(B3)] If $[\g]$ has type $1$, then
$$ \sum_{\kappa_\g(gKZ) \in [\g], ~\text{type $1$}} u^{l_A(\kappa_\g(gKZ))}=
\vol([\g])u^{l_A([\g])}\bigg(\frac{q^{\de +1} -1}{q-1} +
\frac{(q-1)q^{\de +2}u^3}{1-q^3u^3}\bigg).$$
\end{itemize}
\end{itemize}
Moreover, in each case, if $[\g]$ does not have type $1$, none of
the type $1$ cycles in $[\g]$ are geometrically minimal.
\end{theorem}

\begin{remarks}

1. $\mu = 0$ unless $a, e, c$ are all nonunits, in which case it is
$1$ and $\de = 0$.

2. $\mu = 0$ when $[\g]$ has type $1$.

3. $\de > 0$ in case (A2), while $\de$ may be zero in case (A3).

%4. The right hand side of the identities (A1) - (B3) can be expressed as $\vol([\g])$
%times the orbital integrals at the rank-one split
%element $\g$ of suitably chosen spherical functions on $G$ with fast decay.
\end{remarks}

\begin{proof}
Recall that the algebraic length of a cycle in $[\g]$ is equal to
$l_A([\g]) + 3m$ for some $m \ge 0$. We shall follow the same
notation and computation as in the proof of Theorem
\ref{rankoneminlength}, letting $g$ run through all elements in the
double coset representatives $S$ and computing, for each $m \ge 0$,
the number of cycles $\kappa_\g(P_\g gKZ)$ with $l_A(\kappa_\g(P_\g
gKZ)) \le l_A([\g]) + 3m$ using
Corollary \ref{rankonedoublecosets}. As $g = \left(\begin{matrix}1 & x & y \\ & 1 & 0\\
& & \pi^i
\end{matrix} \right)$, this amounts to computing the number of $x, y
\in F/\oo$ and $i \ge 0$ such that
\begin{eqnarray*}
e_1 = \min \{\ord_{\pi} ((a-e)x -d\pi^{-i}y) , ~\ord_{\pi} (-
cd\pi^ix + (a-e-db)y ),
 -i+\ord_{\pi}d \}\ge -m.
\end{eqnarray*} This is equivalent to $0 \le i \le m+\ord_{\pi}$$d $, $(a-e)x -d\pi^{-i}y
 \in \pi^{-m}\oo$ and $-cd\pi^ix + (a-e-db)y \in \pi^{-m}\oo$.
 Denote $\ord_\pi$$ d$ by $\de$ for short. So for each $0
\le i \le m+  \de $, we solve the following system of linear
equations
\begin{eqnarray}\label{system1}
\left(\begin{matrix}\alpha\\
\beta
\end{matrix}\right)
=
\left(\begin{matrix} a-e & -d\pi^{-i}\\
-cd\pi^i & a-e-db
\end{matrix}\right)
\left(\begin{matrix}x\\
y
\end{matrix}\right)=M \left(\begin{matrix}x\\
y
\end{matrix}\right)
\end{eqnarray}
for $\alpha, \beta \in \pi^{-m}\oo$ and count the distinct pairs
$(x, y) \in F/\oo \times F/\oo$. Recall that $a, e, d$ are
integral, at least one of them is a unit, and $a$ and $e$ cannot
be both units since $\ord_{\pi}$$ \det r_\g > 0$. Let
$$\mu := \ord_{\pi} \det M = \ord_{\pi} ((a-e)^2 - db(a-e) -
cd^2),$$ which is 0 unless $a$, $e$ and $c$ are all nonunits, in
which case it is $1$. Put
$$\ve := \min \{\ord_{\pi} (a-e), -i + \de, \ord_{\pi}
(a-e-bd)\},$$ which is equal to $-i + \de$ if $\de \le i \le m+\de$,
and 0 if $0 \le i < \de$. Then the coefficient matrix $M = k_1
\diag(\pi^{\ve}, \pi^{\mu-\ve})k_2$ for some $k_1, k_2 \in
GL_2(\oo)$. Thus system (\ref{system1}) has the same number of
solutions as the system
\begin{eqnarray}\label{system2}
\left(\begin{matrix}\alpha\\
\beta
\end{matrix}\right)
=
\left(\begin{matrix} \pi^{\ve} &  \\
  & \pi^{\mu -\ve}
\end{matrix}\right)
\left(\begin{matrix}x\\
y
\end{matrix}\right)
\end{eqnarray}
for $\alpha, \beta \in \pi^{-m}\oo$ and $(x, y) \in F/\oo \times
F/\oo$. We get the solutions $x \in \pi^{-m -\ve}\oo/\oo$ and $y \in
\pi^{-m -\mu +\ve}\oo/\oo$ so that there are $q^{2m+\mu}$ different
pairs $(x, y)$ for each $0 \le i \le m+  \de $. To proceed, we
distinguish two cases.

Case (A) $\ord_{\pi}$$ c = 0$, that is, $\g$ is unramified rank-one
split. Then $\mu = 0$. By Corollary \ref{rankonedoublecosets}, the
number of classes in $[\g]$ with algebraic length at most
$l_A([\g])+3m$ is
\begin{eqnarray*}
\vol([\g])q^{2m}(1 +
 \sum_{1 \le n \le m+\de} q^n + q^{n-1}) &=&
\vol([\g])q^{2m}(\frac{q^{m+\de}-1}{q-1} + \frac{q^{m+\de+1}-1}{q-1})\\
&=& \frac{\vol([\g])}{q-1}(q^{3m+\de+1}+ q^{3m+\de} -2q^{2m}).
\end{eqnarray*}
Therefore
\begin{eqnarray*}
& & \sum_{\kappa_\g(gKZ) \in [\g]} u^{l_A(\kappa_\g(gKZ))} =
\sum_{\kappa_\g(P_\g gKZ) \in [\g]} u^{l_A(\kappa_\g(P_\g gKZ))} \\
&=&\vol([\g])u^{l_A([\g])}\frac{1}{q-1}\bigg(q^{\de+1} + q^{\de} -2 +\\
& &~~~~~~~~\sum_{m
\ge 1}(q^{3m+\de+1}+q^{3m+\de}-2q^{2m} -q^{3m+\de-2}-q^{3m+\de-3}+2q^{2m-2}) u^{3m}\bigg) \\
&=& \vol([\g])u^{l_A([\g])}\frac{1}{q-1}\bigg(\frac{q^{\de+1} +
q^{\de}}{1 - q^3u^3} - \frac{2}{1-q^2u^3}\bigg)(1-u^3)\\
&=& \vol([\g])u^{l_A([\g])}\bigg(\frac{q^{\de+1}+q^{\de}-2}{q-1}+
\frac{(q+1)q^{\de+2}u^3}{1-q^3 u^3}\bigg) \bigg(\frac{1-u^3}{1-q^2
u^3}\bigg).
\end{eqnarray*}

Among the cycles with $l_A(\kappa_\g(P_\g gKZ)) = l_A([\g]) + 3m$, we
compute the number of those with type $1$.  First consider the case
$m \ge 1$. In order that $l_A(\kappa_\g(P_\g gKZ)) = l_A([\g]) + 3m$
and $\kappa_\g(P_\g gKZ)$ has type $1$, two conditions must be
satisfied:
$$
e_1 = \min \{\ord_{\pi} ((a-e)x -d\pi^{-i}y) , ~\ord_{\pi} (-
cd\pi^ix + (a-e-db)y ),
 -i+\de \}= -m, $$ and
\begin{eqnarray*}
e_1+e_2 = \ord_{\pi} [((a-e)x -d\pi^{-i}y)(e+db) - d\pi^{-i}(-
cd\pi^ix + (a-e-db)y)] = -2m.
\end{eqnarray*}
These two conditions are equivalent to $i = \de +m$, $ \ord_{\pi}$$
(-cd\pi^ix + (a-e-db)y ) = -m,$ and $\ord_{\pi}$$ ((a-e)x
-d\pi^{-i}y)\ge -m.$ This amounts to solving system (\ref{system1})
with $\alpha \in \pi^{-m}\oo$ and $\beta \in \pi^{-m}\oo^\times$,
hence we obtain $(q-1)q^{2m-1}$ distinct pairs $(x, y)$. Combined
with Corollary \ref{rankonedoublecosets}, we see that the number of
type $1$ cycles $\kappa_\g(P_\g gKZ)$ with $l_A(\kappa_\g(P_\g
gKZ))=l_A([\g]) + 3m$ is $\vol([\g])(q-1)q^{2m-1}(q^{\de + m} + q^{\de +
m -1})$.

Next consider the case $m=0$. Under the assumption $\ord_{\pi}$$ c =
0$, we know from Theorem \ref{rankoneminlength} that $[\g]$ has
type $(\ord_{\pi}$$ a, \min\{\ord_{\pi}$$ e, \ord_{\pi}$$ d\})$.
Therefore it has type $1$ if and only if $\ord_{\pi}$$ a
> 0$, in which case all cycles in $[\g]$
with algebraic length equal to $l_A([\g])$ have type $1$, and the
number of such cycles is $\vol([\g])\frac{q^{\de+1}+q^{\de}-2}{q-1}$,
as computed above. If $[\g]$ does not have type $1$, then $\de =
\ord_{\pi}$$ d > 0$; the condition $e_1 = e_2 = 0$ implies $i = \de$
and only one solution $(x, y)=(0, 0)$. In this case the number of
type $1$ cycles in $[\g]$ with algebraic length equal to $l_A([\g])$
is $q^{\de} + q^{\de - 1}$ by Corollary \ref{rankonedoublecosets}.
Put together, we have shown the following:

If $[\g]$ has type $1$, then
\begin{eqnarray*}
& &\sum_{\kappa_\g(gKZ) \in [\g], ~\text{type $1$}}
u^{l_A(\kappa_\g(gKZ))} \\
&=&
\vol([\g])u^{l_A([\g])}\bigg(\frac{q^{\de+1}+q^{\de}-2}{q-1}+ \sum_{m
\ge 1}(q-1)q^{2m-1}(q^{\de + m} + q^{\de + m -1})u^{3m} \bigg) \\
&=& \vol([\g])u^{l_A([\g])}\bigg(\frac{q^{\de+1}+q^{\de}-2}{q-1}+
\frac{(q^2-1)q^{\de +1}u^3}{1-q^3u^3} \bigg),
\end{eqnarray*}
while if $[\g]$ does not have type $1$, then
\begin{eqnarray*}
 \sum_{\kappa_\g(gKZ) \in [\g], ~\text{type $1$}}
u^{l_A(\kappa_\g(gKZ))} = \vol([\g])u^{l_A([\g])}\bigg(q^{\de} + q^{\de
- 1}+ \frac{(q^2-1)q^{\de +1}u^3}{1-q^3u^3} \bigg).
\end{eqnarray*}

Case (B) $\ord_{\pi}$$ c = 1$, that is, $\g$ is ramified rank-one
split. Then $\mu = 0$ or $1$. The same computation as in Case (A)
together with Corollary \ref{rankonedoublecosets} shows that the
number of classes in $[\g]$ with algebraic length at most
$l_A([\g])+3m$ is
$$\vol([\g])q^{2m+\mu}\sum_{0 \le n \le m+  \de} q^n =
\vol([\g])q^{2m+\mu}\frac{q^{m+\de +1} - 1}{q-1} =
\vol([\g])\frac{q^{\mu}}{q-1}(q^{3m+\de +1} - q^{2m}).$$ Therefore
\begin{eqnarray*}
& &\sum_{\kappa_\g(gKZ) \in [\g]} u^{l_A(\kappa_\g(gKZ))}\\
&=&\vol([\g])\frac{q^{\mu}}{q-1}u^{l_A([\g])}\bigg(\sum_{m \ge
0}(q^{3m+\de +1} - q^{2m})u^{3m} - \sum_{m \ge
1}(q^{3m+\de -2} - q^{2m-2})u^{3m}\bigg) \\
&=& \vol([\g])\frac{q^{\mu}}{q-1}u^{l_A([\g])}\bigg(\frac{q^{\de
+1}}{1-q^3u^3} - \frac{1}{1-q^2u^3}\bigg)(1-u^3) \\
&=& \vol([\g])q^{\mu}u^{l_A([\g])}\bigg(\frac{q^{\de +1} -1}{q-1} +
\frac{q^{\de +3}u^3}{1-q^3u^3} \bigg)\frac{1-u^3}{1-q^2u^3}.
\end{eqnarray*}

 Now we compute the number of type $1$ cycles $\kappa_\g(P_\g gKZ)$ with
  algebraic length $l_A(\kappa_\g(P_\g gKZ)) = l_A([\g]) + 3m$.
First consider the case $m \ge 1$. Following the same argument as in
Case (A) and applying Corollary \ref{rankonedoublecosets}, we see
that the number of such cycles is $\vol([\g])(q-1)q^{2m +\mu -1}q^{\de
+ m}$.

Next we discuss the remaining case $m = 0$. By Theorem
\ref{rankoneminlength}, $[\g]$ has type $1$ if and only if
$\ord_{\pi}$$ a > 0$ and $\ord_{\pi}$$ e = 0$, in which case all cycles
in $[\g]$ with algebraic length equal to $l_A([\g])$ are of type
$1$, and the number of such cycles is $\vol([\g])q^{\mu}\frac{q^{\de
+1} -1}{q-1}$. When $[\g]$ does not have type $1$, we have
$\ord_{\pi}$$ e > 0$; the condition $e_1 = e_2 = 0$ implies $i = \de$.
Moreover, if $\mu= 0$, in which case $a$ is a unit, then there is
only one pair $(x, y) = (0, 0)$; while if $\mu = 1$, in which case
$a$ is not a unit, then there are $q-1$ pairs $(x, y) = (0, y)$ with
$ y \in \pi^{-1}\oo^\times/\oo$ so that $\ord_{\pi}$$ (-cd\pi^ix +
(a-e-db)y ) = 0$. Consequently, when $[\g]$ does not have type $1$,
the number of type $1$ cycles in $[\g]$ with algebraic length equal
to $l_A([\g])$ is $\vol([\g])q^{\de}$ if $\mu = 0$, and
$\vol([\g])(q-1)q^{\de}$ if $\mu = 1$. In other words, it is
$\vol([\g])q^{\de}(q^{\mu} - \mu)$. Summing up, we have proved the
following:

If $[\g]$ has type $1$, then
\begin{eqnarray*}
\sum_{\kappa_\g(gKZ) \in [\g], ~\text{type $1$}}
u^{l_A(\kappa_\g(gKZ))} &=&
\vol([\g])u^{l_A([\g])}q^{\mu}\bigg(\frac{q^{\de +1} -1}{q-1} +
\sum_{m \ge 1}(q-1)q^{3m +\de -1}u^{3m}\bigg) \\
&=& \vol([\g])u^{l_A([\g])}q^{\mu}\bigg(\frac{q^{\de +1} -1}{q-1} +
\frac{(q-1)q^{\de +2}u^3}{1-q^3u^3}\bigg),
\end{eqnarray*}
while if $[\g]$ does not have type $1$, then
\begin{eqnarray*}
\sum_{\kappa_\g(gKZ) \in [\g], ~\text{type $1$}}
u^{l_A(\kappa_\g(gKZ))} &=& \vol([\g])u^{l_A([\g])}\bigg(q^{\de}(q^{\mu}
- \mu) + \frac{(q-1)q^{\de + \mu +2}u^3}{1-q^3u^3}\bigg).
\end{eqnarray*}
This completes the proof of the theorem.
\end{proof}

Contained in the proofs of Corollary
\ref{rankonedoublecosets} and Theorem \ref{numberinrankoneclass} is
 the proposition below, in which
\begin{eqnarray}\label{ginuandgiz}
g_{i,j,u} = \left(
\begin{matrix}
1 &  & \\
& \pi^{i-j} & u\\
& & \pi^j
\end{matrix} \right) \quad \text{and} \quad g_{i,z} = \left(
\begin{matrix}
1 &  & \\
& \pi^{i} & z\\
& & 1
\end{matrix} \right).
\end{eqnarray}

\begin{proposition}\label{rankonenumberofalgtailless}
Let $\g \in [\G]$ be rank-one split with $r_\g = \left(
\begin{matrix} a &  &  \\ & e & dc \\ & d& e+db \end{matrix}\right) $. Set
$\de = \de([\g])= \ord_{\pi}$ $ d$. Suppose that $[\g]$ has type $1$
with $ n = \ord_\pi$ $ a$.
If $\g$ is ramified rank-one split, then
\begin{eqnarray*}
\Delta_A([\g]) &=& \{h g_{i,j,u} KZ ~|~ h \in  C_G(r_\g)/(C_G(r_\g)
\cap KZ), ~0 \le j \le i
\le \delta, \\
& & ~~u \in \oo^\times/\pi^{i-j}\oo ~\text{for} ~j < i,
  ~\text{and} ~ u = 0 ~\text{for} ~j = i\};
\end{eqnarray*} while if $\g$ is unramified rank-one split, then
\begin{eqnarray*}
\Delta_A([\g]) = \{h g_{i,j,u} KZ~|~ h ~\text{and} ~g_{i,j,u}
~\text{as above}\} ~\cup ~\{hg_{i,z}KZ ~|~ h ~\text{as above}, ~1 \le
i \le \delta, z \in \pi \oo/\pi^i \oo \}.
\end{eqnarray*}
 Consequently, the number of
algebraically minimal cycles in $[\g]$ is
\begin{eqnarray*}
\#(C_{P_\g^{-1} \G P_\g}(r_\g)\backslash \Delta_A([\g])) =
\vol([\g])\omega_{[\g]},
\end{eqnarray*} where
\begin{eqnarray*}
\omega_{[\g]} = \left\{
\begin{array}{ll}
\frac{q^{\de+1}+q^{\de}-2}{q-1} & \mbox{if $[\g]$ is unramified rank-one split}, \\
\frac{q^{\de +1} -1}{q-1}& \mbox{if $[\g]$ is ramified rank-one
split}.
\end{array}
\right.
\end{eqnarray*}
\end{proposition}

\bigskip

 We end this subsection by comparing $ \Delta_G([\g])$ and $ \Delta_G([\g^2])$, where $ \Delta_G([\g])$ is defined by (\ref{deltaG}).
Suppose $[\g]$ is of type $(m,n)$. If $[\g^2]$ is of type $(2m,2n)$, then a geometrically minimal 1-geodesic $\kappa_\g (P_\g g KZ)$ repeated twice is still geometrically minimal, hence $\Delta_G([\g])\subseteq \Delta_G([\g^2])$.

If $[\g]$ is not of type $(2m,2n)$, then $\g$ is ramified rank-one split of type $(n,1)$ or $(1,n)$.
Assume first that $\g$ is of type $(n,1)$ so that $\mu=1$ and $\delta=0$ in Theorem \ref{numberinrankoneclass}. In this case, there are $q \cdot \vol([\g])$
 algebraically minimal geodesics in $[\g]$. Among these, $(q-1) \vol([\g])$ of them are of type 1, and $\vol([\g])$ of them are of type $(n,1)$. The latter ones are also geometrically minimal.
On the other hand $\kappa_\g (P_\g g KZ)$ is geometrically minimal for all $g \in C_G(r_\g)$.
We conclude that $\kappa_\g (P_\g g KZ)$ is geometrically minimal if and only if $g \in C_G(r_\g)$.
As $C_G(r_\g) \subseteq C_G(r_\g^2)$, any $g\in C_G(r_\g)$ gives rise to a geometrically minimal cycle $\kappa_{\g^2} (P_\g g KZ)$. This shows $\Delta_G([\g])\subseteq \Delta_G([\g^2])$ if $\g$ is of type $(n,1)$.

Finally, note that $\kappa_{\g^{-1}} (P_\g g KZ)$ and $\kappa_{\g} (P_\g g KZ)$ have the same geometric length but opposite types, so the same conclusion holds for $\g$ of type $(1, n)$.
We have shown
\begin{proposition} \label{Deltag2} For $\g \in \G$ we have
$\Delta_G([\g])\subseteq \Delta_G([\g^2])$.
\end{proposition}

\subsection{Counting the number of tailless cycles in $X_\G$ of given algebraic length}
Recall that
$N_n(X_\G)$ counts the number
of tailless cycles of type $1$ in $X_\G$ with algebraic length $n$. These cycles fall in the disjoint union of $[\g]$ as $[\g]$ runs through type $1$ conjugacy classes of $\G$, and Theorem \ref{numberinaclass} and Proposition \ref{rankonenumberofalgtailless} give the number of such cycles in each $[\g]$. Combined with Proposition \ref{Z1andLE} we obtain the following explicit
expressions of the edge zeta functions.
\begin{theorem}\label{logedgezeta} For $i = 1, 2$ we have
$$ u \frac{d}{du} \log Z_{1,i}(X_\G, u) = \sum_{\g \in [\G],~ [\g] ~\rm{of~type} ~1} i ~{\rm vol}([\g]) \omega_{[\g]} ~u^{i~l_A([\g])},$$
 where $\rm{vol}([\g])$ is defined by (\ref{fundamentaldomain'}) and $\omega_{[\g]}$ is as in Theorem \ref{numberinaclass} and Proposition \ref{rankonenumberofalgtailless}.
\end{theorem}

\section{Galleries and Pointed Galleries in $X_\G$}

\subsection{Iwahori-Hecke algebra on $\B$}
Recall that the pointed chambers on $\B$ are parametrized by cosets in $G/BZ$ with the Iwahori subgroup $B$,  admitting the action of $G$
by left translation.
The matrices
$$t_1= \left( \begin{matrix}  &  & \pi^{-1} \\
 &1 &  \\\pi & & \end{matrix}\right), \quad t_2=\left( \begin{matrix}  1 &  & \\
 &  & 1 \\\ & 1 &  \end{matrix}\right),    \quad  \quad \text{and} \quad
 t_3=\left( \begin{matrix}  & 1 & \\
 1& &  \\\ & & 1\end{matrix}\right)$$
 generate the Weyl group $W$ of
 $\SL_3(F)$ subject to the relations $t_i^2 = Id$ and $(t_i t_j)^3
= Id$ for $i \ne j$.

The extended affine Weyl group of $G$ is $ W \ltimes \langle \sigma\rangle$, where
$\sigma =\left( \begin{matrix}  & 1 &  \\
& & 1 \\ \pi & & \end{matrix}\right)$, as in \S \ref{complexonG}, so that
$$ G= \coprod_{w \in W \ltimes \langle \sigma\rangle}  B w BZ. $$
Each element $w \in W \ltimes \langle \sigma \rangle$ defines an operator $L_{w}$
on $L^2(G/BZ)$ by sending a function $f$ to $L_{w}f$ given by
$$ L_{w} f (gBZ) = \sum_{w_i BZ \in BwBZ/BZ} f(g w_i BZ)\quad \qquad \text{for all} ~gBZ.$$
They form a generalized Iwahori-Hecke algebra satisfying
the following relations (cf. \cite{Ga}):
\begin{itemize}
\item[1.] $L_{t_i}\cdot L_{t_i}= (q-1) L_{t_i} + q Id$,

\item[2.] $L_{t_i}\cdot L_{t_j}=  L_{t_i t_j}$ for $i \ne j$,

\item[3.] $L_{t_i}\cdot L_w = L_{t_i w}$ if the length of $t_iw$
is $1$ plus the length of $w$,

 \item[4.]
$L_{\sigma} \cdot L_{t_i}= L_{ \sigma t_i }= L_{
t_{i+1}\sigma} $ for $i = 1, 2, 3$.
\end{itemize} As explained in \S \ref{faceoperator}, the operator
\begin{eqnarray}\label{LB}
L_B = L_{t_1 \sigma}
\end{eqnarray} describes out-neighbors of a pointed chamber. The above relations
imply $(L_B)^{n} = L_{t_1t_2\cdots t_{n}\sigma^{n}}$ for $n \ge 1$. Here the indices are read modulo $3$.

\subsection{Galleries in $\B$}\label{galleries} Two chambers are adjacent if they share a common edge.
Paths formed by adjacent chambers are called
galleries. A gallery between two chambers is called a {\it geodesic gallery} if it
contains the least number of intermediate chambers.
Let $\tilde{B}$ be the stabilizer in $G$ of the chamber with vertices $KZ$, $\sigma KZ$ and $ \sigma^2 KZ$. Thus it  is generated by $B, Z$ and $\sigma$. Since $G$ acts transitively on all chambers,
we can parametrize chambers by $G/\tilde{B}$.
Notice that a chamber $g\tilde{B}$ gives rise to three pointed chambers: $gBZ$, $g\sigma BZ$ and $g \sigma^2 BZ$.

To get a geodesic gallery from $g_1\tilde{B}$ to $g_2\tilde{B}$, we find an element
$w \in W $ such that $g_1^{-1} g_2 \in \tilde{B}w\tilde{B}$ and
write $w = t_{i_1}\cdots t_{i_n}$ as a word using the least number of reflections $t_1, t_2, t_3$; call $n$ the {\it length} of
the gallery. Note that $w$ is unique up to conjugation by some power of $\sigma$.
Write $g_1^{-1} g_2 = b w b'$ for some $b, b' \in \tilde{B}$. Since $g_1\tilde{B}=g_1b\tilde{B}$, we may assume $b$ is the identity so that
$$ g_1 \tilde{B} \to g_1 t_{i_1} \tilde{B} \to \cdots \to  g_1  t_{i_1}\cdots  t_{i_n} \tilde{B} = g_2\tilde{B}$$
represents a geodesic  gallery from $g_1\tilde{B}$ to $g_2\tilde{B}$.
Moreover, since $\sigma \in \tilde{B}$ and $\sigma t_i \sigma^{-1} = t_{i+1}$ for all $i$, replacing $g_1$ by $g_1 \sigma^{1-i_1}$ if necessary, we may assume that $t_{i_1}=t_1$.

 All geodesic galleries from $g_1\tilde{B}$ to $g_2\tilde{B}$ have
length $n$; different galleries arise from different expressions of
$w$ as a product of generators, and they are regarded as {\it
homotopic}. Like the case of paths, given two distinct chambers
$g_1\tilde{B}$ and $g_2\tilde{B}$, there is only one homotopy class of geodesic
galleries in $\B$ from $g_1\tilde{B}$ to $g_2\tilde{B}$.

Observe that a geodesic gallery arising from $w = t_{i_1}\cdots
t_{i_n}$ is a straight strip if and only if the difference $
i_{k+1} - i_k$ remains the same mod 3 for $ 1 \le k \le n-1$. It is said
to have type $1$ or $2$ according to the common difference being $1$
or $2$. Note that the homotopy class of a gallery of type $1$ or $2$
contains only one geodesic gallery, thus we shall drop the word
"homotopy" in this case.
Further, a geodesic gallery of type 1 can always be represented by
$$ g_1 \tilde{B} \to g_1 t_{1} \tilde{B} \to  g_1 t_{1} t_2 \tilde{B} \cdots \to  g_1  t_{1}\cdots  t_{n} \tilde{B} = g_2\tilde{B}.$$

\subsection{Closed galleries and pointed galleries in $X_\G$} A closed gallery in $X_\G$
starting at the chamber $\G g \tilde{B}$ of $X_\G$ can be lifted to a gallery
in $\B$ starting at $g\tilde{B}$ and ending at $\g g \tilde{B}$ for some $\g \in \G$.
Denote by $\ka_\g(g\tilde{B})$ the homotopy class of geodesic galleries in
$\B$ from $g\tilde{B}$ to $\g g\tilde{B}$. By abuse of notation, it also represents
a homotopy class of closed geodesic galleries in $X_\G$ starting at
$\G g \tilde{B}$. A closed geodesic gallery is {\it tailless} if it remains a geodesic when the starting chamber is changed.

Recall that a pointed chamber $g_2BZ=(g_2KZ, g_2\sigma KZ, g_2\sigma^2 KZ)$ is an out-neighbor of $g_1BZ=( g_1KZ, g_1\sigma KZ, g_1\sigma^2 KZ)$ if and only if
$ g_1\sigma KZ=g_2KZ, g_1 \sigma^2 KZ = g_2 \sigma KZ$ and $ g_1 KZ \neq g_2 \sigma^2 KZ$, or equivalently
$ g_1^{-1} g_2 \in L_B Z.$

A sequence $\G g_0 BZ \to \G g_1 BZ  \to \cdots \to \G g_n BZ = \G g_0BZ$ of pointed chambers in $X_\G$ is called a closed pointed gallery of length $n$ if there is a lifting
$g_0 BZ \to \cdots \to  g_nBZ$ in $\B$ so that $g_{i+1}BZ$ is an out-neighbor of $g_{i}BZ$ for $0 \le i \le n-1$ and $g_n BZ = \g g_0 BZ$ for some $\g$ in $\G$. Denote this pointed gallery by $\kappa_\g (g_0 BZ)$ for short.
Note that there is a pointed gallery from $g_0 BZ$ to $\g g_0 BZ$ in $\B$ of length $n$ if and only if $ g_0^{-1} \g g_0 \in  (L_B)^n$.
In this case, $\G g_0 \tilde B \to \G g_1 \tilde B \to \cdots \to \G g_n \tilde{B}$ is the gallery $\kappa_\g (g_0 \tilde{B})$ and we say the gallery $\kappa_\g (g_0 \tilde{B})$ admits the pointed gallery $\ka_\g(g_0BZ)$.
Note that if a gallery admits a pointed gallery, then this pointed gallery is unique.

Analogous to the case of 1-geodesics, we have several descriptions of tailless galleries:
\begin{proposition}  \label{gallerytaillesscriterion}
For a type 1 closed geodesic gallery $\ka_\g(g\tilde{B})$, the following are equivalent:

1. $\ka_\g(g\tilde{B})$ is tailless.

2. $\ka_{\g}(g\tilde{B})$ repeated $m$-times is a type 1 geodesic gallery for all $m>0.$

3.  $\ka_\g(g\tilde{B})$ admits a closed pointed gallery $\ka_\g(g_0BZ)$ for a unique $g_0BZ$ such that $g_0\tilde B = g\tilde B$.
\end{proposition}

Consequently, the map sending $\ka_\g(g\tilde B)$ to $\ka_\g(g_0BZ)$ is a length preserving bijection from the set of tailless type $1$ closed geodesic galleries to the set of closed pointed galleries in $X_\G$.

\begin{proof}
(1 $\Rightarrow$ 2)
Suppose $\ka_\g(g\tilde{B})$ is tailless. Let $g_0\tilde{B} \to \cdots \to g_{mn}\tilde{B}$ be a lifting of $\ka_{\g}(g\tilde{B})$ repeated $m$-times in $\B$.
Then there is a word $w= t_{i_1} \cdots t_{i_{mn}}$ so that
$g_j^{-1} g_k \in \tilde{B} t_{i_{j+1}}t_{i_{j+2}}  \cdots t_{i_k} \tilde{B}$
 for all $0\leq j < k \leq mn$. By the tailless assumption,
 $g_j\tilde{B} \to \cdots \to g_{j+n}\tilde{B}$ is geodesic of type 1 for $j = 0,..., n(m-1)$, so we have $t_{i_{(j+1)}}=t_{i_{j}+1}$ for $j = 0,..., mn-1$.
This shows that $w$ is a reduced word and $\ka_{\g}(g\tilde{B})$ repeated $m$-times is a type 1 geodesic gallery.
\\
(2 $\Rightarrow$ 3) Let $g\tilde{B}=g_0\tilde{B} \to \cdots \to g_{2n}\tilde{B}$ be a lifting of $\ka_{\g}(g\tilde{B})$ repeated twice.
Since it is a type 1 geodesic gallery, as noted before, we may assume that
$g_i \tilde{B} = g t_1 \cdots t_i \tilde{B}$ and $g_i= g t_1\cdots t_i \sigma^i$ for $i= 0,..., 2n$.
Then $g_i BZ$ is a pointed chamber of $g_i\tilde{B}$ and
$g_{i}^{-1} g_{i+1}= \sigma^{-i}t_{i+1}\sigma^{i+1}= t_1\sigma \in L_B$. Therefore
$g_0BZ \to \cdots \to  g_{2n}BZ$ is a pointed gallery. It remains to show that $\g g_0 BZ= g_n BZ$ so that $g_0BZ  \to \cdots \to g_{n}BZ$ is a lifting of a closed pointed gallery of $X_\G$.
From $\g g_0 \tilde{B}= g_n \tilde{B}$ and $\ord_\pi \det \G \subset 3\Z$ by assumption we conclude
$n \in 3\Z$ and $\g g_0 BZ= g_n \sigma^{i} BZ$ for some $i \in \{0, 1, 2\}$. Comparing determinants of both sides gives $i = 0$ since $\det \sigma = \pi$.
This proves $\g g_0 BZ= g_n BZ$.
\\
(3 $\Rightarrow$ 1)
 Let $g_0 BZ \to \cdots \to g_{2n}BZ$ be a lifting in $\B$ of the pointed gallery admitted by $\ka_{\g}(g\tilde{B})$ repeated twice. Thus $g_i^{-1}g_{i+1} \in L_B Z$ for $i=0,..., 2n-1$. Note that every gallery obtained by changing the starting chamber of $\ka_\g(g\tilde{B})$ has a lifting contained in $C: g_0 \tilde{B} \to \cdots \to g_{2n}\tilde{B}$, so it suffices to show that $C$ is a geodesic gallery.
This is because
$$g_0^{-1} g_{2n}= (g_0^{-1}g_1)\cdots(g_{2n-1}^{-1}g_{2n}) \in (L_B Z)^{2n} \subset \tilde{B}t_1\cdots t_{2n} \tilde{B} $$
and $t_1\cdots t_{2n}$ is a reduced word of length $2n$.
\end{proof}

\subsection{Characterizing closed pointed galleries in $X_\G$}\label{charclosedpointedgalleries}
 We begin by extracting information on the starting vertex and the type of a closed pointed gallery.

\begin{proposition} \label{charpointedgallery}
Let $\g \in \G$. Suppose $\ka_\g(gBZ)$ is a closed pointed gallery in $X_\G$ of length $n$. Then vertices $g \sigma^i KZ$ for $i = 0, 1, 2$ all lie in $\Delta_G([\g])$ defined by (\ref{deltaG}).
Moreover, $[\g]$ is of type $(0,n/2)$ for $n$ even, and $(1,(n-1)/2)$ for $n$ odd. In the latter case, $\g$ is ramified rank-one split.
\end{proposition}
\begin{proof}
Observe that
$$ (t_1 \sigma)^{2m} = \begin{pmatrix} 1 & & \\ & \pi^m &  \\ &  & \pi^m \end{pmatrix} \quad \mbox{and}\quad
(t_1\sigma)^{2m+1}= \begin{pmatrix} 1 & & \\ & & \pi^m \\ & \pi^{m+1} & \end{pmatrix}.$$
Since $\ka_\g(gBZ)$ is a pointed gallery of length $n$, we have
$ g^{-1} \g g \in (L_B)^{n} \subset K(t_1\sigma)^n KZ$.
 To find the type of $\kappa_\g(g\sigma^i KZ)$,
write $\sigma =\left(\begin{smallmatrix} 1 & & \\ & 1 & \\ & & \pi \end{smallmatrix}\right)\left(\begin{smallmatrix}  & 1 & \\ &  & 1\\ 1 & &  \end{smallmatrix}\right).$
Using the fact $\sigma B = B\sigma$, we have
$$ \sigma^{-1}g^{-1} \g g\sigma \in \sigma^{-1} (L_B)^n  \sigma =
B \sigma^{-1} (t_1\sigma)^n \sigma BZ \subset K \left(\begin{smallmatrix} 1 & & \\ & 1 & \\ & & \pi^{-1} \end{smallmatrix}\right)  (t_1\sigma)^n \left(\begin{smallmatrix} 1 & & \\ & 1 & \\ & & \pi \end{smallmatrix}\right) KZ = K (t_1\sigma)^n KZ.$$
By the same argument, we also have $\sigma g^{-1}  \g g\sigma^{-1} \in K(t_1 \sigma)^n KZ.$
Therefore, if $n=2m$, $\ka_\g(g\sigma^i KZ)$ has type $(0,m)$; if $n=2m+1$,  $\ka_\g(g \sigma^i KZ)$ has type $(1,m)$ for all $i$.
It remains to show that $[\g]$ has the same type as $\ka_\g(g\sigma^i KZ)$, so that they have the same geometric length.

Since $ g^{-1} \g^{2k} g \subset K(t_1 \sigma)^{2nk} K Z= T_{0,nk}$, we have, by Proposition \ref{canonicallength},
$$ L_A(\gamma) = L_A(g^{-1} \gamma g )= \lim_{k \to \infty} \frac{1}{k} l_A(g^{-1} \gamma^k g) = \lim_{2k \to \infty} \frac{2nk}{2k}=n.$$
The same argument gives $L_A(\g^{-1})=\frac{n}{2}$.

When $n=2m$, we have $L_A(\g)=l_A(\ka_\g(gKZ))=2m$ and  $L_A(\g^{-1})=l_A(\ka_{\g^{-1}}(gKZ))=m$. By Corollary \ref{canonicallengthandtype}, $[\g]$ and $\ka_\g(gKZ)$
have the same type, which is $(0,m)$.

When $n=2m+1$, we have $L_A(\g^{-1})=\frac{2m+1}{2}$, which implies that $r_\g^{-1}$ is ramified rank-one split (and so is $r_\g$).
Now suppose $[\g]$ is of type $(i,j)$. Applying Equation (\ref{lengthcomparision}) to $g^{-1} \g g$ and $g^{-1} \g^{-1} g$, we obtain
$$ i+2j = 2m+1 \qquad \mbox{and} \qquad m+\frac{1}{2} \leq 2i+j \leq m+2.$$
It is easy to see that $(i,j)=(1,m)$ is the only non-negative integral solution.
\end{proof}

We now show that the conditions in the above proposition characterize closed pointed galleries.

\begin{proposition}
Suppose $\g \in \G$ satisfies either (1) $[\g]$ is of type $(0,n)$, or (2) $\g$ is ramified rank-one split and $[\g]$ is of type $(1,n)$.
If the three vertices $g\sigma^i KZ$ of the chamber $g\tilde B$ all lie in $\Delta_G([\g])$, then there is a unique $i \in \{0,1,2\}$ such that $\kappa_\g(g\sigma^iBZ)$ is a closed pointed gallery.
\end{proposition}
\begin{proof}

The uniqueness of $i$ follows from the third statement in Proposition \ref{gallerytaillesscriterion}; we shall show it exists. Denote by $g \A$ the apartment containing the two chambers $g \tilde{B}$ and $\g g \tilde{B}$. {Replacing $g$ by $gb$ for some $b \in \tilde B$ if necessary,} %Up to translation by an element in $B$,
we may assume that $\A$ is the standard apartment whose pointed chambers are represented by $DS_3 \tilde{B}$, where $D$ is the group of
diagonal matrices in $G$ and $S_3$ is the subgroup of permutation
matrices in $G$. Write $g^{-1} \g g = M s b$ for some $M \in D$,
$s \in S_3$ and $b \in BZ$.
Since the vertices of $gBZ$ are in $\Delta_G([\g])$, by Proposition \ref{geometricalminimalequaltothesametype},
$\ka_{\g}(g\sigma^iKZ)$ has the same type as $[\g]$ for all $i$.

Case (I). $[\g]$ has type $(0,n)$. Then
$g^{-1}\g g$, $\sigma^{-1}g^{-1}\g g \sigma$ and $\sigma g^{-1}\g g
\sigma^{-1}$ all lie in $T_{0,n}$. In this case,
$$M \in \left\{ \left(\begin{smallmatrix} \pi^n & & \\ & \pi^n & \\ & & 1 \end{smallmatrix}\right),
\left(\begin{smallmatrix} \pi^n & & \\ & 1 & \\ & & \pi^n \end{smallmatrix}\right),
\left(\begin{smallmatrix} 1 & & \\ & \pi^n & \\ & & \pi^n \end{smallmatrix}\right)\right\}
=\left\{
\sigma^{i} \left(\begin{smallmatrix} 1 & & \\ & \pi^n & \\ & & \pi^n \end{smallmatrix}\right) \sigma^{-i} ~:~i=0,1,2 \right\}. $$

\noindent In other words, $M= \sigma^{i} \left(\begin{smallmatrix} 1 & & \\ & \pi^n & \\ & & \pi^n \end{smallmatrix}\right) \sigma^{-i} = \sigma^{i}(t_1\sigma)^{2n}\sigma^{-i}$ for some $i$. We shall show that $s$ is the identity matrix. If so,  then, since $B\sigma = \sigma B$, we have $b\sigma^i = \sigma^i b'$ with $b' \in B$ and
$$ (g \sigma^i)^{-1} \g (g \sigma^i) = \sigma^{-i}Mb \sigma^{i} = \sigma^{-i}M \sigma^{i} b' \in (L_B)^{2n}.$$
Thus $\kappa_\g(g\sigma^iBZ)$ is a closed pointed gallery.

It suffices to consider the case $M = \left(\begin{smallmatrix} 1 & & \\ & \pi^n & \\ & & \pi^n \end{smallmatrix}\right)$ as the other cases are similar.
To determine $s$, write $\sigma = \left(\begin{smallmatrix} 1 & & \\ & 1 & \\ & & \pi \end{smallmatrix}\right) s_3$ with $s_3 \in S_3$.
Observe that
\begin{eqnarray*}
\sigma^{-1}g^{-1}\g g \sigma &=& \sigma^{-1} Msb \sigma = \sigma^{-1} Ms \sigma b''\\
&=& s_3^{-1} \left(\begin{smallmatrix} 1 & & \\ & 1 & \\ & & \pi^{-1} \end{smallmatrix}\right)
\left(\begin{smallmatrix} 1 & & \\ & \pi^n & \\ & & \pi^n \end{smallmatrix}\right) s \left(\begin{smallmatrix} 1 & & \\ & 1 & \\ & & \pi \end{smallmatrix}\right) s_3 b''
\end{eqnarray*} and $s \left(\begin{smallmatrix} 1 & & \\ & 1 & \\ & & \pi \end{smallmatrix}\right)$ is $ \left(\begin{smallmatrix} \pi & & \\ & 1 & \\ & & 1 \end{smallmatrix}\right)s$,
$ \left(\begin{smallmatrix} 1 & & \\ & \pi & \\ & & 1 \end{smallmatrix}\right) s$,
or $ \left(\begin{smallmatrix} 1 & & \\ & 1 & \\ & & \pi \end{smallmatrix}\right) s$ according as the
first, second, or third row of $s$ is $(0 ~0 ~1)$. In order that
$\sigma^{-1}g^{-1}\g g \sigma \in T_{0,n}$, the third row of $s$ must be $(0 ~0 ~1)$. Similarly, $\sigma g^{-1}\g g \sigma^{-1} \in T_{0,n}$ implies the first
row of $s$ should be $(1 ~0 ~0)$. Therefore $s$ is the identity matrix.

Case (II) $\g$ is ramified rank-one split and $[\g]$ has type $(1,n)$. Then
$g^{-1}\g g$, $\sigma^{-1}g^{-1}\g g \sigma$ and $\sigma g^{-1}\g g
\sigma^{-1}$ all lie in $T_{1,n}$. In this case,
$$M \in \left\{ \sigma^{i} \left(\begin{smallmatrix} 1 & & \\ & \pi^n & \\ & & \pi^{n+1} \end{smallmatrix}\right) \sigma^{-i},
\sigma^{i} \left(\begin{smallmatrix} 1 & & \\ & \pi^{n+1} & \\ & & \pi^n \end{smallmatrix}\right) \sigma^{-i} ~:~ i=0,1,2 \right\}. $$

\noindent As before, it suffices to consider the cases $M = \left(\begin{smallmatrix} 1 & & \\ & \pi^n & \\ & & \pi^{n+1} \end{smallmatrix}\right)$ or $\left(\begin{smallmatrix} 1 & & \\ & \pi^{n+1} & \\ & & \pi^n \end{smallmatrix}\right)$.

A similar argument as in case (I) yields
$$ Ms = \left(\begin{smallmatrix} 1 & & \\ & \pi^n & \\ & & \pi^{n+1} \end{smallmatrix}\right), \left(\begin{smallmatrix} 1 & & \\ & & \pi^n \\ & \pi^{n+1} & \end{smallmatrix}\right), \mbox{or } \left(\begin{smallmatrix} 1 & & \\ & \pi^{n+1} & \\ & & \pi^{n} \end{smallmatrix}\right).$$
Observe that $[\g^2]$ has type $(0,2n+1)$. Since $\Delta_G([\g]) \subset \Delta_G([\g^2])$
by Proposition \ref{Deltag2}, we have $(g^{-1} \g g)^2=(Msb)^2 \in T_{0,2n+1}$.
On the other hand, if $Ms = \left(\begin{smallmatrix} 1 & & \\ & \pi^n & \\ & & \pi^{n+1} \end{smallmatrix}\right)$ or $\left(\begin{smallmatrix} 1 & & \\ & \pi^{n+1} & \\ & & \pi^{n} \end{smallmatrix}\right)$, then a direct computation shows
$(Msb)^2 \in T_{2,2n}$, which is a contradiction. Therefore
$g^{-1} \g g = Msb= \left(\begin{smallmatrix} 1 & & \\ & & \pi^n \\ & \pi^{n+1} & \end{smallmatrix}\right)b \in L_B^{2n+1}$.
Thus $\kappa_\g(g\sigma^iBZ)$ is a closed pointed gallery for some $i$.
\end{proof}

The above two propositions and Proposition \ref{gallerytaillesscriterion} together imply
\begin{theorem} \label{numberoftaillessgallery}
Given $\g \in \G$, the number of closed pointed galleries in $X_\G$ of the form $\ka_\g(gBZ)$ is equal to the number of chambers with vertices $P_\g gKZ$,
where $gKZ \in C_{P_\g^{-1} \G P_\g}(r_\g)\backslash \Delta_G([\g])$.
\end{theorem}

\section{Chamber zeta function of $X_\G$}
\subsection{Type $1$ chamber zeta function of $X_\G$}
Two closed galleries in $X_\G$ are called equivalent if one is obtained from the other
by changing the starting chamber.
A closed gallery is called {\it primitive}
if it is not a repetition of another closed gallery of shorter
length. For a primitive tailless closed
gallery $C$ of length $n$, denote by $[C]$ the collection of the $n$ closed
galleries equivalent to $C$.

The type $1$ chamber zeta function of $X_\G$ is defined as an Euler
product:
\begin{eqnarray}\label{defineZ2}
Z_{2,1}(X_\G, u) = \prod_{[C]} (1 -
u^{l(C)})^{-1},
\end{eqnarray}
where $[C]$ runs through the equivalence classes of
primitive, tailless, type $1$ closed galleries in $X_\G$. Let $M_n(X_\G)$ denote the number of tailless, type $1$ closed galleries in $X_\G$ of length $n$.

\begin{proposition}\label{rationalZ2} The type $1$ chamber zeta
function of $X_\G$ is a rational function with the following expressions:
\begin{eqnarray}\label{L2asdetLB}
Z_{2,1}(X_\G, u) = \exp(\sum_{n \ge 1} \frac{M_n(X_\G)}{n} u^n) = \frac{1}{\det (I - L_B u)}.
\end{eqnarray}
\end{proposition}

\begin{proof}
For $n \ge 1$, $\Tr L_B^{n}$ on $X_\G$ counts the number of closed pointed galleries of length $n$, which is equal to $M_n(X_\G)$ by Proposition \ref{gallerytaillesscriterion}.  The equalities follow from the same argument as the proof of Proposition \ref{Z1andLE}.
\end{proof}

\subsection{Comparing chamber zeta function and edge zeta function}

In this subsection we give an explicit formula for $M_n(X_\G)$, the number of closed pointed galleries in $X_\G$ of length $n$, similar to Theorem \ref{logedgezeta}. This is achieved by computing the difference between logarithmic derivatives of edge zeta function and chamber zeta function and applying Theorem \ref{logedgezeta}.

\begin{theorem}\label{compareZ2andZ1}
\begin{eqnarray*}
& &  u\frac{d}{du}\log Z_{1,2}(X_\G, u) - u\frac{d}{du}\log Z_{2,1}(X_\G, -u) \\
&=&  \sum_{n \ge 1}\big(\sum_{\substack{ \g \in [\G] ~{\rm irregular},\\ ~[\g] {\rm ~of ~type} ~(0, ~n)}}  -(q-1)\vol([\g]) u^{l_A([\g])}
+ \sum_{ \substack{\g \in [\G] ~{\rm unramified ~rank-one
~split ,}\\ ~[\g] {\rm~of ~type} ~(0, ~n)}} 2\vol([\g]) u^{l_A([\g])} \\
&~&~~~~~~~~~ + \sum_{\substack{\g \in [\G] ~{\rm ramified ~rank-one ~split,}\\  ~[\g] {\rm~of ~type} ~(0, ~n) ~{\rm or} ~(1,n)}}
\vol([\g])
u^{l_A([\g])}\big),
\end{eqnarray*}
where $\rm{vol}([\g])$ is defined by (\ref{fundamentaldomain'}).
\end{theorem}

\begin{proof}  Combining Theorem \ref{numberoftaillessgallery}
and Proposition \ref{rationalZ2}, we have
\begin{eqnarray*}
& & u\frac{d}{du}\log Z_{2,1}(X_\G, -u) \\
&=& \sum_{n \ge 1} ~\big(\sum_{\g \in [\G], ~[\g] ~{\rm of ~type} ~(0, ~n)} N_B(\g) u^{2n} - \sum_{ \substack{ \g \in [\G] ~{\rm ramified ~rank-one ~split,}
  \\~[\g] {\rm ~of ~type} ~(1, ~n)}} N_B(\g)u^{2n+1}\big),
\end{eqnarray*}
where $N_B(\g)$ is the number of chambers with vertices $P_\g gKZ$,
where $gKZ \in C_{P_\g^{-1} \G P_\g}(r_\g)\backslash \Delta_G([\g])$. From the definition of $\Delta_G([\g])$, it is clear that $N_B(\g)=N_B(\g^{-1})$ so that
\begin{eqnarray*}
& & u\frac{d}{du}\log Z_{2,1}(X_\G, -u)\\
&=& \sum_{n \ge 1} ~\big(\sum_{\g \in [\G], ~[\g] ~{\rm
 of ~type} ~(n, ~0)} N_B(\g) u^{2n} - \sum_{ \substack{ \g \in [\G] ~{\rm ramified ~rank-one~split, }\\
 ~[\g] {\rm  ~of ~type} ~(n, ~1)}} N_B(\g)u^{2n+1)}\big).
\end{eqnarray*}

 On the other hand, for type $2$ edge zeta function we have
\begin{eqnarray*}
u\frac{d}{du}\log Z_{1,2}(X_\G, u) &=& u\frac{d}{du}\log Z_{1,1}(X_\G, u^2) = \sum_{\g \in [\G]}
~\sum_{\ka_\g(gKZ) \text{~tailless, type $1$}} 2u^{2l_A(\ka_\g(gKZ))}\\
&=& \sum_{n \ge 1} ~\sum_{\g \in [\G], ~[\g] ~{\rm of ~type} ~(n,
~0)} 2N_K(\g) u^{2l_A([\g])},
\end{eqnarray*}
where $N_K(\g) = \vol([\g]) \omega_{[\g]}$ is the number of tailless type $1$ cycles in $[\g]$ (cf. Theorem \ref{logedgezeta}).
We shall compare this with the number $N_B(\g)$.
Recall that for $[\g]$ of type $1$, we have $\Delta_G([\g])=
\Delta_A([\g])$.

Case I. $\g$ split with $[\g]$ of type $(n, 0)$.  Then $r_\g = \diag(1, a,
b)$, where $1, a, b$ are distinct with $\ord_\pi$ $ (a) = 0$ and
$\ord_\pi$ $ b = n$. Let $\delta = \ord_\pi$ $ (1 - a)$. %{ for $\g$ split (hence $a \ne 1$), and $\delta = 0$ for $\g$ irregular (hence $a = 1$). }
The centralizer $C_G(r_\g)$ consists of diagonal elements in $G$.
By Corollary
\ref{primitivesplit}, $C_{P_\g^{-1} \G P_\g}(r_\g) \backslash
\Delta_A([\g])$ has cardinality $N_K(\g) = \vol([\g])q^{\delta}$,
represented by vertices $h_{i,j}v_xKZ$, where $h_{i,j} = \diag(1, \pi^i, \pi^j) \in C_{P_\g^{-1} \G P_\g}(r_\g) \backslash
C_G(r_\g)/(C_G(r_\g) \cap KZ)$ and
$v_x = \left(\begin{matrix} 1 & x &  \\
& 1 &  \\ & & 1\end{matrix}\right)$ with $x \in
\pi^{-\delta}\oo/\oo$.

There are $q+1$ chambers sharing the type $1$ edge $E_0:= (KZ,
\diag(1, 1, \pi)KZ)$ with the third vertex $u_c KZ :=
\left(\begin{matrix}\pi & c & \\ & 1 & \\ & & \pi
\end{matrix}\right)KZ$, $c \in \oo/\pi\oo$, and $u_\infty KZ:=
\left(\begin{matrix}1 &  & \\ & \pi & \\ & & \pi
\end{matrix}\right)KZ$. Left multiplication by $h_{i,j}v_x$ sends $E_0$ to $(h_{i, j}v_x KZ,  h_{i, j+1}v_xKZ)$ and the
third vertex to $h_{i, j}v_x u_c KZ = \left(\begin{matrix}1 & (c + x)/\pi & \\
& \pi^{i-1} & \\ & & \pi^j
\end{matrix}\right)KZ$ and $h_{i, j}v_x u_\infty KZ =
\left(\begin{matrix}1 & x\pi & \\ & \pi^{i+1} & \\ & & \pi^{j+1}
\end{matrix}\right)KZ$, respectively.
We count the number of such vertices belonging to $C_{P_\g^{-1} \G
P_\g}(r_\g) \backslash \Delta_A([\g])$.

There is only one integral $x$, namely, $x = 0$. When $\delta = 0$,
each type $1$ edge $(h_{i, j}v_0KZ,  h_{i, j+1}v_0 KZ)$ can be extended to
a pointed chamber by adding only one of the two vertices $h_{i+1, j+1}v_0 KZ$ and $h_{i-1, j}
v_0 KZ$ in $C_{P_\g^{-1} \G P_\g}(r_\g) \backslash \Delta_A([\g])$. Once the starting pointed chamber $gBZ$ is chosen,
the closed pointed gallery $\kappa_\g(gBZ)$ is determined.
Hence
$N_B(\g) = 2\#(C_{P_\g^{-1} \G P_\g}(r_\g) \backslash
\Delta_A([\g])) = 2N_K(\g).$

Next assume $\delta \ge 1$. In this case, each type $1$ edge $(h_{i,
j}v_0 KZ, h_{i, j+1}v_0 KZ)$ can be extended to a pointed chamber by adding one of the $q+1$
vertices $h_{i, j}v_0 u_cKZ$ and $h_{i, j}v_0 u_\infty KZ$ in
$C_{P_\g^{-1} \G P_\g}(r_\g) \backslash \Delta_A([\g])$. The same
holds when $h_{i,j}v_0$ is replaced by $h_{i,j}v_x$ for $-1 \ge
\ord_\pi$$ x \ge -\delta + 1$. This gives rise to $(q+1)(q^{\delta -
1} - 1)$ pointed chambers. Finally, when $\ord_\pi$$ x = - \delta$, each type $1$ edge $(h_{i, j}v_x KZ, h_{i, j+1}v_x KZ)$
can be extended to a pointed chamber by adding
 only one vertex $h_{i, j}v_x u_\infty KZ$ in $C_{P_\g^{-1} \G
P_\g}(r_\g) \backslash \Delta_A([\g])$, so there are $(q-1)q^{\delta
- 1}$ pointed chambers. Put together, we get $N_B(\g) = \vol([\g]) \big(q+1 +
(q+1)(q^{\delta - 1} - 1) + (q-1)q^{\delta - 1}\big) =
\vol([\g])2q^\delta = 2N_K(\g).$

Hence there is no contribution  to
$u\frac{d}{du}\log Z_{1,2}(X_\G, u) - u\frac{d}{du}\log Z_{2,1}(X_\G,
-u)$ from $\g$ split and $[\g]$ of type $(n,0)$.
\smallskip

Case II. $\g$ irregular with $[\g]$ of type $(n, 0)$.
Then $r_\g = \diag(1, 1, b)$, where $\ord_\pi$ $b= n$. By Corollary
\ref{primitivesplit}, $C_{P_\g^{-1} \G P_\g}(r_\g) \backslash
\Delta_A([\g])$ has cardinality $N_K(\g) = \vol([\g])$. By the same method
as in Case I, one checks that
all $q+1$ chambers sharing an edge with two vertices in $\Delta_A([\g])$
have the third vertex also lie in $\Delta_A([\g])$.
 Hence
$N_B(\g)=(q+1)N_K(\g)$ and
the contribution of an irregular
$\g$ with $[\g]$ of type $(n,0)$ to $u\frac{d}{du}\log Z_{1,2}(X_\G, u)
-u\frac{d}{du}\log Z_{2,1}(X_\G, -u)$ is $-(q-1)\vol([\g])u^{2n}$.
\smallskip

%. Then $C_G(r_\g) = GL_2(F)\times F^\times$ and $v_x$ is the identity matrix. The representatives $h$ have the form $h_{i,j,y} : = \left(\begin{matrix} \pi^i & y & \\ & \pi^j & \\ & & 1 \end{matrix}\right)$ with $y \in F$. Left translation by $h_{i,j,y}$ sends $E_0$ to the edge $h_{i,j,y}E_0:= (h_{i,j,y}KZ, h_{i-1,j-1,y/\pi}KZ)$, and the vertices $u_cKZ$ (resp. $u_\infty KZ$) to $h_{i,j-1,\pi^{i-1}c+y\pi^{-1}}u_c KZ$
%(resp. $h_{i-1,j,y}KZ$) in $\Delta_A([\g])$. Hence each type $1$ edge $h_{i,j,y}E_0$ can be extended to $q+1$ pointed chambers by adding one of the $q+1$ vertices $h_{i,j,y}u_cKZ$ or $h_{i,j,y} u_\infty KZ$. This shows that for each irregular $\g$ of type $(n, 0)$, we have $N_B(\g) - 2 N_K(\g) = (q-1)N_K(\g) = (q-1)\text{vol}([\g])$. This gives the first sum of the right hand side of the desired equality.

Case III. $\g$ unramified rank-one split with $[\g]$ of type $(n, 0)$. In
this case $r_\g = \left(
\begin{smallmatrix} a &  &  \\ & e & dc \\ & d& e+db \end{smallmatrix}\right)$,
and the eigenvalues $a$, $e + d \lambda$ and $e + d \bar {\lambda}$
of $r_\g$  generate an unramified quadratic extension $L$ over $F$.
The type assumption on $\g$ implies that $\ord_\pi$$ a = n$ and
$\min(\ord_\pi$$ e, \ord_\pi$$ d) = 0$ so that $e + d \lambda$ and $e +
d \bar {\lambda}$ are units in $L$. Let $\delta = \ord_\pi$$ d$.

As discussed in \S \ref{centralizers},
the double cosets $C_{P_\g^{-1} \G P_\g}(r_\g)\backslash
C_G(r_\g)/C_G(r_\g)\cap KZ$ are represented by $h_m = \diag(\pi^m, 1,
1)$, $m \mod \vol([\g])$. By Proposition
\ref{rankonenumberofalgtailless}, $C_{P_\g^{-1} \G P_\g}(r_\g)
\backslash \Delta_A([\g])$ has cardinality $N_K(\g) =
\vol([\g])\frac{q^\delta + q^{\delta - 1}-2}{q-1}$ and is represented
by $h_mg_{i,j,u}KZ$ and $h_m g_{i,z}KZ$, where $m \mod \vol([\g])$,
$g_{i,j,u} = \left(
\begin{smallmatrix}
1 &  & \\
& \pi^{i-j} & u\\
& & \pi^j
\end{smallmatrix} \right)$ with $0 \le j \le i \le \delta$, $u \in \oo^\times/\pi^{i-j}\oo$
for $j<i$ and $u = 0$ for $j=i$, and $ \quad g_{i,z} = \left(
\begin{smallmatrix}
1 &  & \\
& \pi^{i} & z\\
& & 1
\end{smallmatrix} \right)$ with $1 \le i \le \delta$ and $z \in \pi
\oo/\pi^i \oo$.% Let $g = h_mg_{i,j,u}$ or $h_m g_{i,z}$. Then, by
%Proposition \ref{rankonenumberofalgtailless}, the type $1$ tailless
%closed geodesic $\kappa_\g(P_\g gKZ)$ is given by $P_{\g}gKZ
%\rightarrow P_{\g}g \diag(\pi, 1, 1)KZ \rightarrow \cdots \rightarrow
%P_{\g}g \diag(\pi^{n}, 1, 1)KZ = \g P_{\g}gKZ$.

It remains to count the number of pointed chambers with vertices in
$C_{P_\g^{-1} \G P_\g}(r_\g) \backslash \Delta_A([\g])$ containing a
type $1$ edge $(gKZ,  g\diag(\pi, 1, 1)KZ)$ for $g
=h_mg_{i,j,u}$ or $h_mg_{i,z}$. When $\delta = 0$, there are no
$g_{i,z}$ and only one $g_{i, j, u}$, equal to the identity matrix,
hence the vertices in $C_{P_\g^{-1} \G P_\g}(r_\g) \backslash
\Delta_A([\g])$ are $h_mKZ$, $m \mod \vol([\g])$.  It is clear that
there are no pointed chambers formed by these vertices. Hence $N_K(\g) =
\vol([\g])$ and $N_B(\g) = 0$ when $\delta = 0$.

Next assume $\delta \ge 1$. There are $q+1$ chambers sharing
the type $1$ edge $E_1:= (KZ, \diag(\pi, 1, 1)KZ)$ with the third
vertex
being $w_x KZ:= \left(\begin{smallmatrix}\pi & & \\
& \pi & x \\ & & 1
\end{smallmatrix}\right)KZ$ with $x \in \oo/\pi\oo$ and $w_\infty KZ:= \diag(1, \pi^{-1}, 1)KZ$,
respectively. Left multiplication by $g = h_mg_{i,j,u}$ or $
h_mg_{i,z}$ sends the edge $E_1$ to the
type $1$ edge $(gKZ,  g\diag(\pi, 1, 1)KZ)$, so we need to
count the number of distinct vertices among $gw_xKZ$ and $gw_\infty
KZ$ which fall in $C_{P_\g^{-1} \G P_\g}(r_\g) \backslash
\Delta_A([\g])$. Observe that
$$h_mg_{i,j,u}w_x KZ = \left(\begin{matrix} \pi^{m+1} & & \\ &
\pi^{i-j+1} & x \pi^{i-j} + u \\ & & \pi^j \end{matrix}\right)KZ,
\qquad  h_mg_{i,j,u}w_\infty KZ= \left(\begin{matrix} \pi^{m} & & \\
& \pi^{i-j-1} &  u \\ & & \pi^j \end{matrix}\right)KZ,$$
$$h_mg_{i,z}w_x KZ= \left(\begin{matrix} \pi^{m+1} & & \\ &
\pi^{i+1} & x \pi^{i} + z \\ & & 1 \end{matrix}\right)KZ, \qquad
\text{and} \qquad h_mg_{i,z}w_\infty KZ =\left(\begin{matrix} \pi^{m} & & \\
& \pi^{i-1} & z \\ & & 1 \end{matrix}\right)KZ.$$

\noindent It is straight forward to check that, for $0 \le i \le \delta - 1$,
all $gw_x KZ$ and $gw_\infty KZ$ are distinct vertices in
$C_{P_\g^{-1} \G P_\g}(r_\g) \backslash \Delta_A([\g])$, thus they give rise to
are $\vol([\g])(q+1)\frac{q^\delta + q^{\delta - 1}-2}{q-1}$ pointed chambers.
When $i = \delta$, for each $g$ above, only $gw_\infty KZ$ lies in
$C_{P_\g^{-1} \G P_\g}(r_\g) \backslash \Delta_A([\g])$, hence they yield
 $\vol([\g])(q^\delta + q^{\delta - 1})$ pointed chambers. Altogether,
$N_B(\g)$ is equal to $2N_K(\g) - 2\vol([\g])$ for $\delta \ge 0$.

In conclusion, the contribution to $u\frac{d}{du}\log Z_{1,2}(X_\G, u)
-u\frac{d}{du}\log Z_{2,1}(X_\G, -u)$ from $\g$ unramified rank-one split with
$[\g]$ of type $(n,0)$  is $2\vol([\g])u^{2n}$.
\smallskip

Case IV. $\g$ ramified rank-one split with $[\g]$ of type $(n, 0)$. Then
$r_\g = \left(
\begin{smallmatrix} a &  &  \\ & e & dc \\ & d& e+db \end{smallmatrix}\right)$
 and the eigenvalues $a$, $e + d \lambda$ and $e + d \bar {\lambda}$
of $\g$  generate a ramified quadratic extension $L$ over $F$. In
this case, $\ord_\pi$$ a = n$ and $\ord_\pi$$ e = 0$ so that $e + d
\lambda$ and $e + d \bar {\lambda}$ are units in $L$. Let $\delta =
\ord_\pi$$ d$.

As discussed in \S \ref{centralizers}, $C_{P_\g^{-1} \G P_\g}(r_\g) \backslash
C_G(r_\g)/C_G(r_\g)\cap KZ$ has cardinality $\vol([\g])$, and it is
represented by $h= \diag(\pi^m, 1, 1)$ with $0 \le m \le
(\vol([\g])-1)/2$ and $ \diag(\pi^m, 1, 1)\pi_L$ with $0 \le m \le
(\vol([\g])-3)/2$ if $\vol([\g])$ is odd, and by $h = \diag(\pi^m, 1, 1)$
and $\diag(\pi^m, 1, 1)\pi_L$ with $m \mod \vol([\g])/2$ if $\vol([\g])$
is even. Here
$\pi_L = \left(\begin{smallmatrix} 1 & & \\ & & c\\
 & 1 & b \end{smallmatrix} \right)$ is imbedded in $G$.

It follows from
Proposition \ref{rankonenumberofalgtailless} that $C_{P_\g^{-1} \G
P_\g}(r_\g) \backslash \Delta_A(\g)$ is represented by $hg_{i,j,u}KZ$
for $g_{i,j,u}$ as in Case III and $h$ as above, so the total number
of vertices is $\vol([\g])(q^{\delta+1}-1)/(q-1)= N_K(\g)$.
%Now, for any $gKZ = hg_{i,j,u}KZ$ in $\Delta_A([\g])$, the type $1$
%tailless cycle $\kappa_\g(P_\g gKZ)$ is $P_\g gKZ \rightarrow P_\g
%g\diag(\pi, 1, 1)KZ \rightarrow \cdots \rightarrow
% P_\g g\diag(\pi^{3n}, 1, 1)KZ = \g P_\g gKZ$ by Proposition \ref{rankonenumberofalgtailless}.
%
%
 To count the number of pointed chambers we proceed as in Case III by
 counting, for each $g = hg_{i,j,u}$, the number of $gw_x KZ$ and
 $gw_\infty KZ$ which lie in $C_{P_\g^{-1} \G P_\g}(r_\g) \backslash
 \Delta_A(\g)$.

 We first discuss the case $\delta = 0$. Then there is only
 one $g_{0,0,u}$, equal to the identity matrix. All representatives
 are given by $hKZ$.
 Observe that $\diag(\pi^m, 1, 1)\pi_L KZ =
 \left(\begin{smallmatrix}\pi^m & & \\ & \pi & 0 \\ & & 1
 \end{smallmatrix}\right) KZ$. So there is only one vertex
 $gw_0 KZ$ which will form a chamber containing the
 type $1$ edge $(gKZ, g\diag(\pi, 1, 1)KZ)$. Hence the
  number of pointed chambers is $N_B(\g)= \vol([\g]) = 2N_K(\g) -
 \vol([\g])$ for $\delta = 0$.

 Now assume $\delta \ge 1$.  One sees from the explicit computation
 in Case III that for $g = hg_{i,j,u}$,
 all $q+1$ vertices $gw_x KZ$ and
 $gw_\infty KZ$ are distinct vertices in  $C_{P_\g^{-1} \G P_\g}(r_\g) \backslash \Delta_A([\g])$
 provided that $0 \le i \le \delta - 1$; when $i = \delta$, only one
 vertex, $gw_\infty KZ$, lies in $C_{P_\g^{-1} \G P_\g}(r_\g) \backslash \Delta_A(\g)$.
 They give rise to $\vol([\g])\big((q^\delta - 1)(q+1)/(q-1) + q^\delta \big) = \vol([\g])\big(2(q^{\delta + 1}
 - 1)/(q-1) - 1\big)$ pointed chambers. Therefore $N_B(\g) = 2N_K(\g) - \vol([\g])$ for $\delta \ge 1$.

 This shows that the contribution to $u\frac{d}{du}\log Z_{1,2}(X_\G, u)
-u\frac{d}{du}\log Z_{2,1}(X_\G, -u)$ from a ramified rank-one split
$\g$ with $[\g]$ of type $(n,0)$ is $\vol([\g])u^{2n}$.
\smallskip

Case V. $[\g]$ of type $(n, ~1)$. Then it has no contribution to the type $2$ edge zeta function,
and it has contribution to the type $1$ chamber zeta function only when $\g$ is ramified rank-one split.
Then $r_\g$ has eigenvalues $a, e+
d\lambda, e + d \bar{\lambda}$, where $a, e, d \in F$, $\ord_\pi$ $ a =
2n-1$, $\ord_\pi$ $ e \ge 1$ and $\delta = \ord_\pi$ $ d = 0$ by the
analysis above Theorem \ref{rankoneminlength}. %As noted before, such
%$[\g]$ has no contribution to $L_1(X_\G, u^2)$ and the length of a
%type $1$ tailless gallery in $[\g]_B$ is $6n+3$, which is also $l_A([\g^2])$.
Its contribution to
$u\frac{d}{du}\log Z_{2,1}(X_\G, -u)$ is $-N_B(\g) u^{2n+1}$ with
$N_B(\g) = \#C_{P_\g^{-1} \G P_\g}(r_\g) \backslash \Delta_G([\g])$.
Since $\delta = 0$ and $\mu = 0$ by the remark following Theorem
\ref{numberinrankoneclass}, we have $\Delta_G([\g]) =
\Delta_A([\g])$ such that $N_B(\g) = \vol([\g])$ by Corollary
\ref{rankonedoublecosets}.

This completes the proof of the theorem.
\end{proof}

An immediate consequence of the theorem above is a description of the number $M_n(X_\G)$ of closed, type $1$, tailless geodesic galleries of length $n$ in $X_\G$, given below:

\begin{corollary}\label{numberoftaillessgalleries}
(1) If $n=2m+1$ is odd, then
\begin{eqnarray*}
M_n(X_\G) = \sum_{\substack{\g \in [\G] ~{\rm ramified ~rank-one ~split,}\\  ~[\g] {\rm~of ~type} ~(1,m)}}
\vol([\g]);
\end{eqnarray*}

(2) If $n = 2m$ is even, then
\begin{eqnarray*}
M_n(X_\G) &=& \sum_{\substack{\g \in [\G] ~{\rm split,}\\  ~[\g] {\rm~of ~type} ~(0,m)}} \vol([\g])\omega_{[\g]} +
\sum_{\substack{ \g \in [\G] ~{\rm irregular},\\ ~[\g] {\rm ~of ~type} ~(0, m)}} \vol([\g])q\\
&+& \sum_{ \substack{\g \in [\G] ~{\rm unramified ~rank-one
~split ,}\\ ~[\g] {\rm~of ~type} ~(0, m)}} \vol([\g])(\omega_{[\g]} - 2)
+ \sum_{\substack{\g \in [\G] ~{\rm ramified ~rank-one ~split,}\\  ~[\g] {\rm~of ~type} ~(0,m)}}
\vol([\g])(\omega_{[\g]} - 1).
\end{eqnarray*}
Here $\rm{vol}([\g])$ is defined by (\ref{fundamentaldomain'}) and $\omega_{[\g]}$ is as in Theorem \ref{numberinaclass} and Proposition \ref{rankonenumberofalgtailless}.
\end{corollary}

\section{A proof of Theorem C}\label{proofofC}

\subsection{Hecke operators on $X_\G$ and cycle counting}

The action of the Hecke operator $T_{n,m}$ on $L^2(\G \backslash G/KZ)$ is represented by the matrix $B_{n,m}$, whose rows and columns are indexed by vertices of $X_\G$
such that the entry at the row indexed by $\G gKZ$ and column indexed by $\G g'KZ$ records the number of homotopy
classes of 1-geodesic paths of type
$(n, m)$ from $\G gKZ$ to $\G g'KZ$ in $X_\G$. Alternatively, this is the number of $\g \in \G$ such that
the homotopy classes of the 1-geodesics in $\B$ from $gKZ$ to $\g g'KZ$ have
type $(n, m)$. The trace of $B_{n,m}$ then gives the number of
1-geodesic cycles of type $(n,m)$ up to homotopy. In other words,
\begin{eqnarray*}
\Tr(B_{n,m}) &=& \#\bigg\{~\kappa_\g(gKZ) ~|~ \g \in [\G],
~\kappa_\g(gKZ) \in [\g] \text{ has type }(n,m) \bigg\}.
\end{eqnarray*}
To facilitate our computations, form two kinds of formal power
series:
\begin{equation}\label{sumBnm}
 \sum_{\substack{n,m \geq 0 \\ (n,m)\neq (0,0)}}
\Tr(B_{n,m})u^{n+2m} = \sum_{\g \in [\G], ~\g \ne id}
~\sum_{\kappa_\g(gKZ) \in [\g]} ~u^{l_A(\kappa_\g(gKZ))},
\end{equation} and
\begin{equation}\label{sumBn0}
 \sum_{n>0} \Tr(B_{n,0})u^{n}
= \sum_{\g \in [\Gamma], ~\g \ne id} ~\sum_{ \kappa_\g(gKZ) \in [\g]
\text{ has type $1$}}  ~u^{l_A(\kappa_\g(gKZ))}.
\end{equation}

We can relate the left hand side of the zeta identity (\ref{zetaidentity}) to cycle counting:
 \begin{proposition}
\begin{eqnarray}\label{lefthand1}
&& u \frac{d}{du}
\log\frac{(1-u^3)^{\chi(X_\G)}}{\det(I-A_1 u + A_2 q u^2 -q^3u^3 I)}\\
&=&q\left(\sum_{n>0} \Tr(B_{n,0}) u^n\right)-(q-1)\left( \sum_{\substack{n,m \geq 0 \\ (n,m)\neq (0,0)}}
\Tr(B_{n,m})u^{n+2m}\right)\frac{1-q^2u^3}{1-u^3}, \notag
\end{eqnarray}
where the operators act on $L^2(\G \backslash G/KZ)$, $\chi(X_\G)= \frac{(q+1)(q-1)^2}{3}V$ is the Euler characteristic of $X_\G$, and $V$ is the number of vertices in $X_\G$.
\end{proposition}
\begin{proof}
As $B_{n,m}$ is $T_{n,m}$ acting on the space $L^2(\G \backslash G/KZ)$, so
(\ref{recursivehecke}) also holds with $T_{n,m}$ replaced by $B_{n,m}$. In other words,\begin{eqnarray*}
& & u \frac{d}{du}
\Tr \log\frac{(1-u^3)^rI}{(I-A_1 u + A_2 q u^2 -q^3u^3 I)} \\
&=& q\left(\sum_{n>0} \Tr(B_{n,0}) u^n\right)-(q-1)\left( \sum_{\substack{n,m \geq 0 \\ (n,m)\neq (0,0)}}
\Tr(B_{n,m})u^{n+2m}\right)\frac{1-q^2u^3}{1-u^3},
\end{eqnarray*} where $r = \frac{(q+1)(q-1)^2}{3}$.
Recall that each vertex is incident to $q^2+q+1$ type 1 edges and $q^2+q+1$ type 2
edges so that the total number of undirected edges in $X_\G$ is
$\frac{2(q^2+q+1)}{2}V$. Since each edge is contained in $(q+1)$ chambers,
the number of chambers in $X_\G$ is
$\frac{(q+1)}{3}(q^2+q+1)V$.  Therefore the Euler characteristic of
$X_\G$ is
$$ \chi(X_\G) = V - (q^2+q+1) V + \frac{(q+1)}{3}(q^2+q+1)V =
\frac{(q-1)^2(q+1)}{3}V = rV.$$
Using the identity
$$ \log( \det A ) =  \Tr ( \log  A)$$
for a $V\times V$ matrix $A$, we have
$$u \frac{d}{du}
\Tr \log\frac{(1-u^3)^rI}{(I-A_1 u + A_2 q u^2 -q^3u^3 I)}=u \frac{d}{du}
 \log\frac{(1-u^3)^{\chi(X_\G)}}{\det(I-A_1 u + A_2 q u^2 -q^3u^3 I)}, $$
which proves the proposition.
\end{proof}

\subsection{Type $1$ edge zeta function revisited}

Although the type $1$ edge zeta function only concerns type $1$ tailless cycles,
to prove the theorem we shall involve all homotopy cycles.
Denote by $P_{n,m,s}$, $P_{n,m,i}$, $Q_{n,m}$, and $R_{n,m}$ the number of algebraically minimal homotopy cycles of type $(n, m)$ contained in the conjugacy classes
of split, irregular, unramified rank-one split, and ramified rank-one
split $\g$'s, respectively. More precisely,
\begin{equation}\label{pnms}
\begin{aligned}
P_{n,m, s} = ~ \sum_{\substack{\g \in [\G] ~\text{split}\\ ~[\g]~
\text{of type }(n,m) }} ~\#(C_{P_\g^{-1} \G P_\g}(r_\g)\backslash
\Delta_A([\g])) = ~ \sum_{\substack{\g \in [\G] ~\text{split}\\ ~[\g]~
\text{of type }(n,m) }} ~\text{vol}([\g])\omega_{[\g]} ,
\end{aligned}
\end{equation}
and $P_{n,m, i}$, $Q_{n,m}$, and $R_{n,m}$ are similarly defined by changing the type of $\g$ accordingly.

\noindent Recall that an irregular $\g$ has type $1$ or $2$ so that $P_{n,m,i} = 0$ if $nm \ne 0$. Further since $\g$ has type $(n,0)$ if and only if $\g^{-1}$ has type $(0,n)$, we have $P_{n,0,i} = P_{0,n,,i}$. By Theorem \ref{logedgezeta}, the type $1$ edge zeta function can be restated as

\begin{proposition}\label{zetaassum}
$$ u \frac{d}{du} \log Z_{1,1}(X_\G, u) =\sum_{n>0}(P_{n,0,s} + P_{n,0,i} + Q_{n,0} + R_{n,0})u^n.$$
\end{proposition}

\subsection{The number of homotopy cycles of type $(n,m)$ }

In order to gain information on $P_{n,0,s}$, $P_{n,0,i}$, $Q_{n,0}$ and $R_{n,0}$,
we extend the summation to include homotopy cycles of type $(n,m)$.
Recall that the number of such cycles is $\Tr(B_{n,m})$, and cycles
with tails are also included.
Their
relation with the number of algebraically tailless cycles is given below.

\begin{proposition}\label{traceBnm}
With the same notation as in Theorem \ref{numberinrankoneclass}, we
have
\begin{eqnarray*}
\sum_{\substack{n,m \geq 0 \\ (n,m)\neq (0,0)}}
\Tr(B_{n,m})u^{n+2m} = \bigg(\sum_{\substack{n,m \geq 0 \\
(n,m)\neq (0,0)}}
P_{n,m,s}u^{n+2m}\bigg)\frac{1-u^3}{1-q^3 u^3} + \bigg(\sum_{\substack{n,m \geq 0 \\
(n,m)\neq (0,0)}}
P_{n,m,i}u^{n+2m}\bigg)\frac{1-u^3}{1-q^2 u^3}\\
+ \sum_{\substack{\g \in [\G]\\ [\g] ~\text{unram. rank-one
split}}}\vol([\g])u^{l_A([\g])}\bigg(\frac{q^{\de([\g])+1}+q^{\de([\g])}-2}{q-1}+
\frac{(q+1)q^{\de([\g])+2}u^3}{1-q^3 u^3}\bigg)
\bigg(\frac{1-u^3}{1-q^2 u^3}\bigg) \\
 + \sum_{\substack{\g \in [\G]\\ [\g] ~\text{ram. rank-one
split}}}\vol([\g])q^{\mu([\g])}u^{l_A([\g])}\bigg(\frac{q^{\de([\g])
+1} -1}{q-1} + \frac{q^{\de([\g]) +3}u^3}{1-q^3u^3}
\bigg)\frac{1-u^3}{1-q^2u^3}.
\end{eqnarray*}
\end{proposition}
\begin{proof} Break the right side of (\ref{sumBnm}) into four parts, over
split, irregular, unramified rank-one split, and ramified rank-one split
$\g$'s, respectively. Applying Theorem \ref{numberinaclass} to the
split and irregular part, Theorem \ref{numberinrankoneclass} to the unramified
and ramified rank-one split parts, and using the definitions of
$P_{n,m,s}$ and $P_{n,m,i}$, we get the desired formula.
\end{proof}

Next we compute the number of type $1$ homotopy cycles on $X_\G$.

\begin{proposition}\label{alltypezerocycles}
With the same notation as in Theorem \ref{numberinrankoneclass}, we
have
\begin{eqnarray*}
& &\sum_{n>0} \Tr(B_{n,0})u^{n} = (1-q^{-1})\bigg(\sum_{(n, m)
\ne (0,0)} P_{n,m,s}u^{n+2m}\bigg)\frac{1-q^2u^3}{1-q^3u^3}\\
&+& \sum_{\substack{\g \in [\G]\\ [\g] ~\text{unram. rank-one
split}}}\vol([\g])u^{l_A([\g])}\bigg(q^{\de([\g])} + q^{\de([\g]) - 1}+
\frac{(q^2-1)q^{\de([\g]) +1}u^3}{1-q^3u^3} \bigg) \\
&+& \sum_{\substack{\g \in [\G]\\ [\g] ~\text{ram. rank-one
split}}}\vol([\g])u^{l_A([\g])}\bigg(q^{\de}(q^{\mu} - \mu) +
\frac{(q-1)q^{\de + \mu +2}u^3}{1-q^3u^3}\bigg)\\
&+& q^{-1}\sum_{n>0} (P_{n,0,s} + qP_{n,0,i} + Q_{n,0} + R_{n,0})u^{n} -
2q^{-1}\sum_{\substack{\g \in [\G], ~\text{type $1$}\\ [\g]
~\text{unram. rank-one split}}}\vol([\g])u^{l_A([\g])} \\
&+& \sum_{\substack{\g \in [\G], ~\text{type $1$}\\ [\g] ~\text{ram.
rank-one split}}}\vol([\g])u^{l_A([\g])}(-q^{\mu([\g])-1} +
\mu([\g])q^{\de([\g])}).
\end{eqnarray*}
\end{proposition}
\begin{proof} By definition,
\begin{eqnarray*}
\sum_{n>0} \Tr(B_{n,0})u^{n} = \sum_{\g \in [\Gamma]}
~\sum_{\kappa_\g(gK) \in [\g] ~ \text{type $1$}}
u^{l_A(\kappa_\g(gK))}.
 \end{eqnarray*}
 We split the sum over $\g$ into four parts according to $\g$ split, irregular, unramified rank-one split,
  or ramified rank-one split. For the split part, we add (A) and (B) of Theorem
 \ref{type0cyclesinaclass} and use the definition of $P_{n,m,s}$ to arrive
 at the sum $$(1-q^{-1})\bigg(\sum_{(n, m) \ne
(0,0)} P_{n,m,s}u^{n+2m}\bigg)\frac{1-q^2u^3}{1-q^3u^3} ~+~
q^{-1}\bigg(\sum_{n>0} P_{n,0,s}u^{n}\bigg).$$ For the irregular part, Theorem
 \ref{type0cyclesinaclass}, (C) gives the contribution $\sum_{n>0} P_{n,0,i} u^n.$ For the unramified
(resp. ramified) rank-one split part, we add (A2) and (A3) (resp.
(B2) and (B3)) of Theorem \ref{numberinrankoneclass} to get
\begin{equation}\label{rankonepart}
\begin{aligned}
& &\sum_{\substack{\g \in [\G]\\ [\g] ~\text{unram. rank-one
split}}}\vol([\g])u^{l_A([\g])}\bigg(q^{\de([\g])} + q^{\de([\g]) - 1}+
\frac{(q^2-1)q^{\de([\g]) +1}u^3}{1-q^3u^3} \bigg) \\
&+& \sum_{\substack{\g \in [\G], ~\text{type $1$}\\ [\g]
~\text{unram. rank-one
split}}}\vol([\g])u^{l_A([\g])}\frac{q^{\de([\g])} +
q^{\de([\g])-1}-2}{q-1} \\
&+&\sum_{\substack{\g \in [\G]\\ [\g] ~\text{ram. rank-one
split}}}\vol([\g])u^{l_A([\g])}\bigg(q^{\de([\g])}(q^{\mu([\g])} -
\mu([\g])) +
\frac{(q-1)q^{\de([\g]) + \mu([\g]) +2}u^3}{1-q^3u^3}\bigg) \\
&+& \sum_{\substack{\g \in [\G], ~\text{type $1$}\\ [\g] ~\text{ram.
rank-one
split}}}\vol([\g])u^{l_A([\g])}\bigg(q^{\mu([\g])}\frac{q^{\de([\g])}-1}{q-1}
+ \mu([\g])q^{\de([\g])}\bigg).
\end{aligned}
\end{equation}
It follows from Proposition \ref{rankonenumberofalgtailless} and the
definitions of $Q_{n,0}$ and $R_{n,0}$ that
\begin{equation}\label{unramodds}
\begin{aligned}
& &\sum_{\substack{\g \in [\G], \text{type $1$}\\ [\g] ~\text{unram.
rank-one split}}}\vol([\g])u^{l_A([\g])}\frac{q^{\de([\g])} +
q^{\de([\g])-1}-2}{q-1} \\
&=& q^{-1}\sum_{n> 0} Q_{n,0}u^n - 2q^{-1}\sum_{\substack{\g \in
[\G], ~\text{type $1$}\\ [\g] ~\text{unram. rank-one
split}}}\vol([\g])u^{l_A([\g])}
\end{aligned}
\end{equation} and
\begin{equation}\label{ramodds}
\begin{aligned}
& &\sum_{\substack{\g \in [\G], ~\text{type $1$}\\ [\g] ~\text{ram.
rank-one
split}}}\vol([\g])u^{l_A([\g])}\bigg(q^{\mu([\g])}\frac{q^{\de([\g])}-1}{q-1}
+ \mu([\g])q^{\de([\g])}\bigg) \\
&=& q^{-1}\sum_{n> 0} R_{n,0}u^n + \sum_{\substack{\g \in [\G],
~\text{type $1$}\\ [\g] ~\text{ram. rank-one
split}}}\vol([\g])u^{l_A([\g])}(-q^{\mu([\g])-1} +
\mu([\g])q^{\de([\g])}).
\end{aligned}
\end{equation}
Finally, plug (\ref{unramodds}) and (\ref{ramodds}) into
(\ref{rankonepart}) to complete the proof.
\end{proof}

\subsection{Proof of Theorem C}

Combining Propositions \ref{alltypezerocycles} and \ref{traceBnm},
we obtain
\begin{eqnarray*}
 & &q \bigg( \sum_{n>0}
\Tr(B_{n,0})u^{n}\bigg) - (q-1)\bigg(\sum_{\substack{n,m \geq 0 \\
(n,m)\neq (0,0)}} \Tr(B_{n,m})u^{n+2m}\bigg)\bigg(\frac{1-q^2
u^3}{1-u^3}\bigg)\\
&=& \sum_{n>0} (P_{n,0,s} + P_{n,0,i} + Q_{n,0} + R_{n,0})u^{n} - (q-1) \sum_{\g \in [\G], ~\text{irregular, ~type $2$}} \vol([\g]) u^{l_A([\g])}\\
&+&
\sum_{\substack{\g \in [\G], ~\text{not type $1$}\\ [\g]
~\text{unram. rank-one split}}}2\vol([\g])u^{l_A([\g])} \\
&+& \sum_{\substack{\g \in [\G], ~\text{not type $1$}\\ [\g]
~\text{ram. rank-one split}}}\vol([\g])u^{l_A([\g])}(q^{\mu([\g])} -
\mu([\g])q^{\de([\g])+1})
\end{eqnarray*}
since all irregular elements have type $1$ or $2$.
As before, to a rank-one split $\g$, we associate $r_\g = \left(
\begin{smallmatrix} a &  &  \\ & e & dc \\ & d& e+db \end{smallmatrix}\right) $.
First assume $\g$ is unramified rank-one split. By Theorem
\ref{rankoneminlength}, $[\g]$ has type $(n, m) = (\ord_\pi$$ a,
\min(\ord_\pi$$ e, \ord_\pi$$ d))$, hence $[\g]$ is not of type $1$ if
and only if $a$ is a unit, which is equivalent to its inverse
$[\g^{-1}]$ having type $(m, 0)$. Note that $l_A([\g]) = 2m =
2l_A([\g^{-1}])$ by Theorem \ref{rankoneminlength}. Next assume that
$[\g]$ is ramified rank-one split. Since $\mu([\g]) = 1$ implies
$\de([\g]) = 0$, we have $q^{\mu([\g])} - \mu([\g])q^{\de([\g])+1} =
0$ in this case. Thus we need only consider the case $\mu([\g]) = 0$
so that $q^{\mu([\g])} - \mu([\g])q^{\de([\g])+1} = 1$. Then $[\g]$
is not of type $1$ if and only if $a$ is a unit, in which case it
has type $(0, \ord_{\pi}$$ e)$ if $\ord_{\pi}$$ e \le \ord_{\pi}$$ d$, and
type $(1, \ord_{\pi}$$ d)$ if $\ord_{\pi}$$ d < \ord_{\pi}$$ e$ by Theorem
\ref{rankoneminlength}. Further, we see that $[\g^{-1}]$ has type
$(\ord_{\pi}$$ e, 0)$ so that $l_A([\g]) = 2 l_A([\g^{-1}]) =
2\ord_{\pi}$$ e$ in the former case, and in the latter case,
$[\g^{-1}]$ has type $(\ord_{\pi}$$ d, 1)$, $[\g^{-2}]$ has type
$(2\ord_{\pi}$$ d + 1, 0)$ and $l_A([\g]) = 1 + 2\ord_{\pi}$$ d =
l_A([\g^{-2}])$. Further, we have $\vol([\g]) = \vol([\g^{-1}]) =
\vol([\g^{-2}])$ for $\g$ rank-one split. Consequently, we may replace
$\g$ by $\g^{-1}$ and rewrite
\begin{eqnarray*}
&~&\sum_{\substack{\g \in [\G], ~\text{not type $1$}\\ [\g]
~\text{unram. rank-one split}}}2\vol([\g])u^{l_A([\g])} +
\sum_{\substack{\g \in [\G], ~\text{not type $1$}\\ [\g] ~\text{ram.
rank-one split}}}\vol([\g])u^{l_A([\g])}(q^{\mu([\g])} -
\mu([\g])q^{\de([\g])+1})\\&=& \sum_{\g \in [\G], ~\text{ type $1$
unram. rank-one split}}2\vol([\g])u^{2l_A([\g])} + \sum_{\g \in [\G],
~\text{ type $1$ ram. rank-one split}}\vol([\g])u^{2l_A([\g])}
\\&+& \sum_{\g \in [\G], [\g] ~\text{of type} ~(m, 1),~\text{ram. ~rank-one ~split}}\vol([\g])u^{l_A([\g^2])}.
\end{eqnarray*}
Together with the term $- (q-1) \sum_{\g \in [\G], ~\text{irregular, ~type $2$}} \vol([\g]) u^{l_A([\g])}$,  it gives the difference of the logarithmic
derivatives of $Z_1(X_\G, u^2)$ and $Z_B(X_\G, -u)$ by Theorem
\ref{compareZ2andZ1}. Here we used the fact that the inverse of a type $2$ irregular element $\g$ is type $1$ irregular, and
$\vol([\g]) = \vol([\g^{-1}])$.

 Combined with Propositions \ref{lefthand1} and \ref{zetaassum}, this proves

\begin{proposition}\label{extraterms}
\begin{eqnarray*}\label{Pn0intrace}
& &u\frac{d}{du}\log \bigg(\frac{(1-u^3)^{\chi(X_\G)}}{ \det(I- A_1 u+q A_2 u^2 - q^3I
u^3)}\bigg)  \\
 &= &q \bigg( \sum_{n>0}
\Tr(B_{n,0})u^{n}\bigg) - (q-1)\bigg(\sum_{\substack{n,m \geq 0 \\
(n,m)\neq (0,0)}} \Tr(B_{n,m})u^{n+2m}\bigg)\bigg(\frac{1-q^2
u^3}{1-u^3}\bigg)\\
&=& u\frac{d}{du} \log Z_{1,1}(X_\G, u) + u\frac{d}{du} \log Z_{1,2}(X_\G,
u) - u\frac{d}{du} \log Z_{2,1}(X_\G, -u). \qquad \qquad
\end{eqnarray*}
\end{proposition}
\medskip

Consequently, we have
 $$\frac{(1-u^3)^{\chi(X_\G)}}{ \det(I- A_1 u+q A_2 u^2 - q^3I
u^3)} = c \frac{Z_{1,1}(X_\G, u)Z_{1,2}(X_\G, u)}{Z_{2,1}(X_\G, -u)} = c
\frac{\det(1 + L_Bu)}{\det(I - L_Eu) \det(I - (L_E)^t u^2)}$$ for some
constant $c$. Here the last equality comes from Propositions
\ref{rationalZ2} and \ref{Z1andLE}. Since both sides are formal
power series with constant term $1$, we find $c=1$. This concludes
the proof of Theorem C.

\end{document}